\documentclass[11pt]{article}
\usepackage{diagbox}
\usepackage{mathrsfs}
\usepackage[enableskew]{youngtab}
\usepackage{amsfonts,amsmath,dsfont,amssymb,amsthm,stmaryrd,bbm,bm,color,comment,graphicx,aligned-overset}
\usepackage{soul}
\usepackage[hidelinks]{hyperref}
\usepackage{multirow}
\usepackage[T1]{fontenc}

\makeatletter
\newtheorem*{rep@theorem}{\rep@title}
\newcommand{\newreptheorem}[2]{%
\newenvironment{rep#1}[1]{%
 \def\rep@title{#2 \ref{##1}}%
 \begin{rep@theorem}}%
 {\end{rep@theorem}}}
\makeatother

\newtheorem{thm}{Theorem}[section]
\newreptheorem{thm}{Theorem}

\newtheorem{prop}[thm]{Proposition}
\newreptheorem{prop}{Proposition}

\newtheorem{defn}[thm]{Definition}

\newtheorem{lem}[thm]{Lemma}
\newreptheorem{lem}{Lemma}

\newtheorem{coro}[thm]{Corollary}
\newreptheorem{coro}{Corollary}

\newtheorem{rmk}{Remark}

\newcommand{\Var}{\mathrm{Var}}

\newcommand{\R}{\mathbb{R}}
\newcommand{\N}{\mathbb{N}}
\newcommand{\Z}{\mathbb{Z}}

\newcommand{\E}{\mathbb{E}}

\def\P{{\mathbb P}}

\newcommand{\la}{\lambda}

\def\O{\mathcal{O}}
\DeclareMathOperator{\co}{d}

\begin{document}
\title{Asymptotic Normality and Concentration Inequalities of Statistics of Core Partitions with Bounded Perimeters}

\author{Jiange Li, Yetong Sha\footnote{Corresponding author.} ~and Huan Xiong\\ ~ \\
Institute for Advanced Study in Mathematics,
   Harbin Institute of Technology\\
   Heilongjiang 150001, P.R. China\\ ~ \\
   Email:    
jiange.li@hit.edu.cn, dellenudaubg@gmail.com, 
huan.xiong.math@gmail.com}
\date{ }
\maketitle

\begin{abstract}
    Core partitions have attracted much attention since Anderson's work (2002) on the number of $(s,t)$-core partitions for coprime $s,t$. Recently, there has been a growing interest in studying the limiting distributions of the sizes of random simultaneous core partitions. In this paper, we prove the asymptotic normality of certain statistics of uniform random core partitions with bounded perimeters in the Kolmogorov and Wasserstein $W_1$ distances, including the length and size of a random (strict) $n$-core partition, the length of the Durfee square and the size of a random self-conjugate $n$-core partition. Accordingly, we prove that these statistics are subgaussian. This contrasts with the asymptotic behavior of the size of a random $(s, t)$-core partition for coprime $s,t$ studied by Even-Zohar (2022), which converges in law to Watson's $U^2$ distribution. Our results show that the distribution of the size of a random strict $(n, dn+1)$-core partition is asymptotically normal when $d \ge 3$ is fixed and $n$ tends to infinity, which is an analog of Zaleski's conjecture (2017) and covers Koml\'{o}s, Sergel, and Tusn\'{a}dy's result (2020) as a special case. Our proof integrates a variety of combinatorial and probabilistic tools, including Stein's method based on Hoeffding decomposition, Hoeffding's combinatorial central limit theorem, the Efron-Stein inequalities on product spaces and slices, and asymptotics of P\'{o}lya frequency sequences. Furthermore, our approach is potentially applicable to the study of the asymptotic normality of functionals of random variables with certain global dependence structures that can be decomposed into appropriate mixture forms.
\end{abstract}


\tableofcontents

\section{Introduction}

\subsection{Background}

A \emph{partition} is a finite non-increasing sequence of positive integers $\lambda = (\lambda_1, \lambda_2, \cdots, \lambda_\ell)$ such that $\lambda_1 \ge \lambda_2 \ge \cdots \ge \lambda_\ell > 0$. Here, $\ell = \ell (\lambda)$ is called the \emph{length}, $\{\lambda_i\}_{i=1}^\ell$ are the \emph{parts} and $\left| \lambda \right| := \sum_{1 \leq i \leq \ell}^{} \lambda_i$ is the \emph{size} of $\lambda$. We say that $\lambda$ is \emph{strict} or with distinct parts if $\lambda_1 > \lambda_2 > \cdots > \lambda_\ell$. Each partition $\lambda$ can be visualized by its \emph{Young diagram}, which is an array of boxes arranged in left-justified rows with $\lambda_i$ boxes in the $i$-th row. The length of the \emph{Durfee square} of $\lambda$ is the largest integer $k$ such that $\lambda$ contains at least $k$ parts that are larger than or equal to~$k$ (see \cite{andrews2004integer}). The \emph{conjugate} of $\lambda$ is obtained by flipping the Young diagram of $\lambda$ along its main diagonal, i.e., turning the rows into columns. We call $\lambda$ a \emph{self-conjugate} partition if it is conjugate to itself. 

For each box $\square=(i,j)$ in the $i$-th row and the $j$-th column of the Young diagram of $\lambda$, its \emph{hook length} $h_\square=h_{ij}$ is defined to be the total number of boxes which are either directly to the right of or directly below the box together with the box itself. The maximal hook length $h_{11} = \lambda_1 + \ell - 1$ is called the \emph{perimeter} of $\lambda$. The \emph{$\beta$-set} $\beta (\lambda)$ is defined as the set of all hook lengths of the boxes in the first column of the Young diagram of $\lambda$. For example, the Young diagram and hook lengths of the partition $(6,3,2,1)$ are given in Fig. \ref{fig:1}. Thus $(6,3,2,1)$ has length $4$, size $12$, perimeter $9$, and the $\beta$-set $\beta (\lambda) = \{9, 5, 3, 1\}$.

The distributions of certain statistics of random partitions under different measures have been widely studied \cite{BaikDeiftJohansson1999, BaikRains2001, BorodinOkounkovOlshanski2000, DF16, DS17,Ivanov2006,IvanovOlshanski2002, Kessler1976,LoganShepp1977,Matsumoto,Mel10,NekrasovOkounkov2006, OkounkovRandomPartitions, okounkov2003correlation, KerovVershikLimitYD}. For example, let $\lambda$ be a uniform random partition of size $n$. Kessler and Livingston \cite{Kessler1976} showed that the average length of $\la$ satisfies
$$
\mathbb{E}[\ell(\la)] = \frac{\sqrt{6}}{2\pi} \sqrt{n} \log n + \O(\sqrt{n}).
$$
Let $\co(\la)$ be the number of distinct parts of $\la$. Goh and Schmutz \cite{goh1995number} showed that $\co(\la)$ satisfies the following asymptotic normality
$$
  \lim_{n\to\infty} \mathbb{P}\biggl(\co(\la) \leq (\sqrt{6}/\pi) \sqrt{n}+  \sqrt{  \sqrt{6}/2\pi - \sqrt{54}/\pi^3  } \cdot n^{1/4} x\biggr) = \Phi(x),
$$
where $\Phi(x)$ is the cumulative distribution function of the standard normal distribution.

A partition $\lambda$ is called an \emph{$s$-core partition} if the set of its hook lengths does not contain any multiples of $s$. Furthermore,  a partition is called an \emph{$(s_1,s_2,\cdots, s_m)$}-core partition if it is simultaneously an $s_1$-core, an $s_2$-core, $\cdots$, and an $s_m$-core partition. For example, we can see from Fig. \ref{fig:1} that $(6,3,2,1)$ is a  $(4,6,11)$-core partition. For $s \ge 2$, there are infinitely many $s$-core partitions. Anderson \cite{A02} showed that for coprime $s, t$, there are finite $(s, t)$-core partitions.

\begin{figure}[htbp] \label{6321}
    \begin{center}
    $\young(975321,531,31,1)$
    \end{center}
    \caption{The Young diagram and hook lengths of the partition
    $(6,3,2,1)$.} \label{fig:1}
\end{figure}

Core partitions arise naturally in the study of modular representation theory and algebraic combinatorics. For example, core partitions label the blocks of irreducible characters of symmetric groups (see \cite{OS07}). Hugh and Nathan \cite{HN18} connect simultaneous core partitions with rational combinatorics. Also, simultaneous core partitions are closely related to Motzkin paths and Dyck paths (see \cite{CH22,CHS20,  YYZ21}).

Some families of simultaneous core partitions, such as strict $(s, ds \pm 1)$-core partitions \cite{A12,NS16,S16,X18a,X18b,XZ21,Z17,Z19}, self-conjugate core partitions \cite{CHW16,FMS09,W16}, $(t, t + 1, \cdots, t + p)$-core partitions \cite{A12,AL15,X16,YZZ15} and $(s, ms - 1, ms + 1)$-core partitions \cite{NS16, SX23}, and their statistics have been paid much attention to in recent years. 

Their distributions have also been studied. The average size of $(s, t)$-core partitions has been studied in \cite{AHJ14,EZ15,J18,SZ15,W16}. Variances and higher moments have been computed in \cite{EZ15,TW17}. Notably, Even-Zohar \cite{E22} proved that for coprime $s, t$, the total size of a random $(s,t)$-core partition converges in law to Watson's $U^2$ distribution, proving Zeilberger's conjecture in \cite{EZ15,ZZ17}.

Furthermore, Zaleski \cite{Z17,Z19} conjectured that when $d$ is fixed and $n$ tends to infinity, the size of a uniform random $(n,dn-1)$-core partition \emph{with distinct parts} converges in law to a normal random variable. Koml\'{o}s, Sergel and Tusn\'{a}dy \cite{KST20} proved the $d = 1$ case of the conjecture with a Berry-Esseen bound. 
Xiong and Zang \cite{XZ21} computed the asymptotic formulas of moments in the general case and showed strong evidence of the validity of Zaleski's conjecture.

Motivated by the above results, in this paper, we study the distribution of the length and size of a random $n$-core partition, random strict $n$-core partition, and random self-conjugate $n$-core partition with bounded perimeters, and prove that these distributions are all asymptotically normal when $n$ tends to infinity, with explicit convergence rates in the Kolmogorov and Wasserstein $W_1$ distances. This behavior contrasts with the asymptotic distribution of the size of a random $(s, t)$-core partition for coprime $s,t$ studied by Even-Zohar \cite{E22}. Our work covers Koml\'{o}s, Sergel and Tusn\'{a}dy's result \cite{KST20} and shows that the size of a uniform random, strict $(n, dn+1)$-core partition is asymptotically normal when $n$ tends to infinity, which is an analog of Zaleski's conjecture \cite{Z19}. 

\subsection{Notations}
This subsection contains the notations that will be used throughout the paper. 

\begin{itemize}
    \item For functions $f, g: \N \to \R_+$, the notation $f \lesssim g$ means that there is an absolute constant $C$ such that $f(n) \leq C g(n) $ for all $n\in\N$.
\item For any positive integers $n$ and $k$, we set $[n] := \{1, 2, \cdots, n\}$ and denote by $\binom{[n]}{k}$ the family of all $k$-subsets of $[n]$.
    \item For an index set $I = \{x_1, \cdots, x_k\} \subset \N$ with $x_1 < \cdots < x_k$, we define the set $\tau(I)$ as 
    \begin{equation}\label{eq:tau-action}
    \tau (I) := \{x_1, x_2 + 1, \cdots, x_k + k - 1\}.
    \end{equation}
     \item For an index set $I \subset [n]$, we define the set $I^{(n)}$ as follows
    \begin{equation}\label{eq:I^n}
    I^{(n)} := \{(x_1, \cdots, x_n): x_i=0~\text{if}~i\notin I,~\text{and}~ x_i\in [d_n]~\text{if}~i\in I\}.
    \end{equation}
    The parameter $d_n$, used in the paper, is given in Section \ref{Sec1.3}.
    \item Let $X$ be a real-valued random variable. We denote by $\mathcal{L}(X)$ the distribution of $X$. We write $\mu_X:=\E [X]$ and $\sigma_X:= \sqrt{\text{Var} \left( X \right)}$ for the expectation and standard deviation, respectively.
Provided $\sigma_X > 0$, we define the normalized version of $X$ as 
\begin{equation}\label{eq:normalization}
\widehat{X}=\frac{X-\mu_X}{\sigma_X}.
\end{equation}
    \item We denote by $\mathcal{N}\left( 0, 1 \right)$ the standard normal distribution. The corresponding probability density function $\varphi (x) $ and cumulative distribution function $\Phi (x)$ are given by
    \begin{align*}
        \varphi (x) = \frac{1}{\sqrt{2 \pi}} e ^ {- x^2 / 2}, ~~~ \text{and} ~~~ \Phi (x) = \int_{- \infty}^{x} \varphi (t) dt.
    \end{align*}
    \item We define $\mathcal{B}_n := \{0, 1, \cdots, d_n\}^{n - 1}$ and 
    \begin{equation}\label{eq:SBn}
    \mathcal{SB}_n := \big\{(x_1, x_2, \cdots, x_{n - 1}) \in \mathcal{B}_n : x_i x_{i + 1} = 0\ \text{for}\ 1 \leq i \leq n - 2\big\}.
    \end{equation}
    One can check that $|\mathcal{SB}_n|=\sum_{0\leq k\leq \lfloor n/2\rfloor} {n-k \choose k}d_n^k$.
    \item We denote by $U$ a random variable taking values in $\{0, 1, \cdots, \lfloor n/2\rfloor\}$ with distribution 
    \begin{equation}\label{eq:rv-U}
    \mathbb{P}(U=k)=\frac{{n-k \choose k}d_n^k}{|\mathcal{SB}_n|}.
    \end{equation}
    \item For a general set $S$, we denote by $\mathcal{U} (S)$ the uniform distribution on $S$.
    \item We denote by $\{\gamma_i\}_{i=1}^\infty$ the moment sequence of the uniform distribution $\mathcal{U}(S)$, where 
    $S=\left\{- \frac{ d_n - 1 }{ 2 }, - \frac{ d_n - 3 }{ 2 }, \cdots, \frac{ d_n - 1 }{ 2 }\right\}$; that is, 
    \begin{equation}\label{eq:gamma_i}
    \gamma_i=\frac{1}{d_n}\sum_{x\in S}x^i.
    \end{equation}
    \item Let $\mu$ and $\nu$ be probability distributions on  $\R$. Let $F$ and $G$ be the corresponding cumulative distribution functions, i.e., $F(x)=\mu((-\infty, x])$ and $G(x)=\nu((-\infty, x])$. The Kolmogorov distance between $\mu$ and $\nu$ is defined as
    \begin{align}\label{eq:K-dist}
        d_K (\mu, \nu) :=\sup_{x\in\R} |F(x)-G(x)|, 
    \end{align}
    and the Wasserstein $W_1$ distance between $\mu$ and $\nu$ is defined as
    \begin{align}\label{eq:W-dist}
        d_W (\mu, \nu) = \int_{0}^{1} \left| F^{-1} (t) - G^{-1} (t) \right| dt = \int_{\R}^{} \left| F(x) - G(x) \right| dx.
    \end{align}
    For brevity, we write $d_{W/K}(\mu, \nu):=\max\{d_W(\mu, \nu),d_K(\mu, \nu)\}$.
\end{itemize}

\subsection{Main results} \label{Sec1.3}  




We denote by $L_{n, 1}$ and $S_{n, 1}$ the length and size of a uniform random $n$-core partition with perimeter at most $D_n$. Similarly, let $L_{n, 2}$ and $S_{n, 2}$ denote respectively the length and size of a uniform random strict $n$-core partition with perimeter at most $D_n$. Recall that the length of the Durfee square of a partition is the largest integer $k$ such that the partition contains at least $k$ parts that are larger than or equal to~$k$.
We write $L_{n, 3}$ and $S_{n, 3}$ for the length of the Durfee square and the size of a uniform random self-conjugate $n$-core partition with perimeter at most $E_n$. \emph{Unless otherwise specified, we assume that $D_n$ is divisible by $n$, $E_n$ is divisible by $2 n$ and set $d_n = D_n/n$ and $e_n =E_n/(2 n)$. For simplicity, we allow the partition to be empty.} In this paper, we prove the following asymptotic normality and concentration inequalities of $\{L_{n, i}\}_{i=1}^3$ and $\{S_{n, i}\}_{i=1}^3$.

\begin{thm} \label{thm:AsymptoticNormalityCorePartition}
For sufficiently large $n$, we have
\begin{align*}
    d_{W / K} \left( \mathcal{L} ( \widehat{L_{n, i}} ), \mathcal{N}\left( 0, 1 \right)  \right)
    & \lesssim \frac{ 1 }{ \sqrt{ n } }, ~~ i = 1, 2, 3; \\
    d_{W / K} \left( \mathcal{L} ( \widehat{S_{n, i}} ), \mathcal{N}\left( 0, 1 \right)  \right)
    & \lesssim \frac{ 1 }{ \sqrt{ n } }, ~~ i = 1, 3.
\end{align*}
When $d_n \ge 3$, we have
\begin{align*}
    d_{W / K} \left( \mathcal{L} ( \widehat{S_{n, 2}} ), \mathcal{N}\left( 0, 1 \right)  \right)
    & \lesssim \frac{ 1 }{ \sqrt{ n } }.
\end{align*}
\end{thm}

\begin{thm} \label{thm:ConcentrationCorePartition}

    There exists an absolute constant $C > 0$ such that the following results hold. For any $r > 0$ and $i = 1, 2$, we have
    \begin{gather}
        \mathbb{P} \left( \left| L_{n, i}  
        - \mathbb{E} \left[ L_{n, i}  \right] \right| \ge r \right) 
        \leq 2 \exp \left( - C \frac{r ^ 2}{n d_n^2} \right),\label{eq:concentilln} \\
        \mathbb{P} \left( \left| S_{n, i} - \mathbb{E} \left[ S_{n, i} \right] \right| \ge r  \right)
        \leq 2 \exp \left( - C \frac{ r ^ 2 }{ n ^ 3 d_n ^ 4 } \right). 
        \label{eq:concentilsn}
    \end{gather}For any $r > 0$, we have
    \begin{gather}
        \mathbb{P} \left( \left| L_{n, 3} - \mathbb{E} \left[ L_{n, 3} \right] \right| \ge r  \right)
        \leq 2 \exp \left( - C \frac{ r ^ 2 }{ n e_n ^ 2 } \right) ,\label{eq:concentilln2} \\
        \mathbb{P} \left( \left| S_{n, 3} - \mathbb{E} \left[ S_{n, 3} \right] \right| \ge r \right)
        \leq 2 \exp \left( - C \frac{ r ^ 2 }{ n ^ 3 e_n ^ 4 } \right) \label{eq:concentilsn2}.
    \end{gather}

\end{thm}

It is known that a partition is a strict $(n, dn + 1)$-core partition if and only if it is a strict $n$-core partition with perimeter at most $dn$ (e.g., see \cite{S16}). More precisely, on the one hand, it is clear that any strict $n$-core partition with perimeter at most $dn$ is a strict $(n, dn + 1)$-core partition. On the other hand, if $\lambda$ is a strict $(n, dn + 1)$-core partition with perimeter $h_{11} = h_{11} (\lambda) \ge dn + 1$, then $h_{11} - dn - 1, h_{11} - dn \in \beta (\lambda)$, which contradicts the fact that $\lambda$ is strict. Thus the perimeter of a strict $(n, dn + 1)$-core partition is at most $dn$. Hence, the following result is a special case of Theorem \ref{thm:AsymptoticNormalityCorePartition} and it generalizes Koml\'{o}s, Sergel and Tusn\'{a}dy's work \cite{KST20} on uniform random strict $(n, n+1)$-core partitions. This result is analogous to Zaleski's conjecture \cite{Z19} on random strict $(n, dn-1)$-core partitions.

\begin{coro}\label{cor:1}
   The size of a random strict $(n, dn+1)$-core partition is asymptotically normal when $d \ge 3$ is fixed and $n$ tends to infinity.
\end{coro}

The asymptotic normality and concentration inequalities of $L_{n, 1}, L_{n, 3}, S_{n, 1}, S_{n, 3}$ directly follow from the quantified version of the classical central limit theorem and the Efron-Stein inequality, respectively. The main task of this paper is to prove the asymptotic normality and concentration inequalities of $L_{n, 2}$ and $S_{n, 2}$. In fact, we prove analogs of Theorems \ref{thm:AsymptoticNormalityCorePartition} and \ref{thm:ConcentrationCorePartition} for more general functionals under certain technical conditions. The approach that we developed is probably applicable for establishing the asymptotic normality of functionals of random variables with certain global dependence structures that can be decomposed into appropriate mixture forms.

\begin{thm} \label{thm:normalgeneral}
Let $d_n \ge 2$. Suppose that the function $f: \R ^ {n - 1} \to \R$ can be written as follows
\begin{align*}
    f (x_1, \cdots, x_{n - 1})
    = \sum_{i = 1}^{n - 1} g_i (x_i)
    + a \sum_{i < j}^{} x_i x_j
    - a \binom{n - 1}{2} \frac{ (d_n + 1) ^ 2 }{ 4 },
\end{align*}
where the function $g_i: \R \to \R$ has root 
$- \frac{ d_n + 1 }{ 2 }$ and $a\in \R$ is a constant. Let $X$ be a random variable uniformly distributed on $S=\left\{ - \frac{ d_n - 1 }{ 2 }, - \frac{ d_n - 3 }{ 2 }, \cdots, \frac{ d_n - 1 }{ 2 } \right\}$. We further assume that the following conditions hold.
\begin{enumerate}
    \item \label{cond:NearIndependence} 
    It holds that 
    \begin{align*}
        a ^ 2 n ^ 2 d_n ^ 4
        \lesssim \inf_{i \in [n - 1]} \text{Var} ( g_i(X)).
    \end{align*}
    \item \label{cond:ArithmeticProgression} For each $x \in S$, the sequence $\{g_i(x)\}_{i=1}^{n-1}$ is an arithmetic progression.  
    \item \label{cond:Boundedness} There exists a universal constant $C > 0$ such that for all $i \in [n - 1]$,
    \begin{align*}
        \sup_{x\in S} g_i ^ 2 (x) < C \inf_{j \in [n - 1]} \text{Var} \left( g_{j}(X) \right).
    \end{align*}
    \item \label{cond:Nondegenerate} If $a \neq 0$, then there exists a universal constant $C \in (0, 1)$ such that for all $i \in [n - 1]$,
    \begin{align*}
        \text{Cov} \left( X, g_i (X) \right) ^ 2
        \leq C \text{Var} \left( X \right) \text{Var} \left( g_i (X) \right).
    \end{align*}
\end{enumerate}
Consider the random variable $W = f ( X_1, \cdots, X_{n - 1} ) $, where $( X_1, \cdots, X_{n - 1} )$ has the distribution such that $\left( X_1 + \frac{ d_n + 1 }{ 2 }, \cdots, X_{n - 1} + \frac{ d_n + 1 }{ 2 } \right) $ is uniformly distributed on $\mathcal{SB}_n$ (which is defined in \eqref{eq:SBn}). 
For sufficiently large $n$, we have
\begin{align*}
    d_{W / K} \left( \mathcal{L} ( \widehat{W} ),~ \mathcal{N}\left( 0, 1 \right) \right) 
    \lesssim \frac{ 1 }{ \sqrt{ n } }.
\end{align*}
\end{thm}

\begin{rmk}
    Note that the convergence rate obtained is independent of $d_n$. 
\end{rmk}


\begin{rmk}
Our proof combines Stein's method based on Hoeffding decomposition, the Efron-Stein inequality, Hoeffding's combinatorial central limit theorem, and asymptotics of P\'{o}lya frequency sequences. The method employed in \cite{KST20} for $d_n = 1$ cannot be easily extended to the general case $d_n \ge 2$, since $W$ follows a mixture distribution and the asymptotic normality of each mixing component is a problem to handle, which does not appear in the case $d_n = 1$. We use Stein's method based on Hoeffding decomposition to deal with this problem. Also, we avoid the lengthy computations in \cite{KST20} by applying the Efron-Stein inequality.
\end{rmk}

\begin{rmk} \label{rmk:AsymptoticNormalityConditionRemark}
Let us provide explanations for the conditions required in Theorem \ref{thm:normalgeneral}.
\begin{enumerate}
    \item The assumption $d_n \ge 2$ is for simplicity and is compatible with the hidden assumption that $\inf_{i \in [n - 1]} \text{Var} \left( g_i (X)\right) > 0$. For given $W$, the case $d_n = 1$ can be analyzed in a manner similar to that of Theorem \ref{thm:normalgeneral}. The only exception is that some steps may be omitted and the inequalities that involve $\inf_{i \in [n - 1]} \text{Var} \left( g_i(X) \right)$ need to be replaced by ad-hoc computations.
    \item Condition \ref{cond:NearIndependence} is to bound the influence of higher order terms in $f$.
    \item Condition \ref{cond:ArithmeticProgression} is needed to apply Hoeffding's combinatorial central limit theorem.
    \item Condition \ref{cond:Boundedness} implies that functions $g_i$ are balanced in some sense.
    \item Condition \ref{cond:Nondegenerate} implies that functions $g_i$ are nonlinear if $a \neq 0$.
\end{enumerate}
\end{rmk}

\begin{rmk}
    Let $\lambda$ be a uniform random self-conjugate $n$-core partition with perimeter at most $E_n$. Suppose that $E_n$ is divisible by $2n$. Define
\begin{align*}
    M_{n, 3} ^ {(k)} := \sum_{x \in MD (\lambda)}^{} x ^ k
    = \sum_{1 \leq i < \frac{ n + 1 }{ 2 }}^{} \left( \sum_{x \in MD (\lambda)_{2 i - 1}}^{} x ^ k + \sum_{x \in MD (\lambda)_{2 (n - i) + 1}}^{} x ^ k \right). 
\end{align*}
Here, $MD (\lambda)$ is the set of hook lengths of boxes on the main diagonal of $\lambda$. Then $M_{n, 3} ^ {(0)} = L_{n, 3}$ and $M_{n, 3} ^ {(2)} = S_{n, 3}$. By the same argument in Theorem \ref{thm:normalgeneral}, one can prove that $M_{n, 3} ^ {(k)}$ is also asymptotically normal. One might ask if similar statistics of random strict $n$-core partitions with perimeters at most $D_n$ are asymptotically normal, with an explicit convergence rate, where $D_n$ is divisible by $n$. The main obstacle is that Hoeffding's combinatorial central limit theorem cannot be applied at higher orders, i.e., Condition \ref{cond:ArithmeticProgression} in Theorem \ref{thm:normalgeneral} fails.    
\end{rmk}




We also derive the following general results on concentration inequalities.

\begin{thm} \label{thm:concengeneral}
Suppose that the function $f: \{0, 1, \cdots, d_n\} ^ {n - 1} \to \R$ satisfies the following conditions:
\begin{enumerate}
    \item \label{cond:BoundDiff} There exists an absolute constant $C_1$ such that for any $i \in [n - 1]$ and $x_1, \cdots, x_i, x_i', \\ x_{i + 1}, \cdots, x_{n - 1} \in \{0, 1, \cdots, d_n\}$,
    \begin{align*}
        \left| f (x_1, \cdots, x_i, \cdots x_{n - 1}) - f (x_1, \cdots, x_i', \cdots, x_{n - 1}) \right| \leq C_1.
    \end{align*}
    \item \label{cond:QuasiAP} There exists an absolute constant $C_2$ such that for all $0\leq k\leq \lfloor n/2\rfloor$ and $I, J \in \binom{[n - k]}{k}$ such that $\left| I \Delta J \right|  = 2$, it holds that
    \begin{align*}
        \left| \big( \E[f(X_I)] - \E [f(X_{\tau(I)})] \big)
        - \big( \E[f(X_J)] - \E [f(X_{\tau(J)})] \big) \right| 
        \leq C_2,
    \end{align*}
    where $X_I, X_{\tau(I)}, X_J, X_{\tau(J)}$ are uniformly distributed on $I^{(n)}, \tau(I)^{(n)}, J^{(n)}, \tau(J)^{(n)}$, respectively. (The definitions of $\tau(I)$ and $I^{(n)}$ are given in \eqref{eq:tau-action} and \eqref{eq:I^n}, respectively).
\end{enumerate}
Suppose that $X$ is selected uniformly at random  from $\mathcal{SB}_n$. Then there exists an absolute constant $K > 0$ such that for $r > 0$,
\begin{align} \label{eq:concensbn}
    \mathbb{P} \left( \left| f (X) - \E [ f (X) ] \right| \ge r \right)
    \leq 2 \exp \left( - \frac{ r ^ 2 }{ nK(C_1 + C_2) ^ 2  } \right).
\end{align}
\end{thm}

\begin{rmk}
Condition \ref{cond:BoundDiff} is the bounded difference condition, and Condition \ref{cond:QuasiAP} is needed to apply Bobkov's concentration inequality on slices, i.e., Proposition \ref{prop:ConcentrationSlice}.
\end{rmk}

\subsection{Outline}

The rest of the paper is organized as follows. In the next section, we will give an overview of the proof of Theorem \ref{thm:normalgeneral}, characterize $\{L_{n, i}\}_{i=1}^3$ and $\{S_{n, i}\}_{i=1}^3$, and list the techniques that will be used to prove the main results, including Stein's method based on Hoeffding decomposition, Hoeffding's combinatorial central limit theorem, the Efron-Stein inequalities, asymptotics of P\'{o}lya frequency sequences, comparisons between normal mixtures and concentration inequalities of mixture distributions. Proofs of the main theorems will be given in Section $3$.

\section{Preliminaries}

\subsection{Proof outline of Theorem \ref{thm:normalgeneral}} \label{sec:thm:normalgeneral}


First recall that $\mathcal{SB}_n = \{(x_1, \cdots, x_{n - 1}) \in \{0, 1, \cdots, d_n\}^{n - 1} : x_i x_{i + 1} = 0\ \text{for}\ 1 \leq i \leq n - 2\}$ and $W = f (X_1, \cdots, X_{n - 1})$, where $\left( X_1 + \frac{d_n + 1}{2}, \cdots, X_{n - 1} + \frac{d_n + 1}{2} \right)  \sim \mathcal{U} \left( \mathcal{SB}_n \right) $. One of the main difficulties of showing the asymptotic normality of $W$ is that most known tools for bounding $d_{W / K} ( \mathcal{L} ( \widehat{W} ), \mathcal{N}( 0, 1 ) )$ assume that $X_1, \cdots, X_{n - 1}$ are at most locally dependent. In our case, however, $X_1, \cdots, X_{n-1}$ have a weak global dependence structure. 
We note that the distribution of $W$ can be represented as the mixture of asymptotically normal mixing components. We further exploit this mixing structure and establish the asymptotic normality of $W$ according to the following informal principle.


Let $X$ be the mixture of $\{X_{\theta}\}_{\theta \in \Theta}$ with $\eta$ being the \emph{mixing law}, i.e., the parameter $\theta$ is distributed according to $\eta$. Then $X$ is asymptotically normal, provided that the following conditions hold: (1) For each $\theta\in \Theta$, the random variable $X_{\theta}$ is asymptotically normal; (2) The random mean $\mu_{X_{\theta}}$ is asymptotically normal or the variance $\text{Var} \left( \mu_{X_{\theta}} \right)$ is small for $\theta \sim \eta$; (3) The variance $\sigma_{X_{\theta}}^2$ concentrates around its mean $\E_{\theta\sim\eta} [\sigma_{X_{\theta}}^2]$ in a proper sense.
This idea is based on the fact that compounding a normal distribution with the mean distributed according to another normal distribution and variance fixed yields again a normal distribution. More precisely, we have for $a, b, x \in \R$ that
\begin{align*}
    \int_{\R}^{} \Phi (ax + bt) \varphi (t) dt = \Phi \left( \frac{ax}{\sqrt{ 1 + b^2 }} \right). 
\end{align*}
More details can be found in \cite{KST20} (Lemma 16).  

We will show in Section \ref{subsection:proofnormalgeneral} that $W$ can be written as the mixture of $W ^ {\tau (J)}$ with weight $d_n ^ {|J|}/|\mathcal{SB}_n|$, where $J \in \cup_{0 \leq k \leq \left\lfloor n / 2 \right\rfloor }^{} \binom{[n - k]}{k}$. Here, $W^{J} = f ( X ^ J ) $ and $X ^ J = (X_1, \cdots, X_{n - 1})$ satisfies that $X_j = - \frac{ d_n + 1 }{ 2 }$ for $j \not \in J$, and $X_j \sim \mathcal{U} \left( \left\{ - \frac{ d_n - 1 }{ 2 }, - \frac{ d_n - 3 }{ 2 }, \cdots, \frac{ d_n - 1 }{ 2 } \right\}  \right)$ for $j \in J$. To prove the asymptotic normality of $W$, we will construct random variables $\{W ^ {(i)}\}_{i=1}^4$ with the same mean and variance as $W$ to interpolate between $W ^ {(0)} = W$ and $W ^ {(5)} \sim \mathcal{N}\left( \mu_W, \sigma_W ^ 2 \right)$. We will establish upper bounds of $d_{W / K} \big( \mathcal{L} ( \widehat{W ^ {(i)}} ), \mathcal{L} ( \widehat{W ^ {(i + 1)}} ) \big)$ for $i=0, 1, \cdots, 4$, via various tools. In the context where condition (1) holds, we will apply Stein's method based on Hoeffding decomposition (i.e., Proposition \ref{prop:hoePrivault}); in the context where condition (2) holds, Hoeffding's combinatorial central limit theorem will be used (i.e., Lemma \ref{lem:hoeffccltexample}); in the context where condition (3) holds, we will employ Efron-Stein inequalities on product spaces and slices (i.e., Propositions \ref{prop:ConcentrationSlice} and \ref{prop:BoundDiffConcentra}); estimates of distances between normal mixtures are also used (i.e., Lemmas \ref{lem:normaldistancekolwass} and \ref{lem:mixdis}).

Now we give detailed constructions of the random variables $\{W ^ {(i)}\}_{i=1}^4$ (more precisely, their distributions).

\textbf { Step 1.} To approximate $\mathcal{L}(W)$, we replace each mixing components $\mathcal{L}(W^{\tau(J)})$ by the normal distribution $\mathcal{N}\big( \mu_{\tau (J)}, \sigma_{\tau (J)} ^ 2\big)$, while keeping the mixing law unchanged. Here, we have $\mu_{\tau (J)} = \mu_{W ^ {\tau (J)}}$ and $\sigma_{\tau (J)} = \sigma_{W ^ {\tau (J)}}$. This yields the distribution $\mathcal{L}({W^{(1)}})$. It is easy to see that $\mu_{W^{(1)}} = \mu_{W}$ and $\sigma_{W^{(1)}} = \sigma_{W}$. Then, by triangle inequality, we have
    \begin{align*}
        d_{W / K} \left( \mathcal{L} ( \widehat{W} ), \mathcal{L} ( \widehat{W ^ {(1)}} )  \right) 
        \leq \E_{J\sim\mathcal{U}(S)} \left[ d_{W / K} \left( \mathcal{L} \left( \frac{W ^ {\tau (J)} - \mu_{W}}{\sigma_{W}} \right), \mathcal{N}\left( \frac{\mu_{\tau (J)} - \mu_{W}}{\sigma_{W}}, \frac{\sigma_{\tau (J)}^2}{\sigma_{W}^2} \right) \right)   \right],
    \end{align*}
    where $S=\cup_{0 \leq k \leq \left\lfloor n / 2 \right\rfloor }^{} \binom{[n - k]}{k}$. We will derive Hoeffding's decomposition of $W ^ {\tau (J)}$ and bound the right-hand side of the inequality via Stein's method (i.e., Proposition \ref{prop:hoePrivault}).

         
        \textbf { Step 2.} For brevity, we write $S_k={[n-k] \choose k}$ for $0\leq k\leq \lfloor n/2\rfloor$. To approximate $\mathcal{L}(W^{(1)})$, we replace each mixing component $\mathcal{N}\big( \mu_{\tau (J)}, \sigma_{\tau (J)} ^ 2 \big)$ by $\mathcal{N}\big( \mu_{\tau (J)}, \E_{J'\sim \mathcal{U}(S_{|J|})} \big[ \sigma_{\tau (J')}^2 \big]\big)$, while keeping the mixing law unchanged. This gives the distribution  $\mathcal{L}(W^{(2)})$. In other words, we replace the variance $\sigma_{\tau (J)} ^ 2$ by the average variance of $\sigma_{\tau (J')} ^ 2$ over $J'$ such that $|J'|=|J|$. One can check that $\sigma_{W^{(2)}} = \sigma_{W^{(1)}}$. (It relies on the fact that for each $k$ the random variables $W ^ {\tau (J)}$ with $J \in S_k$ have the same weight). By Lemma \ref{lem:normaldistancekolwass}, we have 
    \begin{align*} 
        d_K \left( \mathcal{L} ( \widehat{W^{(1)}} ), \mathcal{L} ( \widehat{W^{(2)}} ) \right) 
        \leq \E_{k \sim \mathcal{L} (U)} 
        \left[ \frac{\sqrt{ \text{Var}_{J \sim \mathcal{U}(S_k)} \big( \sigma_{\tau (J)}^2\big) }}
        {\E_{J \sim \mathcal{U}(S_k)} \big[ \sigma_{\tau (J)}^2 \big] } \right],
    \end{align*}
    where the distribution of $U$ is given in \eqref{eq:rv-U}. One can then apply Efron-Stein's inequality on slices (i.e., Proposition \ref{prop:ConcentrationSlice}) to give an upper bound of $\text{Var}_{I \sim \mathcal{U}(S_k)}( \sigma_{\tau (I)}^2)$. The Wasserstein $W_1$ distance between $\widehat{W^{(1)}}$ and $\widehat{W^{(2)}}$ can be bounded in a similar manner.
    
 
 \textbf { Step 3.} We denote by $\mu_k$ and $\sigma_k ^ 2$ the mean and variance, respectively, of the uniform mixture of $W ^ {\tau (J)}$ with $J \sim \mathcal{U}(S_k)$. We define $\mathcal{L} ( W^{(3)} )$ as the mixture of $\mathcal{N}\left( \mu_{k}, \sigma_{k}^2 \right)$ with $k\sim\mathcal{L}(U)$.  One can check that $\mu_{W^{(3)}} = \mu_{W^{(2)}}$  and $\sigma_{W^{(3)}} = \sigma_{W^{(2)}}$. By Lemma \ref{lem:mixdis}, we have 
\begin{equation*}
d_{K} \left( \mathcal{L} ( \widehat{W^{(2)}} ), \mathcal{L} ( \widehat{W^{(3)}} ) \right) \lesssim  \E_{k\sim \mathcal{L}(U)} 
    \left[ \frac{ \sigma_{\mu_{\tau (J)}} \cdot d_W \left( \mathcal{L} \left( \widehat{\mu_{\tau (J)}} \right), \mathcal{N}\left( 0, 1 \right)  \right)}{ \sqrt{ \E_{J'\sim \mathcal{U}\left({[n-k] \choose k}\right)} [ \sigma_{\tau (J')} ^ 2 ] }  }  \ \middle| \ 0 < k < \frac{ n }{ 2 } \right]  .
 \end{equation*}
Then we use Hoeffding's combinatorial central limit theorem (i.e., Lemma \ref{lem:hoeffccltexample}) to bound the right-hand side of the inequality above. The Wasserstein $W_1$ distance between $\widehat{W^{(2)}}$ and $\widehat{W^{(3)}}$ can be bounded in a similar manner.

   
  \textbf { Step 4.} To approximate $\mathcal{L} \left( W^{(3)} \right)$, we define $\mathcal{L} \left( W^{(4)} \right)$ via replacing $\mathcal{N}\left( \mu_{k}, \sigma_{k}^2 \right)$ by $\mathcal{N}\left( \mu_{k}, \E [ \sigma_{U}^2 ] \right)$ where the distribution of $U$ is given in \eqref{eq:rv-U}, while keeping the mixing law unchanged. Then $\mu_{W^{(4)}} = \mu_{W^{(3)}}$  and $\sigma_{W^{(4)}} = \sigma_{W^{(3)}}$. By Lemma \ref{lem:normaldistancekolwass}, we have
    \begin{align*}
        d_K \left( \mathcal{L} ( \widehat{W^{(3)}} ), \mathcal{L} ( \widehat{W^{(4)}} ) \right) 
        \leq \frac{ \sqrt{ \text{Var} \left( \sigma_U ^ 2 \right) } }{ \E \left[ \sigma_{U}^2 \right] }.
    \end{align*}
    Then we use Efron-Stein inequality on product spaces (i.e., Proposition \ref{prop:BoundDiffConcentra}) to bound the variance $\text{Var}( \sigma_U ^ 2 ) $. The estimate of the Wasserstein $W_1$ distance can be obtained in a similar manner.
    
    
\textbf { Step 5.} By Lemma \ref{lem:mixdis}, we have $$
d_{W / K} \left( \mathcal{L} ( \widehat{W^{(4)}} ), \mathcal{N}\left( 0, 1 \right)  \right)\lesssim \frac{ \sigma_{\mu_U} }{ \sqrt{ \E [ \sigma_U ^ 2 ]  } } d_W \left( \mathcal{L} \left( \widehat{\mu_{U}} \right), \mathcal{N}\left( 0, 1 \right)  \right) .
$$
Then we use Stein's method based on Hoeffding decomposition (i.e., Proposition \ref{prop:hoePrivault}) to prove the asymptotic normality of $\mu_{k}$.

\begin{rmk}
In Steps 1 and 5, we use Stein's method based on Hoeffding decomposition to prove the asymptotic normality. There are other ways to apply Stein's method, for example, the exchangeable pairs method. However, other methods have to bound the Kolmogorov distance and Wasserstein $W_1$ distance separately. So we adopt the method from \cite{PS22}.
\end{rmk}

\subsection{Statistics of core partitions and $\beta$-sets}\label{subsec:beta set}

Recall that $L_{n, 1}$ and $S_{n, 1}$ denote respectively the length and size of a uniform random $n$-core partition with perimeter at most $D_n$; $L_{n, 2}$ and $S_{n, 2}$ denote respectively the length and size of a uniform random strict $n$-core partition with perimeter at most $D_n$. Here, $(D_n)_{n \ge 1}$ is a sequence of positive integers divisible by $n$ and we define $d_n := D_n / n$. In this subsection, we give characterizations of $L_{n, 1}, L_{n, 2}, S_{n, 1}, S_{n, 2}$. 

For a partition $\lambda = (\lambda_1, \lambda_2, \cdots, \lambda_\ell)$, its $\beta$-set $\beta(\lambda)$ is defined as the set of \emph{first-column hook lengths} in the Young diagram of $\lambda$, i.e., $\beta(\lambda) = \{h_{i1} : 1 \leq i \leq \ell\}$, where $h_{i1} = \lambda_i +  \ell - i$ and the maximal element $h_{11} = \lambda_1 + \ell - 1$ is called the \emph{perimeter} of $\lambda$. Note that a partition is uniquely determined by its $\beta$-set. 

Let $\mathcal{C}_n$ denote the set of $n$-core partitions with perimeters at most $D_n$ and let $ \mathcal{CD}_n$ denote the set of strict $n$-core partitions with perimeters at most $D_n$. Let $\mathcal{B}_n := \{0, 1, \cdots, d_n\}^{n - 1}$ and recall that $\mathcal{SB}_n := \{(x_1, x_2, \cdots, x_{n - 1}) \in \mathcal{B}_n : x_i x_{i + 1} = 0\ \text{for}\ 1 \leq i \leq n - 2\}$.

It is known that a partition $\lambda$ is an $n$-core partition if and only if for any $h \in \beta(\lambda)$ such that $h \ge n$, we have $h - n \in \beta(\lambda)$ (e.g., see \cite{A02, B71, OS07, X18a}). For $i=1, 2, \cdots, n-1$, we define 
$\beta(\lambda)_i := \beta(\lambda) \cap \left( n \Z + i \right) = \beta(\lambda) \cap \{nk+i : k \in \Z\}$.
It is clear that $\beta(\lambda) = \cup_{i=1}^{n- 1}\beta(\lambda)_i$. Write $x_i (\lambda) = \left| \beta(\lambda)_i \right|$.  Then we have $\beta(\lambda)_i = \{i, i + n, \cdots, i + (x_i (\lambda) - 1)n\}$ provided that $x_i(\lambda) > 0$. An $n$-core partition $\lambda$ has perimeter at most $D_n$ if and only if $x_i (\lambda) \leq d_n$ for all $1 \leq i \leq n-1$. Also, an $n$-core partition $\lambda$ is strict if and only if there don't exist $h_1, h_2 \in \beta(\lambda)$ such that $ \left| h_1 - h_2 \right| = 1$, and hence $x_i(\lambda) x_{i + 1}(\lambda) = 0$ for all $1 \leq i \leq n - 2$. Therefore, the map 
\begin{equation}\label{eq:bijection}
\eta: \lambda \to (x_1(\lambda), x_2(\lambda), \cdots, x_{n - 1}(\lambda))
\end{equation}
builds a bijection from $\mathcal{C}_n$ to $\mathcal{B}_n$, as well as a bijection from $\mathcal{CD}_n$ to $\mathcal{SB}_n$. We refer to \cite{X18a} for more details.

For brevity, we write $x_i$ for $x_i (\lambda)$. Then one can compute from its $\beta$-set $\beta(\lambda)$ the length $\ell (\lambda)$ and size $|\lambda|$ of $\lambda$ as
\begin{align} 
    \ell (\lambda) & = \left| \beta(\lambda) \right| = \sum_{i = 1}^{n - 1} x_i, \label{eq:length}\\
    \left| \lambda \right| 
    & = \sum_{h \in \beta(\lambda)}^{} h - \frac{\left| \beta(\lambda) \right| \left( \left| \beta(\lambda) \right| - 1 \right)}{2} \notag\\
    & = \sum_{i = 1}^{n - 1} \left( i x_i + \frac{x_i (x_i - 1)}{2}n \right) - \frac{\sum_{i = 1}^{n - 1} x_i \left( \sum_{i = 1}^{n - 1} x_i - 1 \right) }{2} \notag\\
    & = \sum_{i = 1}^{n - 1} \left( \frac{n - 1}{2} x_i^2 + \left( i - \frac{n - 1}{2} \right) x_i  \right) - \sum_{1 \leq i < j \leq n - 1}^{} x_i x_j. \label{eq:size}
\end{align}

Owing to the bijection \eqref{eq:bijection} from $\mathcal{C}_n$ to $\mathcal{B}_n$, and from $\mathcal{CD}_n$ to $\mathcal{SB}_n$, as well as
representations of the length and size in \eqref{eq:length} and \eqref{eq:size}, respectively, we obtain the following characterizations of $L_{n, 1}$, $L_{n, 2}$, $S_{n, 1}$, $S_{n, 2}$.

\begin{lem} \label{lem:ls}
    Let $X_1, X_2, \cdots, X_{n - 1}$ be independent and uniformly distributed on $\{0, 1, \cdots, d_n\}$. Then we have
    \begin{align*}
        L_{n, 1} & = \sum_{i = 1}^{n - 1} X_i, \\
        S_{n, 1} & = \sum_{i = 1}^{n - 1} \left( \frac{n - 1}{2} X_i^2 + \left( i - \frac{n - 1}{2} \right) X_i  \right) - \sum_{1 \leq i < j \leq n - 1}^{} X_i X_j.
    \end{align*}
    Let $(X_1, X_2, \cdots, X_{n - 1})$ be uniformly distributed on $\mathcal{SB}_n$. Then we have
    \begin{align*}
        L_{n, 2} & = \sum_{i = 1}^{n - 1} X_i, \\
        S_{n, 2} & = \sum_{i = 1}^{n - 1} \left( \frac{n - 1}{2} X_i^2 + \left( i - \frac{n - 1}{2} \right) X_i  \right) - \sum_{1 \leq i < j \leq n - 1}^{} X_i X_j.
    \end{align*}
\end{lem}

\subsection{Statistics of self-conjugate core partitions and main diagonals}\label{subsec:self-conjugate}

Let $\lambda$ be a self-conjugate partition and let $MD(\lambda)$ be the set of hook lengths of boxes on the main diagonal of $\lambda$.
It is easy to see that a self-conjugate partition $\lambda$ is uniquely determined by its main diagonal hooks. Also, the set $MD(\lambda)$ consists of odd positive numbers. Ford, Mai and Sze \cite{FMS09} gave the following characterization of $MD(\lambda)$ for a self-conjugate $n$-core partition $\lambda$.
\begin{prop}[\cite{FMS09}, Proposition 3] \label{prop:MD} 
A self-conjugate partition $\lambda$ is an $n$-core if and only if the following conditions hold: \newline
(1) $MD (\lambda)$ is a set of distinct odd integers; \newline
(2) if  $h \in MD(\lambda)$ and $h > 2n$, then $h-2n$ is also in $MD(\lambda)$;  \newline 
(3) if $h,s \in MD(\lambda)$, then $h+s \neq 0 \pmod{2n}$.
\end{prop}

Assume that $\lambda$ is a self-conjugate $n$-core partition with perimeter at most $E_n$. Similar to the discussions of the $\beta$-set of partitions in Section \ref{subsec:beta set}, we define $MD (\lambda)_i := MD (\lambda) \cap (2 n \Z + i)$ for $0 \leq i \leq 2 n - 1$ and $x_i = \left| MD (\lambda)_i \right|$. By Proposition \ref{prop:MD}, the following holds.
\begin{enumerate}
    \item[(1)] For even $i$, we have $x_i = 0$;
    \item[(2)] For $0 \leq i \leq n$, we have $x_i x_{2 n - i} = 0$.
    \item[(3)] For $0 \leq i \leq 2 n - 1$, we have $MD (\lambda)_i = \{i, 2 n + i, \cdots, 2 (x_i - 1)n + i\}$, provided $x_i > 0$.
\end{enumerate}
So $x_0 = x_n = 0$. One can check that
\begin{align*}
    \ell (\lambda) &= \sum_{i = 1}^{n} x_{2 i - 1},\\
    \left| \lambda \right| 
    & = \sum_{h \in MD (\lambda)}^{} x 
    = \sum_{1 \leq i \leq n, i \neq \frac{ n + 1 }{ 2 } }^{} \sum_{h \in MD (\lambda)_{2 i - 1}}^{} h \\
    & = \sum_{1 \leq i \leq n, i \neq \frac{ n + 1 }{ 2 } }^{} \left( (2 i - 1) x_{2 i - 1} + n x_{2 i - 1} (x_{2 i - 1} - 1) \right) \\
    & = \sum_{1 \leq i \leq n, i \neq \frac{ n + 1 }{ 2 } }^{} \left( n x_{2 i - 1} ^ 2 + (2 i - n - 1) x_{2 i - 1} \right). 
\end{align*}
Recall that $L_{n, 3}$ and $S_{n, 3}$ denote respectively the length of the Durfee square and size of a uniform random self-conjugate $n$-core partition with perimeter at most $E_n$. \emph{Assume that $E_n$ is divisible by $2n$} and $e_n =  E_n / 2 n $. Let $\mathcal{SC}_n$ denote the set of self-conjugate $n$-core partitions with perimeters at most $E_n$ and $\mathcal{MD}_n := \{(x_1, \cdots, x_n) \in \{0, 1, \cdots, e_n\} ^ n: x_i x_{n + 1 - i} = 0\ \text{for}\ 1 \leq i \leq n \}$. The map $\zeta: \lambda \to (x_1, x_2, \cdots, x_n)$ gives a bijection from $\mathcal{SC}_n$ to $\mathcal{MD}_n$. We have the following characterization of $L_{n, 3}$ and $S_{n, 3}$.

\begin{lem} \label{lem:SelfConjugateDistribution}
Let $(X_1, \cdots, X_n)$ be uniformly distributed on $\mathcal{MD}_n$. Then we have
\begin{align*}
    L_{n, 3} & = \sum_{1 \leq i \leq n, i \neq \frac{ n + 1 }{ 2 }}^{} X_i, \\
    S_{n, 3} & = \sum_{1 \leq i \leq n, i \neq \frac{ n + 1 }{ 2 } }^{} \left( n X_i ^ 2 + (2 i - n - 1) X_i \right).
\end{align*}
\end{lem}

\subsection{Stein's method based on Hoeffding decomposition}


Given a family of independent random variables $\{X_i\}_{i=1}^n$, the family
$\{\mathcal F_J\}_{J\subset [n]}$ of $\sigma$-algebras are defined as $\mathcal F_J:=\sigma (X_j\ : \ j\in J)$.

\begin{defn} \label{def:HoeffdingDecomposition}
A centered $\mathcal F_{[n]}$-measurable random variable $W_n$ admits a Hoeffding decomposition if it can be written as 
\begin{align} \label{eq:Hd}
    W_n=\sum_{J\subset [n]}W_J,
\end{align}where $\{W_J\}_{J\subset [n]}$ is a family of random variables such that $W_J$ is ${\cal F}_J$-measurable, and 
\begin{align*}
    \E[W_J \mid \mathcal{F}_K]=0, \qquad J\not\subseteq K\subset [n].
\end{align*}
\end{defn}

We need the following result from \cite{PS22}.

\begin{prop}[\cite{PS22}, Theorem 4.1] \label{prop:hoePrivault}
Let $1 \leq d \leq n$. Let $W_n$ be a centered random variable with unit variance and finite fourth moment. Suppose that $W_n$ admits the Hoeffding decomposition \eqref{eq:Hd} with $|J|\leq d$. Then there exists a constant $C_d>0$ depending only on $d$ such that
\begin{align*}
d_{W / K} (W_n, \mathcal{N}\left( 0, 1\right) ) 
&\leq C_d  \left\{ \sum_{0 \leq l < i \leq d} \sum_{| J | = i - l} \E \left[ \left( \sum_{\substack{| K | = l\\ K \cap J = \emptyset}} \E \left[ W_{J \cup K}^2 \mid \mathcal{F}_J \right]  \right) ^ 2  \right] \right. \\
& ~~~+ \sum_{1 \leq l < i \leq d} \sum_{\substack{J_1 \cap J_2 = \emptyset\\| J_1 | = | J_2 | = i - l}} \E \left[ \left( \sum_{\substack{| K | = l\\ K \cap (J_1 \cup J_2) = \emptyset}} \E \left[ W_{J_1 \cup K} W_{J_2 \cup K} \mid \mathcal{F}_{J_1 \cup J_2} \right]  \right) ^2 \right] \\
&~~~ \left. + \sum_{1\leq l < i \leq d} \sum_{| J | = i - l} \E \left[ \left( \sum_{\substack{| K | = l\\ K \cap J = \emptyset}} \E \left[ W_K W_{J \cup K} \mid \mathcal{F}_J \right]  \right) ^2 \right] \right\} ^ {1 / 2}.
\end{align*}
\end{prop} 
  





\subsection{Hoeffding's combinatorial central limit theorem} \label{subsec:HoeCCLT}

Let $A=(a_{ij})_{i, j=1}^m$ be an $m\times m$ matrix of real numbers. Let $\pi$ be a uniform random permutation of $[m]$ and let $W=\sum_{i = 1}^{m} a_{i\pi(i)}$. Hoeffding \cite{H51} proved the central limit theorem for $W$. Now we introduce a quantitative version of Hoeffding's classical result, which will be used in our proof of the main theorems. Let 
$$
a_{i\cdot} = \frac{1}{m}\sum_{j=1}^{m} a_{ij},~~~~ a_{\cdot j} = \frac{1}{m}\sum_{i=1}^{m} a_{ij},~~~~
a_{\cdot \cdot} = \frac{1}{m^2}\sum_{i,j = 1}^{m} a_{ij}
$$
and $\dot{a_{ij}} = a_{ij} - a_{i\cdot} - a_{\cdot j} + a_{\cdot \cdot}$. We have the following proposition.


\begin{prop}[\cite{B84,CF15,G07,H51}] \label{prop:hoePrivaultcclt}
    For $m \ge 3$ and all $A$ with $\sigma_W^2 > 0$, we have
    \begin{align} \label{eq:cclt}
        d_{W / K} \left( \mathcal{L} ( \widehat{W} ), \mathcal{N}\left( 0, 1 \right)  \right) 
        \leq \frac{451}{m \sigma_W^3} \sum_{i, j=1}^{m} \left| \dot{a_{ij}} \right| ^3. 
    \end{align}
\end{prop}

\begin{lem} \label{lem:hoeffccltexample}
Let $J$ be a uniform random element of $\binom{[m]}{k}$ and set $W_{m, k} = \sum_{x \in J}^{} x$. Then we have
\begin{align*}
    d_{W / K} \left( \mathcal{L} \big(\widehat{W_{m, k}}\big), \mathcal{N}( 0, 1 )  \right) 
    \lesssim \sqrt{ \frac{m}{k (m - k)} }.
\end{align*}
\end{lem}

\begin{proof}[Proof of Lemma \ref{lem:hoeffccltexample}]
    
    Let $\alpha = (1^k, 0^{m-k})$, i.e., $\alpha_1 = \alpha_2 = \cdots = \alpha_k = 1$ and $\alpha_{k + 1} = \cdots = \alpha_m = 0$, and $x = (1, 2, \cdots, m)$. Let $A = (a_{ij})_{i, j=1}^m$ be the outer product of $\alpha$ and $x$. Then we have $W_{m, k} = \sum_{i = 1}^{m} a_{i \pi (i)}$, where $\pi$ is a uniform random permutation of $[m]$. One can check that
    \begin{align*}
        \dot{a_{ij}} &= \left( \alpha_i - \frac{k}{m} \right) \left( x_j - \frac{m + 1}{2} \right), \\
        \sigma_{W_{m, k}}^2 & = \frac{1}{m - 1} \sum_{i, j=1}^{m} \dot{a_{ij}}^2 = \frac{k (m - k) (m + 1)}{12}, 
    \end{align*}
    and
    \begin{align*}
        \sum_{i, j=1}^{m} \left|\dot{a_{ij}}\right|^3 
        & = \sum_{i=1}^{m} \left| \alpha_i - \frac{k}{m} \right|^3 \cdot\sum_{j=1}^{m} \left| x_j - \frac{m + 1}{2} \right|^3 \\
        & \leq m^4 \left( k \frac{(m - k)^3}{m^3} + (m - k) \frac{k^3}{m^3} \right) \\
        & \leq m ^ 3 k (m - k).
    \end{align*}Then the desired inequality follows by plugging the corresponding quantities into Proposition \ref{prop:hoePrivaultcclt}.
\end{proof}

\begin{rmk}
The bound for the Kolmogorov distance in \eqref{eq:cclt} was first obtained by Bolthausen \cite{B84} without the explicit constant. The constant $451$ was obtained by Chen and Fang \cite{CF15}. Goldstein \cite{G07} obtained the bound for the Wasserstein $W_1$ distance using zero-bias coupling. 
\end{rmk}

\subsection{P\'{o}lya frequency sequences} \label{sec:PolyaFrequence}

Let $\mathbf{p} = (p_0, p_1, \cdots, p_m)$ be a sequence of nonnegative real numbers summing up to $1$. Let $f(x) = \sum_{k=0}^{m} p_k x^k$ be its generating function. The sequence $\mathbf{p}$ is called a \emph{P\'{o}lya frequency (PF) sequence} if and only if the polynomial $f(x)$ is constant or has only real roots. PF sequences arise in combinatorics quite often, see Harper \cite{H67}.

Suppose that $\mathbf{p}$ is a PF sequence. Let $\mu = \sum_{i=1}^{m} i p_i$ and $\sigma^2 = \sum_{i=1}^{m} i^2 p_i - \mu^2$ denote respectively the mean and variance of the probability distribution generated by $\mathbf{p}$. We have the following tail bound for $\mathbf{p}$ from Pitman \cite{P97}.

\begin{prop}[\cite{P97}, Proposition 1] \label{prop:polyafreqdef}
    Let $(a_0, \cdots, a_n)$ be a sequence of nonnegative real numbers associated with the polynomial $A (z) := \sum_{k = 0}^{n} a_k z^k$ such that $A (1) > 0$. The following conditions are equivalent:
    \begin{enumerate}
        \item[(i)] The polynomial $A (z)$ is either constant or has only real roots;
        \item[(ii)] The normalized sequence $( \frac{ a_0 }{ A (1) }, \cdots, \frac{ a_n }{ A (1) } )$ is the distribution of the number $S_n$ of successes in $n$ independent trials with probability $p_i$ of success on the $i$-th trial, for some sequence of probabilities $0 \leq p_i \leq 1$. The roots of $A (z)$ are then given by $-(1 - p_i)/p_i$ for $i$ with $p_i > 0$.   
    \end{enumerate}
\end{prop}

\begin{prop}[\cite{P97}, inequality (11)] \label{prop:tail}
    For all $r > 0$, we have
    \begin{align*}
        \sum_{0 \leq i \leq \mu - r}^{} p_i \leq \exp \left( - \frac{r^2}{2 \mu} \right). 
    \end{align*}
\end{prop}



\begin{coro} \label{coro:weightPolyaFrequence}
Let $U$ be the random variable defined in \eqref{eq:rv-U}. Then there exist independent Bernoulli random variables $Y_1, \cdots, Y_{\lfloor n/2\rfloor}$ such that $U = \sum_{1 \leq i \leq \lfloor n /2 \rfloor }^{} Y_i$.
\end{coro}

\begin{proof}
The distribution of $U$ is $\big\{\binom{n - k}{k} d_n ^ k\big\}_{1\leq k \leq \lfloor n/2\rfloor}$, whose generating function is
    \begin{align*}
        g_n (z) = \sum_{k=1}^{\lfloor n/2\rfloor} \binom{n - k}{k} d_n ^ k z ^ k = \frac{ 1 }{ \sqrt{ 1 + 4 d_n z } } \left( \left( \frac{ 1 + \sqrt{ 1 + 4 d_n z } }{ 2 } \right) ^ {n + 1} - \left( \frac{ 1 - \sqrt{ 1 + 4 d_n z } }{ 2 } \right) ^ {n + 1} \right).
    \end{align*}
It is not hard to see that $g_n (z)$ has only real roots $z \leq - 1/ (4 d_n) $ such that $\left| 1 + \sqrt{ 1 + 4 d_n z } \right| = \left| 1 - \sqrt{ 1 + 4 d_n z } \right|$. Then the statement follows from Proposition \ref{prop:polyafreqdef}. 
\end{proof}



\begin{lem} \label{lem:weimean}
Let $U$ be the random variable defined in \eqref{eq:rv-U}. We have 
$$
\frac{5 - \sqrt{ 5 }}{10} n + O (1) \leq \E [U] \leq \frac{n}{2}.
$$
\end{lem}

\begin{proof}

    It is obvious that $\E [U] \leq n/2$. For the other inequality, notice that
    \begin{align*}
        &\ \ \ \ \sum_{k=1}^{\lfloor n/2\rfloor} \binom{n - k}{k} \sum_{k=1}^{\lfloor n/2\rfloor} \binom{n - k}{k} d_n ^ k k - \sum_{k=1}^{\lfloor n/2\rfloor} \binom{n - k}{k} k \sum_{k=1}^{\lfloor n/2\rfloor} \binom{n - k}{k} d_n ^ k \\
        & = \sum_{i < j}^{} \binom{n - i}{i} \binom{n - j}{j} (d_n ^ j - d_n ^ i) (j - i) \\
        & \ge 0.
    \end{align*}
    Therefore, we have
    \begin{align*}
        \E [U] = \frac{ \sum_{1\leq k\leq \lfloor n/2 \rfloor}^{} \binom{n - k}{k} d_n ^ k k }{ \sum_{1\leq k\leq \lfloor n/2 \rfloor}^{} \binom{n - k}{k} d_n ^ k } 
        \ge \frac{ \sum_{1\leq k\leq \lfloor n/2 \rfloor}^{} \binom{n - k}{k} k }{ \sum_{1\leq k\leq \lfloor n/2 \rfloor} \binom{n - k}{k} } 
        = \frac{5 - \sqrt{ 5 }}{10} n + O (1).
    \end{align*}The last equality is obtained in \cite{KST20} by the probability generating function.
\end{proof}

\begin{coro} \label{coro:lowtail}
    For sufficiently large $n$, we have
    \begin{align*}
        \mathbb{E} \left[ \min \left\{ \frac{ 1 }{ \sqrt{ U } }, 1 \right\}  \right]
        \lesssim \frac{ 1 }{ \sqrt{ n } }.
    \end{align*}
\end{coro}
    
\begin{proof}
    
        Combining Proposition \ref{prop:tail} and Lemma \ref{lem:weimean}, we have for sufficiently large $n$ that
        \begin{align*}
            \P \left( U \leq \frac{n}{6} \right) 
            \leq \exp \left( - \frac{\left( \E [U] - \frac{n}{6} \right) ^2}{2 \E [U]} \right) 
            \leq \exp \left( - \frac{ \left( \frac{ 5 - \sqrt{ 5 } }{ 10 } n - \frac{ n }{ 6 } + O (1) \right) ^ 2 }{ n } \right) 
            \leq e^{- 0.001 n}.
        \end{align*}
        Hence, we have for $n$ large enough that
        \begin{align*}
            \mathbb{E} \left[ \min \left\{ \frac{ 1 }{ \sqrt{ U } }, 1 \right\}  \right] 
            & = \mathbb{E} \left[ \min \left\{ \frac{ 1 }{ \sqrt{ U } }, 1 \right\} \ \middle| \ U \leq \frac{ n }{ 6 } \right]
            \mathbb{P} \left( U \leq \frac{ n }{ 6 } \right) \\
            &\ \ \  + \mathbb{E} \left[ \frac{ 1 }{ \sqrt{ U } } \ \middle| \ \frac{ n }{ 6 } < U \leq \left\lfloor \frac{ n }{ 2 } \right\rfloor  \right]
            \mathbb{P} \left( \frac{ n }{ 6 } < U \leq \left\lfloor \frac{ n }{ 2 } \right\rfloor \right) \\
            & \leq \mathbb{P} \left( U \leq \frac{ n }{ 6 } \right) + \sqrt{ \frac{ 6 }{ n } } \\
            & \lesssim \frac{ 1 }{ \sqrt{ n } }.
        \end{align*} 
\end{proof}
    


\subsection{Distances between normal mixtures}

The following result bounds the Kolmogorov and Wasserstein $W_1$ distances of two normal distributions with the same mean, respectively. The proof given by \cite{NP12} (for the case $\sigma_1, \sigma_2 > 0$) is based on Stein's method. The result is trivial for $\sigma_1 \sigma_2 = 0$.

\begin{lem}[\cite{NP12}, Proposition 3.6.1] \label{lem:normaldistancekolwass}
    Let $\mathcal{N}\left( \mu, \sigma_1^2 \right)$ and $ \mathcal{N}\left( \mu, \sigma_2^2 \right) $ be two normal distributions with variances $\sigma_1^2$ and $\sigma_2^2$ such that $\sigma_1^2 + \sigma_2^2 > 0$. Then we have
    \begin{align*}
        d_K \left( \mathcal{N}\left( \mu, \sigma_1^2 \right), \mathcal{N}\left( \mu, \sigma_2^2 \right) \right)  
        & \leq 
        \frac{\left| \sigma_1^2 - \sigma_2^2 \right|}{\max\{\sigma_1^2, \sigma_2^2\}}, \\
        d_W \left( \mathcal{N}\left( \mu, \sigma_1^2 \right), \mathcal{N}\left( \mu, \sigma_2^2 \right) \right) 
        & \leq 
        \sqrt{ \frac{ 2 }{ \pi} } \frac{\left| \sigma_1^2 - \sigma_2^2 \right|}{\max\{\sigma_1, \sigma_2\}}.
    \end{align*}
\end{lem}


We obtain the following lemma.

\begin{lem} \label{lem:mixdis}
    Let $\sigma, \sigma_* > 0$. Suppose that $X$ is a Gaussian mixture of $\mathcal{N}(\mu, \sigma^2)$, where $\mu$ is a random variable with mean $\mu_0$ and variance $\sigma_*^2$. We write $\widehat{X} := (X-\mu_0)/\sqrt{\sigma^2+\sigma_*^2}$ and $\widehat{\mu} := (\mu-\mu_0)/\sigma_*$. Then we have
    \begin{align*}
        d_{K} \left( \mathcal{L} ( \widehat{X} ) , \mathcal{N}(0, 1) \right) 
        & \leq \frac{ 1 }{ \sqrt{ 2 \pi } } \frac{ \sigma_* }{ \sigma } d_W \left( \mathcal{L} \left( \widehat{\mu} \right), \mathcal{N}\left( 0, 1 \right)  \right). 
        \\
        d_W \left( \mathcal{L} ( \widehat{X} ) , \mathcal{N}(0, 1) \right) 
        & \leq \frac{\sigma_*}{\sqrt{ \sigma^2 + \sigma_*^2 }} d_W \left( \mathcal{L} \left( \widehat{\mu} \right) , \mathcal{N}\left( 0, 1 \right) \right). 
    \end{align*}
\end{lem}

\begin{proof}
One can check that $\E [X] = \mu_0$ and $\Var (X) = \sigma^2 + \sigma_*^2$. Hence, $\widehat{X}$ is the normalized version of $X$. We write $F(y)=\P\big(\widehat{X}\leq y\big)$ for $y \in \R$. Recall that $\varphi$ and $\Phi$ are the probability density and cumulative distribution functions of the standard normal distribution, respectively. One can use Fubini's theorem to show that
    \begin{align*}
        F (y)  &= \int_{\R}^{} \varphi (u) \P \left( \frac{\mu - \mu_0}{\sigma_*} \leq \frac{\sqrt{ \sigma^2 + \sigma_*^2 } y - \sigma u}{\sigma_*} \right) d u,\\
        \Phi (y) &= \int_{\R}^{} \varphi (u) \Phi \left( \frac{\sqrt{ \sigma^2 + \sigma_*^2 } y - \sigma u}{\sigma_*} \right) d u.
    \end{align*}
    For the Kolmogorov distance, we have
    \begin{align*}
        \left| F(y) - \Phi \left( y \right) \right| 
        &\leq \int_{\R}^{} \varphi (u) \left| \P \left( \frac{\mu - \mu_0}{\sigma_*} \leq \frac{\sqrt{ \sigma^2 + \sigma_*^2 } y - \sigma u}{\sigma_*} \right) - \Phi \left( \frac{\sqrt{ \sigma^2 + \sigma_*^2 } y - \sigma u}{\sigma_*} \right) \right|   d u \\
        & \leq \frac{ 1 }{ \sqrt{ 2 \pi } } 
        \int_{\R}^{} \left| \P \left( \frac{\mu - \mu_0}{\sigma_*} \leq \frac{\sqrt{ \sigma^2 + \sigma_*^2 } y - \sigma u}{\sigma_*} \right) - \Phi \left( \frac{\sqrt{ \sigma^2 + \sigma_*^2 } y - \sigma u}{\sigma_*} \right) \right|   d u \\
        & = \frac{ 1 }{ \sqrt{ 2 \pi } } \frac{ \sigma_* }{ \sigma }
        \int_{\R}^{} \left| \P \left( \frac{\mu - \mu_0}{\sigma_*} \leq u \right) - \Phi \left( u \right) \right|   d u \\
        & = \frac{ 1 }{ \sqrt{ 2 \pi } } \frac{ \sigma_* }{ \sigma } d_W \left( \mathcal{L} \left( \widehat{\mu} \right), \mathcal{N}\left( 0, 1 \right)  \right).
    \end{align*}
    For Wasserstein $W_1$ distance, we have
    \begin{align*}
        &\ \ \ \ \int_{\R}^{} \left| F(y) - \Phi (y) \right| dy \\
        & = \int_{\R}^{} \left| \int_{\R}^{} \varphi (u) \left( \P \left( \frac{\mu - \mu_0}{\sigma_*} \leq \frac{\sqrt{ \sigma^2 + \sigma_*^2 } y - \sigma u}{\sigma_*} \right) - \Phi \left( \frac{\sqrt{ \sigma^2 + \sigma_*^2 } y - \sigma u}{\sigma_*} \right) \right) du \right| dy \\
        & \leq \int_{\R}^{} dy \int_{\R}^{} \varphi (u) \left| \P \left( \frac{\mu - \mu_0}{\sigma_*} \leq \frac{\sqrt{ \sigma^2 + \sigma_*^2 } y - \sigma u}{\sigma_*} \right) - \Phi \left( \frac{\sqrt{ \sigma^2 + \sigma_*^2 } y - \sigma u}{\sigma_*} \right) \right|  du \\
        & = \int_{\R}^{} \varphi (u) du \int_{\R}^{} \left| \P \left( \frac{\mu - \mu_0}{\sigma_*} \leq \frac{\sqrt{ \sigma^2 + \sigma_*^2 } y - \sigma u}{\sigma_*} \right) - \Phi \left( \frac{\sqrt{ \sigma^2 + \sigma_*^2 } y - \sigma u}{\sigma_*} \right) \right| dy \\
        & =  \frac{\sigma_*}{\sqrt{ \sigma^2 + \sigma_*^2 }} \int_{\R}^{} \varphi (u) du \int_{\R}^{} \left| \P \left( \frac{\mu - \mu_0}{\sigma_*} \leq y \right) - \Phi \left( y \right) \right| dy \\
        & \leq \frac{\sigma_*}{\sqrt{ \sigma^2 + \sigma_*^2 }} d_W \left( \mathcal{L} \left( \widehat{\mu} \right) , \mathcal{N}\left( 0, 1 \right)  \right) \int_{\R}^{} \varphi (u) du \\
        & = \frac{\sigma_*}{\sqrt{ \sigma^2 + \sigma_*^2 }} d_W \left( \mathcal{L} \left( \widehat{\mu} \right), \mathcal{N}\left( 0, 1 \right)  \right).
    \end{align*}
\end{proof}

\subsection{Concentration inequalities of mixture distributions}

For a probability distribution $\mu$, the logarithmic moment generating function $\alpha_{\mu}(\lambda)$ and the tail probability $\beta_{\mu} (r)$ are defined respectively as
\begin{align*}
    \alpha_{\mu} (\lambda) &= \log_{} \mathbb{E}_{\mu} \big[ e ^ {\lambda (X - \mathbb{E}_{\mu} \left[ X \right])} \big],~~\lambda\in\R\\
    \beta_{\mu} (r) &= \mathbb{P}_{\mu} \left( \left| X - \mathbb{E}_{\mu} \left[ X \right] \right| \ge r \right),~~r>0.
\end{align*}

\begin{prop}[\cite{BLM13}, Theorem 2.1] \label{prop:subgaudef}
The following results hold.

\begin{enumerate}
    \item[(1)] Suppose that there exists a constant $C>0$ such that $\alpha_{\mu} (\lambda) \leq C \lambda^2$/4 for all $\lambda \in \R$, then we have $\beta_{\mu} (r) \leq 2 e^{-r^2/C}$ for all $ r > 0$.
    \item[(2)] Suppose that there exists a constant $C>0$ such that $\beta_{\mu} (r) \leq 2 e^{-r^2/C}$ for all $r > 0$, then we have $\alpha_{\mu} (\lambda) \leq 4 C \lambda^2$ for all $\lambda \in \R$.
\end{enumerate}
\end{prop}

We obtain the following concentration inequality of mixture distributions.
\begin{lem} \label{lem:mixcon}
Let $\mu$ be the mixture of $\{\mu_{\theta}\}_{\theta \in \Theta}$ with the mixing distribution $\nu$, i.e., $\theta\sim \nu$. Then we have
\begin{align*}
    \alpha_{\mu} (\lambda) 
    \leq \sup_{\theta \in \Theta} \alpha_{\mu_{\theta}} (\lambda) 
    + \log_{} \int_{\Theta}^{} e ^ {\lambda \left( \mathbb{E}_{\mu_{\theta}} \left[ X \right] - \mathbb{E}_{\mu} \left[ X \right] \right) } \nu (d \theta).
\end{align*}
\end{lem}

\begin{proof} 

We have
\begin{align*}
    \alpha_{\mu} (\lambda) 
    = \log_{} \mathbb{E}_{\mu} \left[ e ^ {\lambda \left( X - \mathbb{E}_{\mu} \left[ X \right] \right) } \right]
    & = \log_{} \mathbb{E}_{\theta} \left[ \mathbb{E}_{\mu_{\theta}} \left[ e ^ {\lambda \left( X - \mathbb{E}_{\mu_{\theta}} \left[ X \right] \right) } \cdot e ^ {\lambda \left( \mathbb{E}_{\mu_{\theta}} \left[ X \right] - \mathbb{E}_{\mu} \left[ X \right] \right) } \right] \right] \\
    & = \log_{} \mathbb{E}_{\theta} 
    \left[ e ^ {\alpha_{\mu_{\theta}} (\lambda)} \cdot e ^ {\lambda \left( \mathbb{E}_{\mu_{\theta}} \left[ X \right] - \mathbb{E}_{\mu} \left[ X \right] \right) } \right] \\
    & \leq \sup_{\theta \in \Theta} \alpha_{\mu_{\theta}} (\lambda) + \log_{} \mathbb{E}_{\theta} 
    \left[ e ^ {\lambda \left( \mathbb{E}_{\mu_{\theta}} \left[ X \right] - \mathbb{E}_{\mu} \left[ X \right] \right) } \right].
\end{align*}

\end{proof}

We say that a function $f: X ^ n \to \R$ has the \emph{bounded differences property} if there are constants $c_1, \cdots, c_n$ such that for $1 \leq i \leq n$,
\begin{align*}
    \sup_{x_1, \cdots, x_n, x_i' \in X} \left| f (x_1, \cdots, x_{i-1}, x_i, x_{i+1}, \cdots, x_n) - f (x_1, \cdots, x_{i - 1}, x_i', x_{i + 1}, \cdots, x_n) \right| \leq c_i.
\end{align*}

The following concentration inequality of McDiarmid \cite{M89} will be used in our proof of the main theorems.

\begin{prop}[\cite{M89}] \label{prop:boundiff}
Suppose the function $f$ satisfies the bounded differences assumption with constants $c_1, \cdots, c_n$. 
Let $\{X_i\}_{i=1}^n$ be independent random variables. Then we have
\begin{align*}
    \mathbb{P} \left( \left| f(X_1, \cdots, X_n)-\E[f(X_1, \cdots, X_n)]\right| \ge t \right) \leq 2 \exp \left( - \frac{2t^2}{\sum_{i = 1}^{n} c_i^2} \right) .
\end{align*}
\end{prop}






The following concentration inequality on the slices $\binom{[n]}{k}$ comes from Bobkov \cite{B04}.

\begin{prop}[\cite{B04}, Theorem 2.1] \label{prop:ConcentrationSlice}
Let $L$ be a positive constant and let $g: \binom{[n]}{k} \to \R$ be a function such that $\left| g (J) - g (J') \right| \leq L$ for any $J, J' \in \binom{[n]}{k}$ with $\left| J \Delta J' \right| = 2$. Then we have for all $t \ge 0$ that
\begin{align*}
    \mathbb{P} \left( \left| g (I) - \E [g (I)] \right| \ge t \right)
    \leq 2 \exp \left( - \frac{ t ^ 2 }{ \min \left\{ k, n - k \right\} L ^ 2 } \right),
\end{align*}
where $I$ is a uniform random element of $\binom{[n]}{k}$. As a consequence, we have
\begin{align*}
    \text{Var} \left( g (I) \right) \leq 2 \min \left\{ k, n - k \right\} L ^ 2.
\end{align*}
\end{prop}






We will also need Efron-Stein inequality on product spaces.


\begin{prop}[\cite{BLM13}, Corollary 3.2] \label{prop:BoundDiffConcentra}
Suppose the function $f$ satisfies the bounded differences assumption with constants $c_1, \cdots, c_n$. 
Let $\{X_i\}_{i=1}^n$ be independent random variables. Then we have
\begin{align*}
    \text{Var} \left( f (X_1, \cdots, X_n) \right) \leq \frac{ 1 }{ 4 } \sum_{i = 1}^{n} c_i ^ 2.
\end{align*}

\end{prop}





\section{Proofs of Main Theorems}


\subsection{Some technical estimates}

This subsection contains some technical estimates that will be used in our proofs of the main theorems. We first illustrate the structure of $W=f(X_1, \cdots, X_{n-1})$ (defined in Theorem \ref{thm:normalgeneral}), where $X_1, \cdots, X_{n-1}$ are random variables such that $\big(X_1+\frac{d_n+1}{2}, \cdots, X_{n-1}+\frac{d_n+1}{2}\big)$ is uniform on $\mathcal{SB}_n =\{(x_1, \cdots, x_{n-1})\in\{0, 1, \cdots, d_n\}^{n - 1}: x_ix_{i+1}=0~\text{for}~1\leq i\leq n-2\}$. Recall the shift operation $\tau$ defined in \eqref{eq:tau-action}. It is not hard to see that $\mathcal{SB}_n$ can be partitioned into sets $\mathcal{SB}_{n, \tau (J)}=\{(x_1, \cdots, x_{n-1})\in \mathcal{SB}_n: x_i\neq 0~\text{if and only if}~i\in \tau (J)\}$ with ${J \in \cup_{0 \leq k \leq n / 2 } \binom{[n - k]}{k}}$. Clearly, we have $|\mathcal{SB}_{n, \tau (J)}|=d_n^{|J|}$ and $|\mathcal{SB}_n|=\sum_{k=0}^{\lfloor n/2\rfloor} {n-k\choose k}d_n^k$. For any $J \subseteq [n - 1]$, we write $W^{J} = f ( X_1^J, \cdots, X_{n - 1}^J ) $, where $X_j^J = - (d_n + 1)/ 2$ for $j \not \in J$ and $X_j^J$ is uniformly distributed on $S:=\left\{ - \frac{ d_n - 1 }{ 2 }, - \frac{ d_n - 3 }{ 2 }, \cdots, \frac{ d_n - 1 }{ 2 } \right\}$ for $j \in J$. For $J \in \cup_{0 \leq k \leq n /2 }^{} \binom{[n - k]}{k}$, the random variable $W ^ {\tau (J)}$ is the conditioning of $W$ on the event that $\big(X_1+\frac{d_n+1}{2}, \cdots, X_{n-1}+\frac{d_n+1}{2}\big)\in \mathcal{SB}_{n, \tau (J)}$. Hence, we can write $W$ as the mixture of $W ^ {\tau (J)}$ with the mixing law $\nu$ defined as
\begin{equation}\label{eq:nu}
\nu(J)=d_n^{|J|}/|\mathcal{SB}_n|,~~\text{where}~J \in \cup_{0 \leq k \leq n /2 }^{} \binom{[n - k]}{k}.
\end{equation}


Now we deduce the Hoeffding decomposition of $W ^ J$ for $J \subseteq [n - 1]$. For brevity, we will drop the superscript in the random variables $X_j^J$ when it is clear from the context. Let $X$ be a uniform random variable taking values in $S$ defined before. Recall the function $f$ in Theorem \ref{thm:normalgeneral}. One can check that
\begin{align*}
W^J&=\sum_{i\in J}g_i(X_i)+a\sum_{\substack{i<j\\i, j \in J}}X_iX_j-a(n-1-|J|)\frac{d_n+1}{2}\sum_{i\in J}X_i\\
&~~~~~~~~~+a{n-1-|J|\choose 2}\frac{(d_n+1)^2}{4}-a{n-1\choose 2}\frac{(d_n+1)^2}{4}.
\end{align*}
For $i \in J$, we define $W_{\{i\}} = g_i (X_i) - \mathbb{E} \left[ g_i (X_i) \right] - a (n - 1 - |J|) \frac{ d_n + 1 }{ 2 } X_i$. For $i, j \in J$ and $i < j$, we define $W_{\{i, j\}} = a X_i X_j$. For brevity, we also write $W_i$ for $W_{\{i\}}$. Note that $\{X_i\}_{i\in J}$ are symmetric i.i.d. random variables. Hence, we have
\begin{align} 
    \mu_J :&= \mathbb{E} \left[ W ^ J \right]
    = \sum_{i \in J}^{} \mathbb{E} \left[ g_i (X_i) \right]
    + a \binom{n - 1 - |J|}{2} \frac{ (d_n + 1) ^ 2 }{ 4 }
    - a \binom{n - 1}{2} \frac{ (d_n + 1) ^ 2 }{ 4 } \label{eq:mujdefinition}\\
    \sigma_J ^ 2 :&= \text{Var} \left( W ^ J \right)
    = \sum_{i \in J}^{} \E \left[ W_{i} ^ 2 \right] 
    + a ^ 2 \binom{|J|}{2} \text{Var}(X)^2 \nonumber \\
    & = \sum_{i \in J}^{} \text{Var} \left( g_i(X_i) \right) - a (n - 1 - |J|) (d_n + 1) \sum_{i \in J}^{} \text{Cov} \left( g_i(X_i), X_i \right) \nonumber \\
    &~~~~~~+a ^ 2 |J| (n - 1 - |J|) ^ 2 \frac{ (d_n + 1) ^ 2 }{ 4 } \text{Var}(X) + a ^ 2 \binom{|J|}{2} \text{Var}(X)^2. \label{eq:varjdef}
\end{align}
Then one can check that
\begin{align} \label{eq:hoeffdecompwj}
    W ^ {J} - \mu_J
    = \sum_{i \in J}^{} W_i + \sum_{\{i, j\} \subseteq J}^{} W_{\{i, j\}}
\end{align}
is the Hoeffding decomposition of $W ^ J - \mu_J$ (according to Definition \ref{def:HoeffdingDecomposition}). In particular, for non-identical $i, j \in J$, it holds that $\E [W_i] = \E [W_{\{i, j\}}] = 0$. \emph{Note that $W_i$ and $W_{\{i, j\}}$ are relevant to $J$, but for simplicity, we omit this connection in the symbols.}   

Now we derive estimates of $\sigma_J ^ 2$ for $J \subseteq [n - 1]$. If $a \neq 0$, we know from Condition \ref{cond:Nondegenerate} that there exists an absolute constant $C \in [0, 1)$ such that for any $\alpha \in \R$ and all $i \in [n - 1]$,
\begin{align*} 
    \text{Var} \left( g_i (X_i) - \alpha X_i \right) 
    & = \text{Var} \left( g_i (X_i) \right) + \alpha ^ 2 \text{Var} \left( X_i \right) - 2 \alpha \text{Cov} \left( g_i (X_i), X_i \right) \\
    & \ge \text{Var} \left( g_i (X_i) \right) - \frac{ \text{Cov} \left( g_i (X_i), X_i \right) ^ 2 }{ \text{Var} \left( X_i \right) } \\
    & \ge (1 - C) \text{Var} \left( g_i (X_i) \right).
\end{align*}
Particularly, we have for all $i \in [n - 1]$ that
\begin{align} \label{eq:var(Y-aX)}
    \E [W_i ^ 2] \ge (1 - C) \text{Var} \left( g_i (X_i) \right).
\end{align}
For all $i\in [n-1]$ and $x \in S$, Conditions \ref{cond:NearIndependence} and \ref{cond:Boundedness} yield
\begin{align} \label{eq:wjboundbyVar}
    \left| W_i (x) \right| \lesssim \left| g_i (x) \right| + \left| \mathbb{E} \left[ g_i(X_i) \right] \right| + |a| n d_n ^ 2
    \lesssim \sqrt{ \inf_{j \in [n - 1]} \text{Var} \left( g_j(X_j) \right) }.
\end{align}
Assume that $J \neq \varnothing$. On the one hand, by \eqref{eq:varjdef} and \eqref{eq:var(Y-aX)}, we have
\begin{align} \label{eq:varwJLowerBound}
    \sigma_J^2
    &= \sum_{i \in J}^{} \E [W_i ^ 2] + a ^ 2 \binom{|J|}{2} \Var(X)^2\notag\\
    & \ge \sum_{i \in J}^{} \E [W_i ^ 2] 
    \gtrsim |J| \inf_{i \in [n - 1]} \text{Var} \left( g_i(X_i) \right). 
\end{align}
On the other hand, by Condition \ref{cond:NearIndependence} and \eqref{eq:wjboundbyVar}, we have
\begin{align} \label{eq:varwJUpperBound}
    \sigma_J^2
    \leq \sum_{i \in J}^{} \E [W_i ^ 2] + a ^ 2 n ^ 2 d_n ^ 4
    \lesssim |J| \inf_{i \in [n - 1]} \text{Var} \left( g_i (X_i)\right).
\end{align}
Recall that $\sigma_k ^ 2$ denotes the variance of the uniform mixture of $W ^ {\tau (J)}$ for $J \in \binom{[n - k]}{k}$. As a consequence, by the law of total variance and Lemma \ref{lem:weimean}, we obtain
\begin{align} \label{eq:sigmakmeanlowerbound}
    \sigma_W ^ 2:=\text{Var}(W) 
    &\ge \E_{k\sim \mathcal{L}(U)} \left[ \sigma_k ^ 2 \right] 
    \ge \E_{k\sim \mathcal{L}(U)} \left[ \mathbb{E}_{J \sim \mathcal{U}\left(\binom{[n - k ]}{k}\right)} \big[ \sigma_{\tau (J)} ^ 2 \big] \right]  \notag\\
    &\gtrsim \inf_{i \in [n - 1]} \text{Var} \left( g_i(X_i) \right) \cdot \E_{k\sim \mathcal{L}(U)} [k]\notag\\
    &\gtrsim n \inf_{i \in [n - 1]} \text{Var} \left( g_i(X_i) \right) . 
\end{align}
Here, the random variable $U$ is defined in \eqref{eq:rv-U}.

Next, we establish bounds on the mean, variance, and covariance of the functions $g_i$. For all $i\in [n-1]$, we apply Condition \ref{cond:Boundedness} to obtain
\begin{equation} \label{eq:gi-mean}
    |\E [g_i(X)]| \leq \E |g_i(X)|\leq \max_{x\in S} |g_i(x)|
    \lesssim \sqrt{ \inf_{j \in [n - 1]} \text{Var} \left( g_j(X) \right) }.
\end{equation} 
In particular, we have $\inf_{j \in [n - 1]} \text{Var} \left( g_j(X) \right) > 0$. Owing to Condition \ref{cond:ArithmeticProgression}, the sequence $\{\E [g_i(X)]\}_{i=1}^{n-1}$ is an arithmetic progression. Hence, we have for all $i\in [n-2]$ that
\begin{align} \label{eq:gi-mean-diff}
    |\E [g_i(X)]-\E [g_{i-1}(X)]| &=\frac{1}{n-2}|\E [g_{n-1}(X)]-\E [g_{1}(X)]|\notag\\
    &\lesssim \frac{1}{n}\sqrt{ \inf_{j \in [n - 1]} \text{Var} \left( g_j(X) \right) },
\end{align} 
where the inequality follows from the triangle inequality and inequality \eqref{eq:gi-mean}. Similarly, we have the following variance difference bound
\begin{align} \label{eq:gi-var-diff}
    |\text{Var} ( g_{i + 1}(X)) - \text{Var} ( g_i(X) ) | 
    & = \big| \E \big[ ( g_{i + 1} (X) - \E [g_{i + 1}(X)] ) ^ 2 - ( g_{i} (X) - \E [g_i(X))] ) ^ 2 \big]  \big| \nonumber \\
    &=\big|\E \big[ \big(g_{i+1}(X)-g_i(X)-\E [g_{i + 1}(X)]+\E [g_{i }(X)]\big)\notag\\
    &~~~~~\times \big(g_{i+1}(X)+g_i(X)-\E [g_{i + 1}(X)]-\E [g_{i }(X)]\big)\big]\big|\notag\\
    &=\frac{1}{n-2}\big|\E \big[ \big(g_{n-1}(X)-g_1(X)-\E [g_{n-1}(X)]+\E [g_1(X)]\big)\notag\\
    &~~~~~~~~~~~~\times \big(g_{i+1}(X)+g_i(X)-\E [g_{i + 1}(X)]-\E [g_{i }(X)]\big)\big]\big|\notag\\
    &\lesssim \frac{1}{n}\inf_{j \in [n - 1]} \text{Var} \left( g_j(X) \right),
\end{align}
where the inequality follows from bounding each term in the expectation via Condition \ref{cond:Boundedness}. Then we apply inequality \eqref{eq:gi-var-diff} and the triangle inequality to obtain  
\begin{equation}\label{eq:gi-var}
\text{Var}(g_i(X))\lesssim \inf_{j\in [n-1]} \text{Var}(g_j(X)).
\end{equation}
For all $i \in [n - 2]$, the arithmetic property of $\{g_i(x)\}_{i=1}^{n-1}$ for $x\in S$ and Condition \ref{cond:Boundedness} yield that
\begin{align} \label{eq:gi-diff-var}
    \text{Var} \left( g_{i + 1} (X) - g_i (X) \right) 
    &\leq \E\big(g_{i + 1} (X) - g_i (X)\big)^2\notag\\
&= \frac{ 1 }{ (n - 2) ^ 2 } \E\big(g_{n- 1} (X) - g_1 (X)\big)^2\notag\\
    &\lesssim \frac{ 1 }{ n ^ 2 } \inf_{j \in [n - 1]} \text{Var} \left( g_j(X) \right).
\end{align}
For all $i \in [n - 2]$, we apply Cauchy-Schwarz inequality and inequality \eqref{eq:gi-diff-var} to obtain
\begin{align} \label{eq:gi-cov-diff}
    | \text{Cov} ( g_{i + 1}(X), X) -\text{Cov} ( g_i(X), X  )| 
    & = | \text{Cov}( g_{i + 1} (X)-g_i (X), X ) | \nonumber \\
    & \leq \sqrt{\text{Var}(X) \cdot \text{Var} ( g_{i + 1} (X)-g_i (X))} \nonumber \\
    & \lesssim \frac{ d_n }{ n } \sqrt{\inf_{j \in [n - 1]} \text{Var} \left( g_j(X) \right)}.
\end{align}
For all $i \in [n - 1]$, the Cauchy-Schwartz inequality and inequality \eqref{eq:gi-var} yield that
\begin{equation} \label{eq:gi-cov}
    | \text{Cov}( g_i(X), X)| 
    \leq \sqrt{\text{Var}(g_i(X))\cdot\text{Var}(X)}
    \lesssim d_n \sqrt{ \inf_{j \in [n - 1]} \text{Var} \left( g_j(X) \right) }.
\end{equation}

\subsection{Proof of Theorem \ref{thm:normalgeneral}} \label{subsection:proofnormalgeneral}

In this subsection, we provide the proof of Theorem \ref{thm:normalgeneral}.

\begin{proof}[Proof of Theorem \ref{thm:normalgeneral}]

As outlined in Section \ref{sec:thm:normalgeneral}, we will construct $\{W ^ {(i)}\}_{i=1}^4$ with the same mean and variance as $W$, and bound the distances $d_{W / K} \big( \mathcal{L} ( \widehat{W ^ {(i)}} ), \mathcal{L} ( \widehat{W ^ {(i + 1)}} ) \big)$ sequentially. Recall that $W$ is the mixture of $W ^ {\tau (J)}$ with the mixing law $\nu$ defined in \eqref{eq:nu}.


\textbf { Step 1.} \textbf{Bound} \bm{$d_{W / K} \big( \mathcal{L} ( \widehat{W}), \mathcal{L} ( \widehat{W ^ {(1)}}) \big)$}\textbf{.} 
Recall that $\mathcal{L} \left( W^{(1)} \right)$ is the mixture of normal distributions $\mathcal{N}\big( \mu_{\tau (J)}, \sigma_{\tau (J)} ^ 2 \big)$ with the same mixing law $\nu$.
It is easy to see that $\mu_{W^{(1)}} = \mu_{W}$ and $\sigma_{W^{(1)}} = \sigma_{W}$. Hence, $\mathcal{L} \big(\widehat{W^{(1)}}\big)$ is the mixture of $\mathcal{N}\Big( \frac{\mu_{\tau (J)} - \mu_{W}}{\sigma_{W}}, \frac{\sigma_{\tau (J)}^2}{\sigma_{W}^2} \Big)$ with the mixing law $\nu$.
We first bound $d_{W / K} \big( \mathcal{L} ( \widehat{W ^ L}), \mathcal{N}\left( 0, 1 \right)\big)$ for all nonempty sets $L \subseteq [n - 1]$. The Hoeffding decomposition of $W^L-\mu_L$ given in \eqref{eq:hoeffdecompwj} and Proposition \ref{prop:hoePrivault} (the case $d = 2$) yield that
\begin{equation}\label{eq:W_L-to-normal}
d_{W / K} \left( \mathcal{L} ( \widehat{W ^ L}), \mathcal{N}\left( 0, 1 \right)\right)\lesssim \frac{\sqrt{A+B+C}}{\text{Var}(W^L)},
\end{equation}
where the quantities $A, B, C$ are defined below.
\begin{align}\label{eq:term-A}
    A :&=\sum_{0\leq l<i\leq 2} \sum_{\left| J \right| = i - l }^{} \E \left[ \left( \sum_{\left| K \right| = l, K \cap J = \varnothing }^{} \mathbb{E} \left[ (W_{J \cup K}) ^ 2 ~| ~\mathcal{F}_J \right] \right) ^ 2 \right] \notag\\
    & = \sum_{i = 1}^{2} \sum_{\left| J \right| = i, J \subseteq L }^{} \mathbb{E} \left[ W_J ^ 4 \right] 
    + \sum_{j \in L}^{} 
    \mathbb{E} \left[ \left( \sum_{k \neq j, k \in L}^{} a ^ 2 \gamma_2 X_j ^ 2 \right) ^ 2 \right] \notag\\
    & = \sum_{i \in L}^{} \mathbb{E} \left[ W_i ^ 4 \right]
    + a ^ 4 \binom{|L|}{2} \gamma_4 ^ 2
    + a ^ 4 |L| (|L| - 1) ^ 2 \gamma_2 ^ 2 \gamma_4 \notag\\
    & \lesssim \sum_{i \in L}^{} \mathbb{E} \left[ W_i ^ 4 \right]
    + a ^ 4 |L| ^ 3 d_n ^ 8,
\end{align}
where $\gamma_i=\E[X^i]$, defined in \eqref{eq:gamma_i}, and $X$ is uniform on $\left\{ - \frac{ d_n - 1 }{ 2 }, - \frac{ d_n - 3 }{ 2 }, \cdots, \frac{ d_n - 1 }{ 2 } \right\}$. Also,
\begin{align}\label{eq:term-B}
    B :&= \sum_{1\leq l<i\leq 2} \sum_{\substack{|J_1| = |J_2| = i - l \\ J_1 \cap J_2 = \varnothing} }^{} \mathbb{E} \left[ \left( \sum_{\substack{|K| = l\\ K \cap (J_1 \cup J_2) = \varnothing} }^{} \E \left[W_{J_1 \cup K } W_{J_2 \cup K} ~| ~\mathcal{F}_{J_1 \cup J_2}\right] \right) ^ 2 \right] \notag\\
    & = \sum_{\substack{j_1, j_2 \in L \\ j_1 \neq j_2}}^{}
    \mathbb{E} \left[ \left( \sum_{k \in L, k \neq j_1, j_2}^{} a ^ 2 \gamma_2 X_{j_1} X_{j_2} \right) ^ 2 \right] \notag\\
    & \leq a ^ 4 |L| (|L| - 1) (|L| - 2) ^ 2 \gamma_2 ^ 4 \notag\\
    & \lesssim a ^ 4 |L| ^ 4 d_n ^ 8.
\end{align}
By Cauchy-Schwarz inequality, we have
\begin{align}\label{eq:term-C}
    C :&=\sum_{1\leq l<i\leq 2}^{} \sum_{|J| = i - l}^{} \mathbb{E} \left[ \left( \sum_{|K| = l, K \cap J = \varnothing }^{} \E [W_K W_{J \cup K} | \mathcal{F}_J] \right) ^ 2 \right] \notag\\
    & = \sum_{j \in L}^{} 
    \mathbb{E} \left[ \left( \sum_{k \in L, k \neq j}^{} a X_j \E [W_k X_k] \right) ^ 2 \right] \notag\\
    & = a ^ 2 \gamma_2 \sum_{j \in L}^{} \left( \sum_{k \in L, k \neq j}^{} \E [W_k X_k] \right) ^ 2  \notag\\
    & \leq a ^ 2 \gamma_2 (|L| - 1) \sum_{j \in L}^{} \sum_{k \in L, k \neq j}^{} (\E [W_k X_k]) ^ 2 \notag\\
    & \leq a ^ 2 \gamma_2 (|L| - 1) \sum_{j \in L}^{} \sum_{k \in L, k \neq j}^{} \E [X_k ^ 2] \E [W_k ^ 2] \notag\\
    & \lesssim a ^ 2 |L| ^ 2 d_n ^ 4 \sum_{j \in L}^{} \E [W_j ^ 2].
\end{align}
For all $J \in \cup_{1 \leq i \leq \left\lfloor n / 2 \right\rfloor }^{} \binom{[n - k]}{k}$,  we combine Condition \ref{cond:NearIndependence} and inequalities \eqref{eq:var(Y-aX)}, \eqref{eq:wjboundbyVar}, \eqref{eq:varwJLowerBound}, \eqref{eq:W_L-to-normal}, \eqref{eq:term-A}, \eqref{eq:term-B}, \eqref{eq:term-C} to obtain
\begin{align*}
    &~~~~d_{W / K} \left( \mathcal{L} ( \widehat{W ^ {\tau(J)}}), \mathcal{N}\left( 0, 1 \right)\right) \\
    & \lesssim \frac{ \sqrt{ \sum_{j \in \tau(J)}^{} \mathbb{E} [ W_j ^ 4 ]
    + a ^ 4 |J| ^ 4 d_n ^ 8 + a ^ 2 |J| ^ 2 d_n ^ 4 \sum_{j \in \tau(J)}^{} \E [W_j ^ 2]} }{ \sum_{j \in \tau(J)}^{} \E [ W_j ^ 2 ] } \\
    & \leq \frac{ \sqrt{ \sum_{j \in \tau(J)}^{} \E [ W_j ^ 4 ] } }{ \sum_{j \in \tau(J)}^{} \E [ W_j ^ 2 ] }
    + \frac{ a ^ 2 |J| d_n ^ 4 }{ \inf_{i \in [n - 1]} \text{Var} \left( g_i(X) \right) }
    + \frac{ a \sqrt{ |J| } d_n ^ 2 }{ \sqrt{ \inf_{i \in [n - 1]} \text{Var} \left( g_i(X) \right) } } \\
    & \lesssim \frac{ \sqrt{ |J| } \inf_{i \in [n - 1]} \text{Var} \left( g_i(X) \right) }{ |J| \inf_{i \in [n - 1]} \text{Var} \left( g_i(X) \right) } + \frac{ 1 }{ \sqrt{ |J| } } \\
    & \lesssim \frac{ 1 }{ \sqrt{ |J| } }.
\end{align*} 
For $J \in \cup_{1 \leq k \leq \left\lfloor n / 2 \right\rfloor }^{} \binom{[n - k]}{k}$,
\begin{align}\label{eq:d_K-W^J-normal}
    d_K \left( \mathcal{L} \left( \frac{W ^ {\tau(J)} - \mu_{W}}{\sigma_{W}} \right), \mathcal{N}\left( \frac{\mu_{\tau (J)} - \mu_{W}}{\sigma_{W}}, \frac{\sigma_{\tau(J)}^2}{\sigma_{W}^2} \right) \right) 
    = d_K \left( \mathcal{L} ( \widehat{W ^ {\tau(J)}} ), \mathcal{N}\left( 0, 1 \right)  \right) 
    \lesssim \frac{1}{\sqrt{ \left| J \right| }}. 
\end{align}
Also, by \eqref{eq:varwJUpperBound} and \eqref{eq:sigmakmeanlowerbound}, we have
\begin{align}\label{eq:d_W-W^J-normal}
    &~~~~d_W \left( \mathcal{L} \left( \frac{W ^ {\tau(J)} - \mu_{W}}{\sigma_{W}} \right), \mathcal{N}\left( \frac{\mu_{\tau (J)} - \mu_{W}}{\sigma_{W}}, \frac{\sigma_{\tau(J)}^2}{\sigma_{W}^2} \right) \right)\notag\\
    & = \frac{\sigma_{\tau(J)}}{\sigma_{W}} d_W \left( \mathcal{L} ( \widehat{W ^ {\tau(J)}} ), \mathcal{N}\left( 0, 1 \right) \right) \notag\\
    & \lesssim \sqrt{ \frac{|J| \inf_{i \in [n - 1]} \text{Var} \left( g_i(X) \right)}{n\inf_{i \in [n - 1]} \text{Var} \left( g_i(X) \right)}}  d_W \left( \mathcal{L}(\widehat{W ^ {\tau (J)}}), \mathcal{N}\left( 0, 1 \right) \right) \notag\\
    & \lesssim \frac{1}{\sqrt{ n }}.
\end{align}
When $J = \varnothing $, we have $d_{W / K} \left( \mathcal{L} \left( \frac{W ^ {\tau (J)} - \mu_{W}}{\sigma_{W}} \right), \mathcal{N}\left( \frac{\mu_{\tau (J)} - \mu_{W}}{\sigma_{W}}, \frac{\sigma_{\tau (J)}^2}{\sigma_{W}^2} \right) \right) = 0$. Then we apply Corollary \ref{coro:lowtail} and inequalities \eqref{eq:d_K-W^J-normal}, \eqref{eq:d_W-W^J-normal} to obtain for $n$ large enough that
\begin{align} \label{eq:d_W/K(W,W1)}
    &\ \ \ \ d_{W / K} \left( \mathcal{L} ( \widehat{W} ), \mathcal{L} ( \widehat{W ^ {(1)}} )  \right) \nonumber \\
    & \leq \E_{J\sim \nu} \left[ d_{W / K} \left( \mathcal{L} \left( \frac{W ^ {\tau (J)} - \mu_{W}}{\sigma_{W}} \right), \mathcal{N}\left( \frac{\mu_{\tau (J)} - \mu_{W}}{\sigma_{W}}, \frac{\sigma_{\tau (J)}^2}{\sigma_{W}^2} \right) \right)   \right] \nonumber \\
    & \lesssim \E_{J\sim \nu} \left[ \frac{ 1 }{ \sqrt{|J|} } \ \middle| J\neq \varnothing \right] \cdot \mathbb{P} \left( J\neq \varnothing \right) \nonumber \\
    & = \mathbb{E} \left[ \frac{ 1 }{ \sqrt{ U } } \ \middle| \ U> 0 \right] \cdot \mathbb{P} \left( U > 0 \right) \nonumber \\
    & \lesssim \frac{ 1 }{ \sqrt{ n } }.
\end{align}
Here, the distribution $\nu$ is given in \eqref{eq:nu} and $U$ is the random variable defined in \eqref{eq:rv-U}.

\textbf { Step 2. } \textbf{Bound} \bm{$d_{W / K} \big( \mathcal{L} ( \widehat{W^{(1)}} ), \mathcal{L} ( \widehat{W^{(2)}} ) \big)$}\textbf{.} 
Recall that $\mathcal{L} \left( W^{(1)} \right)$ is the mixture of normal distributions $\mathcal{N}\big( \mu_{\tau (J)}, \sigma_{\tau (J)} ^ 2 \big)$ with the mixing law $\nu$.
We construct $\mathcal{L}(W^{(2)})$ via replacing the mixing component $\mathcal{N}\big( \mu_{\tau (J)}, \sigma_{\tau (J)} ^ 2 \big)$ by $\mathcal{N}\Big( \mu_{\tau (J)}, \E_{J'\sim \mathcal{U}\left(\binom{[n - |J|]}{|J|}\right)} \big[ \sigma_{\tau (J')}^2 \big]\Big)$, while keeping the mixing law unchanged. 
One can check that $\mu_{W^{(2)}} = \mu_{W^{(1)}} = \mu_W$ and $\sigma_{W^{(2)}} = \sigma_{W^{(1)}} = \mu_W$. Hence, $\mathcal{L} ( \widehat{W^{(2)}} )$ is the mixture of $\mathcal{N}\Bigg( \frac{ \mu_{\tau (J)} - \mu_{W} }{ \sigma_W }, \frac{ \E_{J'\sim \mathcal{U}\left(\binom{[n - |J|]}{|J|}\right)} [ \sigma_{\tau (J')}^2] }{ \sigma_{W}^2 } \Bigg)$ with the mixing law $\nu$.



We first bound $d_{K} \big( \mathcal{L} ( \widehat{W^{(1)}} ), \mathcal{L} ( \widehat{W^{(2)}} ) \big)$. By triangle inequality, Cauchy-Schwarz inequality and Lemma \ref{lem:normaldistancekolwass}, we have 
\begin{align} \label{eq:d_k(w_1,w_2)}
    &\ \ \ d_K \left( \mathcal{L} ( \widehat{W^{(1)}} ), \mathcal{L} ( \widehat{W^{(2)}} ) \right) \notag\\
    & \leq \E_{J\sim \nu} 
    \left[ d_K \left( \mathcal{N}\left( \frac{\mu_{\tau (J)} - \mu_{W}}{\sigma_{W}}, \frac{\sigma_{\tau (J)}^2}{\sigma_{W}^2} \right), \mathcal{N}\left( \frac{\mu_{\tau (J)} - \mu_{W}}{\sigma_{W}}, \frac{\E_{J'\sim \mathcal{U}\left(\binom{[n - |J|]}{|J|}\right)} \big[ \sigma_{\tau (J')}^2 ] }{\sigma_{W}^2} \right) \right) \right] \nonumber  \notag\\
    & \leq \mathbb{E}_{J\sim \nu} \left[ \frac{ \left| \sigma_{\tau (J)}^2 -\E_{J'\sim \mathcal{U}\left(\binom{[n - |J|]}{|J|}\right)} [ \sigma_{\tau (J')}^2 ] \right|}
    {\E_{J'\sim \mathcal{U}\left(\binom{[n - |J|]}{|J|}\right)} [ \sigma_{\tau J')}^2 ] } \ \middle| \ 0 < |J| < \frac{ n }{ 2 }  \right] \cdot
    \mathbb{P} \left( 0 < |J| < \frac{ n }{ 2 } \right) \nonumber \notag\\
    & \leq \mathbb{E}_{k\sim \mathcal{L}(U)} \left[ \frac{\sqrt{ \text{Var}_{J\sim \mathcal{U}\left(\binom{[n - k]}{k}\right)} ( \sigma_{\tau (J)}^2 ) }}
    {\E_{J\sim \mathcal{U}\left(\binom{[n - k]}{k}\right)} [ \sigma_{\tau (J)}^2 ] } \ \middle| \ 0 < k < \frac{ n }{ 2 } \right] \cdot
    \mathbb{P} \left( 0 < k < \frac{ n }{ 2 } \right).
\end{align}
Both cases $|J| = 0$ and $|J| = n/2$ have no contribution, since there is only one set $J$ in each case and the corresponding mixing components are not affected by our averaging operation.

In the following, we apply the Efron-Stein inequality for slices (i.e., Proposition \ref{prop:ConcentrationSlice}) to give an upper bound of $\text{Var}_{J\sim \mathcal{U}\left(\binom{[n - k]}{k}\right)} ( \sigma_{\tau (J)}^2 )$ for $0 < k < \frac{ n }{ 2 }$. We consider the function $h: \binom{[n - k]}{k} \to \R$ defined as $h(J)=\sigma_{\tau(J)}^2$. By equation \eqref{eq:varjdef}, we have
\begin{align} 
    h (J) 
    &= \sum_{i \in \tau (J)}^{} \text{Var} \left( g_i \right) - a (n - 1 - k) (d_n + 1) \sum_{i \in \tau (S)}^{} \text{Cov} \left( g_i, X_i \right) \notag\\
    &\ \ \ \ + a ^ 2 k (n - 1 - k) ^ 2 \frac{(d_n + 1) ^ 2 }{ 4 } \gamma_2 + a ^ 2 \binom{k}{2} \gamma_2 ^ 2 \label{eq:h_S}. 
\end{align} 
Now we bound the difference $|h(J)-h(J')|$ for any $J, J' \in \binom{[n - k]}{k}$ such that $| J \Delta J' | = 2$. Consider the case $J = \{j_1, j_2, \cdots, j_{k}\}$ and $J' = \{j_2, j_3, \cdots, j_{k + 1}\}$ with $j_1<j_2<\cdots<j_{k+1}$. By inequalities \eqref{eq:gi-var-diff} and \eqref{eq:gi-var}, we have
\begin{align}
    &\ \ \ \ \left| \sum_{i \in \tau (J)}^{} \text{Var} \left( g_i \right) -
    \sum_{i \in \tau (J')}^{} \text{Var} \left( g_i \right)  \right| \notag\\
    & \leq \left| \text{Var} \left( g_{j_1} \right) - \text{Var} \left( g_{j_{k + 1} + k - 1} \right) \right| 
    + \sum_{i = 2}^{k} \left| 
        \text{Var} \left( g_{j_i + i - 1} \right) - \text{Var} \left( g_{j_i + i - 2} \right) \right| \notag\\
    & \lesssim \inf_{i \in [n - 1]} \text{Var} \left( g_i \right). \label{eq:h_S-var-diff}
\end{align}
By inequalities \eqref{eq:gi-cov-diff} and \eqref{eq:gi-cov}, we have
\begin{align}
 \left| \sum_{i \in \tau (J)}^{} \text{Cov} \left( g_i, X_i \right)\! -\! \sum_{i \in \tau (J')}^{} \text{Cov} \left( g_i, X_i \right) \right| 
    & \leq \left| \text{Cov} \left( g_{j_1}, X_{j_1} \right) - \text{Cov} \left( g_{j_{k + 1} +k - 1}, X_{j_{k + 1} +k - 1} \right) \right| \notag\\
   &+ \sum_{i = 2}^{k} \left| \text{Cov} \left( g_{j_i + i - 1}, X_{j_i + i - 1} \right) - \text{Cov} \left( g_{j_i + i - 2}, X_{j_i + i - 2} \right) \right| \notag\\
    & \lesssim d_n \sqrt{ \inf_{i \in [n - 1]} \text{Var} \left( g_i \right) }.\label{eq:h_S-covar-diff}
\end{align}
The general case of $J, J' \in \binom{[n - k]}{k}$ such that $| J \Delta J' | = 2$ can be treated in a similar manner and inequalities \eqref{eq:h_S-var-diff} and \eqref{eq:h_S-covar-diff} still hold. 
Plug the bounds \eqref{eq:h_S-var-diff} and \eqref{eq:h_S-covar-diff} into equation \eqref{eq:h_S} and apply Condition \ref{cond:NearIndependence} to obtain
\begin{align*}
    \left| h (J) - h (J') \right| 
    & \leq \left| \sum_{i \in \tau (J)}^{} \text{Var} \left( g_i \right)
    - \sum_{i \in \tau (J')}^{} \text{Var} \left( g_i \right) \right| \\
    &\ \ \ \ + |a| (n - 1 - k) (d_n + 1) \left| \sum_{i \in \tau (J)}^{} \text{Cov} \left( g_i, X_i \right) - \sum_{i \in \tau (J')}^{} \text{Cov} \left( g_i, X_i \right) \right|  \\
    & \lesssim \inf_{i \in [n - 1]} \text{Var} \left( g_i \right)
    + |a| n d_n ^ 2 \sqrt{ \inf_{i \in [n - 1]} \text{Var} \left( g_i \right) }  \\
    & \lesssim  \inf_{i \in [n - 1]} \text{Var} \left( g_i \right).
\end{align*}
By Proposition \ref{prop:ConcentrationSlice}, we have
\begin{align}\label{eq:h_S-var}
    \text{Var}_{J\sim \mathcal{U}\left(\binom{[n - k]}{k}\right)} ( \sigma_{\tau (J)} ^ 2 )
    = \text{Var}_{J\sim \mathcal{U}\left(\binom{[n - k]}{k}\right)} \left( h (J) \right)
    \lesssim \min \left\{ k, n - 2 k \right\} \left( \inf_{i \in [n - 1]} \text{Var} \left( g_i \right) \right)  ^ 2. 
\end{align}
Plug bounds \eqref{eq:varwJLowerBound} and \eqref{eq:h_S-var} into inequality \eqref{eq:d_k(w_1,w_2)}
and apply Corollary \ref{coro:lowtail} to obtain for $n$ large enough that
\begin{align} \label{eq:d_K(W1,W2)}
    d_K \left( \mathcal{L} ( \widehat{W^{(1)}} ), \mathcal{L} ( \widehat{W^{(2)}} ) \right)
    & \leq \E_{k\sim \mathcal{L}(U)} 
    \left[ \frac{\sqrt{ \text{Var}_{J\sim \mathcal{U}\left(\binom{[n - k]}{k}\right)} ( \sigma_{\tau (J)}^2 ) }}
    {\E_{J\sim \mathcal{U}\left(\binom{[n - k]}{k}\right)} [ \sigma_{\tau (J)}^2 ] } \ \middle| \ 0 < k < \frac{ n }{ 2 } \right] \cdot
    \mathbb{P} \left( 0 < k < \frac{ n }{ 2 } \right) \nonumber \\
    & \lesssim  \E_{k\sim \mathcal{L}(U)} 
    \left[ \frac{ \sqrt{ k } \inf_{i \in [n - 1]} \text{Var} \left( g_i \right) }{ k \inf_{i \in [n - 1]} \text{Var} \left( g_i \right) } \ \middle| \ 0<k<\frac{n}{2} \right] \cdot
    \mathbb{P} \left( 0 < k < \frac{ n }{ 2 } \right) \nonumber \\
    & = \E_{k\sim \mathcal{L}(U)} \left[ \frac{ 1 }{ \sqrt{ k } } \ \middle| \ 0<k<\frac{n}{2} \right] \cdot
    \mathbb{P} \left( 0 < k < \frac{ n }{ 2 } \right) \nonumber \\
    & \lesssim \frac{ 1 }{ \sqrt{ n } }.
\end{align}
Similarly, for Wasserstein $W_1$ distance, we apply Lemma \ref{lem:normaldistancekolwass} and inequalities \eqref{eq:varwJLowerBound}, \eqref{eq:varwJUpperBound} and \eqref{eq:sigmakmeanlowerbound} to obtain for large $n$ that
\begin{align} \label{eq:d_W(W1,W2)}
    &~~~~ d_W \left( \mathcal{L} ( \widehat{W^{(1)}} ), \mathcal{L} ( \widehat{W^{(2)}} ) \right) \nonumber \\
    & \leq \E_{J\sim \nu} \left[ d_W \left( \mathcal{N}\left( \frac{\mu_{\tau (J)} - \mu_{W}}{\sigma_{W}}, \frac{\sigma_{\tau (J)}^2}{\sigma_{W}^2} \right), \mathcal{N}\left( \frac{\mu_{\tau (J)} - \mu_{W}}{\sigma_{W}}, \frac{\E_{J'\sim \mathcal{U}\left(\binom{[n - |J|]}{|J|}\right)} \big[ \sigma_{\tau (J')}^2 \big] }{\sigma_{W}^2} \right) \right) \right] \nonumber \\
    & \leq \E_{ J\sim \nu } 
    \left[ \frac{ \left| \sigma_{\tau (J)}^2 - \E_{I'\sim \mathcal{U}\left(\binom{[n - |J|]}{|J|}\right)} \big[ \sigma_{\tau (J')}^2 \big] \right| }{\sigma_{W} \sqrt{ \E_{J'\sim \mathcal{U}\left(\binom{[n - |J|]}{|J|}\right)} \big[ \sigma_{\tau (J')}^2 \big] }  } \ \middle| \ 0 < |J| < \frac{ n }{ 2 } \right] \cdot
    \mathbb{P} \left( 0 < |J| < \frac{ n }{ 2 }  \right) \nonumber \\
    & \leq \frac{1}{\sigma_{W}} \cdot
    \E_{k\sim \mathcal{L}(U)} 
    \left[ \sqrt{ \frac{\text{Var}_{J\sim \mathcal{U}\left(\binom{[n - k]}{k}\right)} ( \sigma_{\tau (J)}^2 ) }{\mathbb{E}_{J\sim \mathcal{U}\left(\binom{[n - k]}{k}\right)} \big[ \sigma_{\tau (J)}^2 \big]} } 
 \ \middle| \ 0 < k < \frac{ n }{ 2 }\right]\cdot 
    \mathbb{P} \left( 0 < k < \frac{ n }{ 2 } \right) \nonumber \\
    & \lesssim \frac{ \mathbb{P} \left( 0 < k < \frac{ n }{ 2 } \right)  }{ \sqrt{ n\inf_{i \in [n - 1]} \text{Var} \left( g_i \right) }} \cdot \E_{k\sim \mathcal{L}(U)}
    \left[ \sqrt{ \frac{ k \left( \inf_{i \in [n - 1]} \text{Var} \left( g_i \right) \right)  ^ 2 }{ k \inf_{i \in [n - 1]} \text{Var} \left( g_i \right) } } \ \middle| \ 0 < k < \frac{ n }{ 2 }\right] \notag\\
    & \lesssim \frac{ 1 }{ \sqrt{ n } }.
\end{align}

\textbf { Step 3.} 
\textbf{Bound} \bm{$d_{W / K} \big(\mathcal{L} (\widehat{W^{(2)}}), \mathcal{L} ( \widehat{W^{(3)}})\big)$} \textbf{.} For $0 \leq k \leq n / 2 $, we write $U_k$ for a random variable whose distribution is the uniform mixture of $\mathcal{N}\Bigg( \frac{\mu_{\tau (J)} - \mu_{W}}{\sigma_{W}}, \frac{\mathbb{E}_{J'\sim \mathcal{U}\left(\binom{[n - k]}{k}\right)  } [ \sigma_{\tau (J')}^2 ]}{\sigma_{W}^2} \Bigg)$ for $J \in \binom{[n - k]}{k}$. Then $\mathcal{L} ( \widehat{W^{(2)}} )$, given in Step 2, is the mixture of $\mathcal{L} \left( U_k \right)$ with the mixing law $k\sim \mathcal{L}(U)$, where the random variable $U$ is given in \eqref{eq:rv-U}. Let $\mu_k$ and $\sigma_k ^ 2$ be the mean and variance of the uniform mixture of $W ^ {\tau (J)}$ with $J \in \binom{[n - k]}{k}$, respectively. We set $\mathcal{L} ( W^{(3)} )$ to be the mixture of  $\mathcal{N}\left( \mu_{k}, \sigma_{k}^2 \right)$ with the mixing law $k\sim \mathcal{L}(U)$. One can check that $\mu_{W^{(3)}} = \mu_{W^{(2)}} = \mu_{W^{(1)}} = \mu_{W}$  and $\sigma_{W^{(3)}} = \sigma_{W^{(2)}} = \sigma_{W^{(1)}} = \sigma_{W}$.

By triangle inequality and Lemma \ref{lem:mixdis}, we obtain
\begin{align} \label{eq:d_K(W2,W3)mediate}
    &\ \ \ \ d_{K} \left( \mathcal{L} ( \widehat{W^{(2)}} ), \mathcal{L} ( \widehat{W^{(3)}} ) \right) \nonumber \\
    & \leq \E_{k\sim \mathcal{L}(U)} \left[ d_{K} \left( \mathcal{L} \left( U_k \right), \mathcal{N}\left( \frac{\mu_{k} - \mu_{W}}{\sigma_{W}}, \frac{\sigma_{k}^2}{\sigma_{W}^2} \right)  \right)  \right]  \nonumber \\
    & = \E_{k\sim \mathcal{L}(U)} \left[ d_K \left( \mathcal{L} ( \widehat{U_k} ), \mathcal{N}\left( 0, 1 \right)  \right) \ \middle| \ 0 < k < \frac{ n }{ 2 } \right] \cdot \mathbb{P} \left( 0 < k < \frac{n}{2} \right) \nonumber \\
    & \lesssim  \E_{k\sim \mathcal{L}(U)} 
    \left[ \frac{ \sigma_{\mu_{\tau (J)}} }{ \sqrt{ \E_{J'\sim \mathcal{U}\left({[n-k] \choose k}\right)} [ \sigma_{\tau (J')} ^ 2 ] }  } d_W \left( \mathcal{L} \left( \widehat{\mu_{\tau (J)}} \right), \mathcal{N}\left( 0, 1 \right)  \right) \ \middle| \ 0 < k < \frac{ n }{ 2 } \right] 
    \cdot\mathbb{P} \left( 0 < k < \frac{n}{2} \right). 
\end{align}
In the identity, we use the fact that $\E[U_k]=(\mu_k-\mu_W)/\sigma_W$ and $\text{Var}(U_k)=\sigma_k^2/\sigma_W^2$. 
Let $J = \{a_1, \cdots, a_k\}$ with $a_1<\cdots<a_k$ be selected uniformly at random from $\binom{[n - k]}{k}$. By equation \eqref{eq:mujdefinition} and the arithmetic property of $\{g_i(x)\}_{i=1}^{n-1}$, we have
\begin{align} \label{eq:mujcompute}
    \mu_{\tau (J)} 
    & = \sum_{i = 1}^{k} \mathbb{E} \left[ g_{a_i + i - 1} \right] 
    + a \binom{n - 1 - k}{2} \frac{ (d_n + 1) ^ 2 }{ 4 } - a \binom{n - 1}{2} \frac{ (d_n + 1) ^ 2 }{ 4 } \nonumber \\
    & = \left( \mathbb{E} \left[ g_2 \right] - \mathbb{E} \left[ g_1 \right] \right) 
    \left( \sum_{i = 1}^{k} a_i + \frac{ k (k - 3) }{ 2 } \right) + k \mathbb{E} \left[ g_1 \right]\notag\\
    &~~~+ a \binom{n - 1 - k}{2} \frac{ (d_n + 1) ^ 2 }{ 4 }  - a \binom{n - 1}{2} \frac{ (d_n + 1) ^ 2 }{ 4 }. 
\end{align}
Define random variable $N_k := \sum_{i = 1}^{k} a_i$. By Lemma \ref{lem:hoeffccltexample}, we have
\begin{align}\label{eq:d_W(mu-tao-j)}
    d_{W} \left( \mathcal{L} \left( \widehat{\mu_{\tau (J)}} \right), \mathcal{N}\left( 0, 1 \right)  \right)
    = d_{W} \left( \mathcal{L} ( \widehat{N_k} ), \mathcal{N}\left( 0, 1 \right)  \right)
    \lesssim \sqrt{ \frac{ n }{ k (n - 2 k) } }.
\end{align}
By \eqref{eq:mujcompute} and the linearity of expectation, one can directly compute that
\begin{align} \label{eq:varofmuj}
    \sigma_{\mu_{\tau(J)}}^2=\text{Var}_{J\sim \mathcal{U}\left({[n-k] \choose k}\right)} \left( \mu_{\tau (J)} \right)
    & = (\E [g_2] - \E [g_1]) ^ 2 \text{Var}_{J\sim \mathcal{U}\left({[n-k] \choose k}\right)} \left( \sum_{i = 1}^{k} a_i \right) \nonumber \\
    & = \frac{ (\E [g_2] - \E [g_1]) ^ 2 }{ 12 } k (n - k + 1) (n - 2 k).
\end{align}
For $n$ large enough, we plug \eqref{eq:varwJLowerBound} , \eqref{eq:d_W(mu-tao-j)} and \eqref{eq:varofmuj} into  \eqref{eq:d_K(W2,W3)mediate} and apply Corollary \ref{coro:lowtail} to obtain

\begin{align} \label{eq:d_K(W2,W3)}
    &\ \ \ d_{K} \left( \mathcal{L} ( \widehat{W^{(2)}} ), \mathcal{L} ( \widehat{W^{(3)}} ) \right) \nonumber \\
    & \lesssim \left| \E [g_2-g_1] \right|\cdot \frac{  }{  } \E_{k\sim \mathcal{L}(U)} \left[ \frac{ \sqrt{ k (n - k + 1) (n - 2 k) } }{ \sqrt{k \inf_{i \in [n - 1]} \text{Var} \left( g_i \right) } } \sqrt{ \frac{ n }{ k (n - 2 k) } } \ \middle| \ 0 < k < \frac{ n }{ 2 }\right] \cdot
    \mathbb{P} \left( 0 < k < \frac{n}{2} \right) \nonumber \\
    & \lesssim \mathbb{E}_{k\sim \mathcal{L}(U)} \left[ \frac{ 1 }{ \sqrt{ k } } \ \middle| \ 0 < k < \frac{n}{2} \right] \cdot
    \mathbb{P} \left( 0 < k < \frac{n}{2} \right) \nonumber \\
    & \lesssim \frac{ 1 }{ \sqrt{ n } }.
\end{align}
Similarly, we apply Lemma \ref{lem:mixdis} and inequalities \eqref{eq:sigmakmeanlowerbound}, \eqref{eq:d_W(mu-tao-j)} and \eqref{eq:varofmuj} to obtain
\begin{align} \label{eq:d_W(W2,W3)}
    &\ \ \ d_W \left( \mathcal{L} ( \widehat{W^{(2)}} ), \mathcal{L} ( \widehat{W^{(3)}} ) \right) \nonumber \\
    & \leq \E_{k\sim \mathcal{L}(U)} \left[ d_W \left( \mathcal{L} \left( U_k \right), \mathcal{N}\left( \frac{ \mu_{k} - \mu_{W} }{ \sigma_{W} }, \frac{ \sigma_{k}^2 }{ \sigma_{W}^2 } \right)  \right)  \right] \nonumber \\
    & = \E_{k\sim \mathcal{L}(U)} \left[ \frac{ \sigma_{k} }{ \sigma_{W} } d_W \left( \mathcal{L} ( \widehat{U_k} ), \mathcal{N}\left( 0, 1 \right) \right) \ \middle| \ 0 < k < \frac{ n }{ 2 }\right] \cdot
    \mathbb{P} \left( 0 < k < \frac{ n }{ 2 } \right) \nonumber \\
    & \leq \E_{k\sim \mathcal{L}(U)} \left[ \frac{ \sigma_{\mu_{\tau (J)}} }{ \sigma_W } d_W \left( \mathcal{L} \left( \widehat{\mu_{\tau (J)}} \right), \mathcal{N}\left( 0, 1 \right)  \right) \ \middle| \ 0 < k < \frac{ n }{ 2 } \right] \cdot \mathbb{P} \left( 0 < k < \frac{ n }{ 2 } \right) \nonumber \\
    & \lesssim \frac{ \sqrt{ n } }{ \sigma_W } \left| \E [g_2] - \E [g_1] \right|\cdot  \E_{k\sim \mathcal{L}(U)} \left[  \frac{ \sqrt{ k (n - k + 1) (n - 2 k) } }{ \sqrt{ k (n - 2 k) } } \right] \cdot\mathbb{P} \left( 0 < k < \frac{ n }{ 2 } \right) \nonumber \\
    & \lesssim \frac{ \sqrt{ \inf_{i \in [n - 1]} \text{Var} \left( g_i \right) } }{ \sqrt{ n  \inf_{i \in [n - 1]} \text{Var} \left( g_i \right) } }  =\frac{ 1 }{ \sqrt{ n } }.
\end{align}

\textbf{Step 4.} \textbf{Bound} \bm{$d_{W / K} \big(\mathcal{L} (\widehat{W^{(3)}}), \mathcal{L} ( \widehat{W^{(4)}})\big)$} \textbf{.}
Recall that $\mathcal{L} ( W^{(3)} )$ is the mixture of normal distributions $\mathcal{N}\left( \mu_{k}, \sigma_{k}^2 \right)$ with the mixing law $k\sim \mathcal{L}(U)$, where $\mu_k$ and $\sigma_k^2$ are the mean and variance of the uniform mixture of $W^{\tau(J)}$, respectively, for $J\in {[n-k]\choose k}$ and the random variable $U$ is defined in \eqref{eq:rv-U}. 
We set $\mathcal{L} \left( W^{(4)} \right)$ to be the mixture of $\mathcal{N}\left( \mu_{k}, \E\left[ \sigma_{U}^2 \right] \right)$ with the mixing law $k\sim \mathcal{L}(U)$. Then one can check that $\mu_{W^{(4)}} = \mu_{W^{(3)}} = \mu_{W^{(2)}} = \mu_{W^{(1)}} = \mu_{W}$  and $\sigma_{W^{(4)}} = \sigma_{W^{(3)}} = \sigma_{W^{(2)}} = \sigma_{W^{(1)}} = \sigma_{W}$. 

By Lemma \ref{lem:normaldistancekolwass}, we have
\begin{align} \label{eq:d_K(w3,w4)}
    d_K \left( \mathcal{L} ( \widehat{W^{(3)}} ), \mathcal{L} ( \widehat{W^{(4)}} ) \right)
    & \leq \E_{k\sim \mathcal{L}(U)} \left[ d_K \left( \mathcal{N}\left( \frac{\mu_{k} - \mu_{W}}{\sigma_{W}}, \frac{\sigma_{k}^2}{\sigma_{W}^2} \right) , \mathcal{N}\left( \frac{\mu_{k} - \mu_{W}}{\sigma_{W}}, \frac{\E [\sigma_{U}^2] }{\sigma_{W}^2} \right)  \right) \right] \nonumber \\
    & \leq \E_{k\sim \mathcal{L}(U)} \left[ \frac{\left|\sigma_{k}^2 - \E[\sigma_{U}^2]\right|}{ \E [\sigma_{U}^2] } \right] \leq \frac{ \sqrt{ \text{Var}(\sigma_U ^ 2) } }{ \E[\sigma_{U}^2] }.
\end{align}
It is known from Corollary \ref{coro:weightPolyaFrequence} that there exist independent Bernoulli random variables $\{Y_i\}_{0\leq i\leq \lfloor n/2 \rfloor}$ such that $U=\sum_{0\leq i\leq \lfloor n/2 \rfloor} Y_i$. Consider $F: \{0, 1\} ^ {\lfloor n/2 \rfloor } \to \R$ defined as
\begin{equation}\label{eq:def-F-a}
    F (y_1, \cdots, y_{\lfloor n/2 \rfloor}) 
    = \sigma_{y_1+\cdots+y_{\lfloor n/2\rfloor}}^2.
\end{equation}
Clearly, the distribution of $F(Y_1, \cdots, Y_{\lfloor n/2\rfloor})$ is identical to that of $\sigma_U ^ 2$. We will give an upper bound of the variance of $F(Y_1, \cdots, Y_{\lfloor n/2\rfloor})$ by the Efron-Stein inequality. By the law of total variance, one can check that
\begin{align}\label{eq:sigma_k^2-diff}
\sigma_k^2-\sigma_{k+1}^2 &=\left[\mathbb{E}_{J \sim \mathcal{U}\left(\binom{[n - k]}{k}\right)} \big[ \sigma_{\tau (J)} ^ 2 \big]-\mathbb{E}_{K \sim \left(\binom{[n - k-1]}{k+1}\right)} \big[ \sigma_{\tau (K)} ^ 2 \big]\right] \notag\\
&~~~+\left[\text{Var}_{J \sim \mathcal{U}\left(\binom{[n - k]}{k}\right)} (\mu_{\tau (J)})-\text{Var}_{K \sim \mathcal{U}\left(\binom{[n - k-1]}{k+1}\right)} (\mu_{\tau (K)})\right].
\end{align}

We first give an upper bound of the second summand. By \eqref{eq:varofmuj} and \eqref{eq:gi-mean-diff}, we have
\begin{align} \label{eq:variancemujdiff}
    &\ \ \ \ \left| \text{Var}_{J \sim \mathcal{U}\left(\binom{[n - k]}{k}\right)} (\mu_{\tau (J)})-\text{Var}_{K \sim \mathcal{U}\left(\binom{[n - k-1]}{k+1}\right)} (\mu_{\tau (K)}) \right|  \nonumber \\
    & = \left| \frac{ (\E [g_2] - \E [g_1]) ^ 2 }{ 12 } 
    \big( k (n - k + 1) (n - 2 k) - (k + 1) (n - k) (n - 2 k - 2) \big) \right|  \nonumber \\
    & = \left| \frac{ (\E [g_2] - \E [g_1]) ^ 2 }{ 12 } 
    \big(- n ^ 2 + 2 n + (6 n - 2) k - 6 k ^ 2\big) \right| \nonumber \\
    & \lesssim (\E [g_2] - \E [g_1]) ^ 2 n ^ 2  \lesssim \inf_{i \in [n - 1]} \text{Var} \left( g_i \right).
\end{align}

Now we give an upper bound of the first summand. For $J = \{j_1, \cdots, j_k\} \subseteq [n - 1]$ with $j_1<\cdots<j_k$ and $j_0$ such that $j_i < j_{0} < j_{i + 1}$ for $1 \leq i \leq k - 1$, we have $\tau (J) = \{j_1, j_2 + 1, \cdots, j_k + k - 1\}$ and $\tau (J \cup_{}^{} \{j_0\}) = \{j_1, \cdots, j_i + i - 1, j_0 + i, j_{i + 1} + i + 1, \cdots, j_k + k\}$. If $\tau (J), \tau (J \cup_{}^{} \{j_0\}) \subseteq [n - 1]$, then by \eqref{eq:varjdef}, \eqref{eq:gi-cov-diff}, \eqref{eq:gi-cov}, \eqref{eq:gi-var-diff} and Condition \ref{cond:NearIndependence}, we have
\begin{align*}
    &\ \ \ \ \left| \sigma_{\tau (J)} ^ 2 - \sigma_{\tau (J \cup \{j_{0}\})} ^ 2 \right| \\
    & \leq \left| \sum_{i \in \tau (J)}^{} \text{Var} \left( g_i \right) - \sum_{i \in \tau (J \cup \{j_0\})}^{} \text{Var} \left( g_i \right)  \right|
    + \left| a \right| (d_n + 1) \left| \sum_{i \in \tau (J \cup \{j_0\})}^{} \text{Cov} \left( g_i, X_i \right) \right| \\
    &\ \ \ \ +  \left| a \right|  (n - 1 - k) (d_n + 1) \left| \sum_{i \in \tau (J)}^{} \text{Cov} \left( g_i, X_i \right) - \sum_{i \in \tau (J \cup \{j_0\})}^{} \text{Cov} \left( g_i, X_i \right) \right|  \\
    &\ \ \ \ +a ^ 2 k \gamma_2 ^ 2 + a ^ 2 \left| k (4 n - 7) - 3 k ^ 2 - (n - 2) ^ 2 \right|  \frac{ (d_n + 1) ^ 2 }{ 4 } \gamma_2  \\
    & \lesssim  \left| \text{Var} \left( g_{j_0 + i} \right)  \right| + \sum_{l = i + 1}^{k} \left| \text{Var} \left( g_{j_l + l} \right) - \text{Var} \left( g_{j_l + l - 1} \right) \right| 
     + \left| a \right| d_n \sum_{i \in \tau (J \cup \{j_0\})}^{} \left| \text{Cov} \left( g_i, X_i \right) \right| + a ^ 2 n ^ 2 d_n ^ 4     \\
    &~~~+ \left| a \right| n d_n \left( \left| \text{Cov} \big( g_{j_0 + i}, X_{j_0 + i} \big) \right| + \sum_{l = i + 1}^{k} \left| \text{Cov} \big( g_{j_l + l}, X_{j_l + l} \big) - \text{Cov} \big( g_{j_l + l - 1}, X_{j_l + l - 1} \big) \right|  \right) \\
    & \lesssim \inf_{i \in [n - 1]} \text{Var} \left( g_i \right)
     + |a| n d_n ^ 2 \sqrt{ \inf_{i \in [n - 1]} \text{Var} \left( g_i \right) } + a ^ 2 n ^ 2 d_n ^ 4\\
    & \lesssim \inf_{i \in [n - 1]} \text{Var} \left( g_i \right).
\end{align*}
Similarly, one can show that the inequality above also holds in the cases $j_0 < j_1$ or $j_0 > j_k$. Hence, we showed that if $\tau (J), \tau (J \cup_{}^{} \{j_0\}) \subseteq [n - 1]$ with $j_{0} \not \in J$, then 
\begin{align}\label{eq:sigma_tau_J^2-diff}
    \left| \sigma_{\tau (J)} ^ 2 - \sigma_{\tau (J \cup \{j_{0}\})} ^ 2 \right| 
    \lesssim \inf_{i \in [n - 1]} \text{Var} \left( g_i \right).
\end{align}
Note that if we first select a set $J$ uniformly at random from $\binom{[n - k - 1]}{k - 1}$ and then select $j_0$ uniformly at random from $[n - k - 1] \setminus J$, then the resulting set $J \cup \{j_0\}$ will be uniformly distributed on $\binom{[n - k - 1]}{k}$. This observation and inequality \eqref{eq:sigma_tau_J^2-diff} yield
\begin{align} \label{eq:expectsigmamujdiff-a}
    &~~~~ \left|  \mathbb{E}_{J \sim \mathcal{U}\left(\binom{[n - k]}{k}\right)} \big[ \sigma_{\tau (J)} ^ 2 \big] - \mathbb{E}_{L \sim \mathcal{U}\left(\binom{[n - k - 1]}{k}\right)} \big[ \sigma_{\tau (L)} ^ 2 \big] \right| \notag\\
    & = \left| \frac{ n - 2 k }{ n - k } \mathbb{E}_{J \sim \mathcal{U}( \binom{[n - k - 1]}{k})} \big[ \sigma_{\tau (J)} ^ 2 \big] 
   \! +\! \frac{ k }{ n - k } \mathbb{E}_{J \sim \mathcal{U}( \binom{[n - k - 1]}{k - 1})} \big[ \sigma_{\tau (J \cup \{n - k\})} ^ 2 \big] 
    \!-\! \mathbb{E}_{L \sim \mathcal{U}( \binom{[n - k - 1]}{k})} \big[ \sigma_{\tau (L)} ^2 \big] \right| \notag\\
    & = \frac{ k }{ n - k } \left| \mathbb{E}_{J \sim \mathcal{U}\left(\binom{[n - k - 1]}{k - 1}\right)} \big[ \sigma_{\tau (J \cup \{n - k\})} ^ 2 \big] 
    - \mathbb{E}_{L \sim \mathcal{U}\left(\binom{[n - k - 1]}{k}\right)} \big[ \sigma_{\tau (L)} ^ 2 \big] \right| \notag\\
\begin{split}
    & = \frac{ k }{ n - k } \left\{ \left| \mathbb{E}_{J \sim \mathcal{U}\left(\binom{[n - k - 1]}{k - 1}\right)} \big[ \sigma_{\tau (J \cup \{n - k\})} ^ 2 \big] - \mathbb{E}_{J \sim \mathcal{U}\left(\binom{[n - k - 1]}{k - 1}\right)} \big[ \sigma_{\tau (J)} ^ 2 \big] \right| \right.\notag\\
    & ~~~~~~~~~~~~~~~~~~~\left. + \left| \mathbb{E}_{J \sim \mathcal{U}\left(\binom{[n - k - 1]}{k - 1}\right)} \big[ \sigma_{\tau (J)} ^ 2 \big] 
    - \mathbb{E}_{L \sim \mathcal{U}\left(\binom{[n - k - 1]}{k}\right)} \big[ \sigma_{\tau (L)} ^ 2 \big] \right| \right\}  \end{split} \notag\\
    & \lesssim \inf_{i \in [n - 1]} \text{Var} \left( g_i \right).
\end{align}
We treat the second absolute difference in the last identity via the observation. By a similar coupling, we can prove that
\begin{equation}\label{eq:expectsigmamujdiff-b}
    \left|  \mathbb{E}_{L \sim \mathcal{U}\left(\binom{[n - k - 1]}{k}\right)} \big[ \sigma_{\tau (L)} ^ 2 \big] - \mathbb{E}_{K \sim \mathcal{U}\left(\binom{[n - k - 1]}{k + 1}\right)} \big[ \sigma_{\tau (K)} ^ 2 \big] \right|
    \lesssim \inf_{i \in [n - 1]} \text{Var} \left( g_i \right).
\end{equation}
Hence, we apply inequalities \eqref{eq:expectsigmamujdiff-a} 
and \eqref{eq:expectsigmamujdiff-b} 
to obtain
\begin{align} \label{eq:expectsigmamujdiff}
    &~~~~\left| \mathbb{E}_{J \sim \mathcal{U}\left(\binom{[n - k]}{k}\right)} \big[ \sigma_{\tau (J)} ^ 2 \big] - \mathbb{E}_{K \sim \mathcal{U}\left(\binom{[n - k - 1]}{k + 1}\right)} \big[ \sigma_{\tau (K)} ^ 2 \big] \right|\notag\\
    & \leq \left|  \mathbb{E}_{J \sim \mathcal{U}\left(\binom{[n - k]}{k}\right)} \big[ \sigma_{\tau (J)} ^ 2 \big] - \mathbb{E}_{L \sim \mathcal{U}\left(\binom{[n - k - 1]}{k}\right)} \big[ \sigma_{\tau (L)} ^ 2 \big] \right|\notag\\
    &~~~+ \left|  \mathbb{E}_{L \sim \mathcal{U}\left(\binom{[n - k - 1]}{k}\right)} \big[ \sigma_{\tau (L)} ^ 2 \big] - \mathbb{E}_{K \sim \mathcal{U}\left(\binom{[n - k - 1]}{k + 1}\right)} \big[ \sigma_{\tau (K)} ^ 2 \big] \right| \nonumber \\
    & \lesssim \inf_{i \in [n - 1]} \text{Var} \left( g_i \right).
\end{align}
Combine inequalities \eqref{eq:sigma_k^2-diff}, \eqref{eq:variancemujdiff} and \eqref{eq:expectsigmamujdiff} to obtain
$$
\left| \sigma_{k} ^ 2 - \sigma_{k + 1} ^ 2 \right|\lesssim \inf_{i \in [n - 1]} \text{Var} \left( g_i \right).
$$
By the definition of $F$ in \eqref{eq:def-F-a}, we have for $y_1, \cdots, y_{i - 1}, y_i, y_i', y_{i + 1}, \cdots, y_{\lfloor n/2 \rfloor } \in \{0, 1\}$ that
\begin{align*}
    \left| F \left( y_1, \cdots, y_i,\cdots, y_{\lfloor n/2 \rfloor } \right) 
    - F \left( y_1, \cdots, y_i', \cdots, y_{\lfloor n/2 \rfloor } \right)  \right| \lesssim \inf_{i \in [n - 1]} \text{Var} \left( g_i \right).
\end{align*}
Then we apply Proposition \ref{prop:BoundDiffConcentra} to obtain 
$$
\text{Var}(\sigma_U^2)=\text{Var}(F(Y_1, \cdots, Y_{\lfloor n/2\rfloor}))\lesssim n\left(\inf_{i \in [n - 1]} \text{Var} ( g_i )\right)^2.
$$
Combine the previous inequality with  Corollary \ref{coro:weightPolyaFrequence}, \eqref{eq:d_K(w3,w4)} and \eqref{eq:sigmakmeanlowerbound} to obtain 
\begin{align} \label{eq:d_K(W3,W4)}
    d_K \left( \mathcal{L} ( \widehat{W^{(3)}} ), \mathcal{L} ( \widehat{W^{(4)}}) \right) 
    \leq \frac{ \sqrt{ \text{Var} \left( \sigma_U ^ 2 \right) }  }{ \E \left[ \sigma_{U}^2 \right] } \lesssim \frac{ \sqrt{ n } \inf_{i \in [n - 1]} \text{Var} \left( g_i \right) }{ n\inf_{i \in [n - 1]} \text{Var} \left( g_i \right) } \lesssim \frac{ 1 }{ \sqrt{ n } }.
\end{align}
Similarly, apply Lemma \ref{lem:normaldistancekolwass} and inequalities \eqref{eq:sigmakmeanlowerbound} and \eqref{eq:d_K(W3,W4)} to obtain
\begin{align} \label{eq:d_W(W3,W4)}
    d_W \left( \mathcal{L} ( \widehat{W ^ {(3)}} ), \mathcal{L} ( \widehat{W ^ {(4)}} ) \right) 
    &\leq \E_{k\sim \mathcal{L}(U)} \left[\frac{\left| \sigma_{k}^2 - \E\left[ \sigma_{U}^2 \right]  \right| }{\sigma_{W} \sqrt{ \E\left[ \sigma_{U}^2 \right]  }}  \right]\notag\\
    &\leq \E_{k\sim \mathcal{L}(U)} \left[ \frac{\left| \sigma_{k}^2 - \E\left[ \sigma_{U}^2 \right]  \right| }{ \E \left[ \sigma_{U}^2 \right]  }  \right]\leq \frac{ \sqrt{ \text{Var} \left( \sigma_U ^ 2 \right) }  }{ \E \left[ \sigma_{U}^2 \right] } \lesssim \frac{1}{\sqrt{ n }}.
\end{align}

\textbf { Step 5. Bounding} $\bm{d_{W / K} \big( \mathcal{L} ( \widehat{W^{(4)}}), \mathcal{N}(0, 1)\big)}$\textbf{.}
We first recall that $\mathcal{L} \left( W^{(4)} \right)$ is the mixture of $\mathcal{N}\left( \mu_{k}, \E\left[ \sigma_{U}^2 \right] \right)$ with the mixing law $k\sim \mathcal{L}(U)$. 
By definition, we have
\begin{align} \label{eq:mukcompute}
    \mu_k &= {n-k\choose k}^{-1}\sum_{J :  |J|=k}\mu_{\tau (J)}\notag\\
    & = (\E [g_2] - \E [g_1]) \left( \frac{ k (n - k + 1) }{ 2 } + \frac{ k (k - 3) }{ 2 } \right) 
    + k \E [g_1]\notag\\  
    &~~~+ a \binom{n - 1 - k}{2} \frac{ (d_n + 1) ^ 2 }{ 4 } 
     - a \binom{n - 1}{2} \frac{ (d_n + 1) ^ 2 }{ 4 } \nonumber \\
    & = \left( \frac{ n - 2 }{ 2 } (\E [g_2] - \E [g_1]) + \E [g_1] - \frac{ a (d_n + 1) ^ 2 (2 n - 3) }{ 8 } \right) k 
    + \frac{ a (d_n + 1) ^ 2 }{ 8 } k ^ 2 .
\end{align}
Write $A = \frac{ n - 2 }{ 2 } (\E [g_2] - \E [g_1]) + \E [g_1] - \frac{ a (d_n + 1) ^ 2 (2 n - 3) }{ 8 }$ and $B = \frac{ a (d_n + 1) ^ 2 }{ 8 }$. Set $V=\mu_U=AU+BU^2$. By Lemma \ref{lem:mixdis}, we have
\begin{align} \label{eq:d_W/K(W4,N)}
    d_{W / K} \left( \mathcal{L} ( \widehat{W^{(4)}} ), \mathcal{N}\left( 0, 1 \right)  \right) 
    & \lesssim \frac{ \sigma_{\mu_U} }{ \sqrt{ \E [ \sigma_U ^ 2 ]  } } d_W \left( \mathcal{L} \left( \widehat{\mu_{U}} \right), \mathcal{N}\left( 0, 1 \right)  \right) \notag\\
     & = \frac{ \sigma_{V} }{ \sqrt{ \E\left[ \sigma_U ^ 2 \right] }  } d_W \left( \mathcal{L} ( \widehat{V} ), \mathcal{N}\left( 0, 1 \right)  \right) ,
\end{align}
As shown in Corollary \ref{coro:weightPolyaFrequence}, there exist independent Bernoulli random variables $Y_1, \cdots, Y_{\lfloor n/2\rfloor }$ such that $U = \sum_{1 \leq i \leq \lfloor n/2 \rfloor } Y_i$. Write $p_i=\E[Y_i]>0$ and $Y_i=Z_i+p_i$. Then $Z_i$ has mean 0 and $|Z_i|\leq 1$. We have the following representation
\begin{align} \label{eq:hoeffdecompOfV}
    V - \E [V] & = A U + B U ^ 2 - \E [V] \nonumber \\
    & = A \sum_{1 \leq i \leq \lfloor n/2 \rfloor } Y_i + B \left( \sum_{1 \leq i \leq \lfloor n/2 \rfloor } Y_i \right)  ^ 2 - \E [V] \nonumber \\
    & = (A + B) \sum_{1 \leq i \leq \lfloor n/2 \rfloor } Y_i + 2 B \sum_{1\leq i < j\leq \lfloor n/2 \rfloor} Y_i Y_j - \E [V] \nonumber \\
    & = \sum_{1 \leq i \leq \lfloor n/2 \rfloor} \left( A + B + 2 B \E [U] - 2 B p_i \right) Z_i + 2 B \sum_{1 \leq i <j\leq \lfloor n/2 \rfloor} Z_i Z_j.
\end{align}
Write $V_i = \left( A + B + 2 B \E [U] - 2 B p_i \right) Z_i$ and $V_{\{i, j\}} = 2 B Z_i Z_j$. Then we have the Hoeffding decomposition of $V - \E [V]$ as
$$
V - \E [V] = \sum_{1 \leq i \leq \lfloor n/2 \rfloor} V_i + \sum_{1 \leq i <j\leq \lfloor n/2 \rfloor}^{} V_{\{i, j\}}.
$$
By \eqref{eq:hoeffdecompOfV}, we have
\begin{align*}
    \text{Var} \left( V \right) 
    = \sum_{1 \leq i \leq \lfloor n/2 \rfloor}^{} \left( A + B + 2 B \E [U] - 2 Bp_i \right) ^ 2 \E [Z_i ^ 2] + 4 B ^ 2 \sum_{1 \leq i <j\leq \lfloor n/2 \rfloor}^{} \E [Z_i ^ 2] \E [Z_j ^ 2] > 0.
\end{align*}
By Proposition \ref{prop:hoePrivault} (the case $d = 2$) , we have
\begin{align}\label{eq:d_W(V, N)}
d_W \left( \mathcal{L} ( \widehat{V} ), \mathcal{N}\left( 0, 1 \right)  \right)\lesssim \frac{\sqrt{D+E+F}}{\text{Var}(V)},
\end{align}
where the quantities $D, E, F$ are listed and estimated below.

\begin{align}\label{eq:term-D}
    D :&= \sum_{0\leq l<i\leq 2}^{} \sum_{\left| J \right| = i - l }^{} \E \left[ \left( \sum_{\left| K \right| = l, K \cap J = \varnothing }^{} \mathbb{E} \left[ (V_{J \cup K}) ^ 2 | \mathcal{F}_J \right] \right) ^ 2 \right] \notag\\
    & = \sum_{i = 1, 2} \sum_{\left| J \right| = i }^{} \mathbb{E} \left[ V_J ^ 4 \right] + \sum_{1\leq j\leq \lfloor n/2\rfloor}
    \mathbb{E} \left[ \left( 4 B ^ 2\sum_{k \neq j}^{}  \E [Z_k ^ 2] Z_j ^ 2 \right) ^ 2 \right] \notag\\
    & = \sum_{1\leq i\leq \lfloor n/2\rfloor} \big( A + B + 2 B \E [U] - 2 B p_i \big) ^ 4 \E \left[ Z_i ^ 4 \right] \notag\\
    &\ \ \ \ + 16 B ^ 4 \sum_{1\leq i < j\leq \lfloor n/2\rfloor} \E \left[ Z_i ^ 4 \right] \E \left[ Z_j ^ 4 \right] + 16 B ^ 4 \sum_{1\leq j\leq \lfloor n/2\rfloor}\E [Z_j ^ 4] \left( \sum_{k \neq j}^{} \E [Z_k ^ 2] \right) ^ 2 \notag\\
    & \lesssim \left( |A| + (n + 1) |B| \right) ^ 2 \sum_{1\leq i\leq \lfloor n/2\rfloor}^{} \left( A + B + 2 B \E [U] - 2 B p_i \right) ^ 2 \E \left[ Z_i ^ 2 \right] \notag\\
    &~~~+ n B ^ 4 \sum_{1\leq i < j\leq \lfloor n/2\rfloor} \E [Z_i ^ 2] \E [Z_j ^ 2] \notag\\
    & \lesssim \left( |A| + (n + 1) |B| \right) ^ 2 \text{Var} \left( V \right).
\end{align}
In the first inequality, we use the fact that $|Z_i|\leq 1$. Also,
\begin{align}\label{eq:term-E}
    E :&= \sum_{1\leq l<i\leq 2}^{} \sum_{\substack{|J_1| = |J_2| = i - l \\ J_1 \cap J_2 = \varnothing} }^{} \mathbb{E} \left[ \left( \sum_{\substack{|K| = l\\ K \cap (J_1 \cup J_2) = \varnothing }}^{} \E [V_{J_1 \cup K } V_{J_2 \cup K} | \mathcal{F}_{J_1 \cup J_2}] \right) ^ 2 \right] \notag\\
    & = \sum_{j_1 \neq j_2}^{}
    \mathbb{E} \left[ \left( 4 B ^ 2\sum_{k \neq j_1, j_2}^{} \E [Z_k ^ 2] Z_{j_1} Z_{j_2} \right) ^ 2 \right] \notag\\
    & = 16 B ^ 4 \sum_{j_1 \neq j_2}^{} \E \left[ Z_{j_1} ^ 2 \right] \E \left[ Z_{j_2} ^ 2 \right] \left( \sum_{k \neq j_1, j_2}^{} \E [Z_k ^ 2] \right) ^ 2 \notag\\
    & \lesssim n ^ 2 B ^ 4 \sum_{i < j}^{} \E [Z_i ^ 2] \E [Z_j ^ 2] \notag\\
    & \lesssim n ^ 2 B ^ 2 \text{Var} \left( V \right).
\end{align}And,
\begin{align}\label{eq:term-F}
    F :&= \sum_{1\leq l<i\leq 2}^{} \sum_{|J| = i - l}^{} \mathbb{E} \left[ \left( \sum_{|K| = l, K \cap J = \varnothing }^{} \E [V_K V_{J \cup K} | \mathcal{F}_J] \right) ^ 2 \right] \notag\\
    & = \sum_{1\leq j\leq \lfloor n/2\rfloor} 
    \mathbb{E} \left[ \left( 2 B\sum_{k \neq j}^{}  Z_j \E [V_k Z_k] \right) ^ 2 \right] \notag\\
    & = 4 B ^ 2 \sum_{1\leq j\leq \lfloor n/2\rfloor}^{} \E \left[ Z_j ^ 2 \right] \left( \sum_{k \neq j}^{} \E [V_k Z_k] \right) ^ 2 \notag\\
    & = 4 B ^ 2 \sum_{1\leq j\leq \lfloor n/2\rfloor}^{} \E \left[ Z_j ^ 2 \right] \left( \sum_{k \neq j} \left( A + B + 2 B \E [U] - 2 B p_k \right) \E [Z_k ^ 2] \right) ^ 2 \notag\\
    & \lesssim n ^ 2 B ^ 2 \sum_{1\leq i\leq \lfloor n/2\rfloor}^{} \left( A + B + 2 B \E [U] - 2 B p_k \right) ^ 2 \left(\E [Z_i ^ 2] \right)^ 2 \notag\\
    & \leq n ^ 2 B ^ 2 \sum_{1\leq i\leq \lfloor n/2\rfloor}^{} \left( A + B + 2 B \E [U] - 2 B p_k \right) ^ 2 \E [Z_i ^ 2] \notag\\
    & \leq n ^ 2 B ^ 2 \text{Var} \left( V \right).
\end{align}
Plug bounds \eqref{eq:term-D}, \eqref{eq:term-E}, \eqref{eq:term-F} into \eqref{eq:d_W(V, N)} and apply Condition \ref{cond:NearIndependence} to obtain
\begin{align} \label{eq:d_W(V, N)-a}
    d_{W} \left( \mathcal{L} ( \widehat{V} ), \mathcal{N}\left( 0, 1 \right)  \right) 
   & \lesssim \frac{ \sqrt{ (|A| + (n + 1) |B|) ^ 2 \text{Var} \left( V \right) }  }{ \text{Var} \left( V \right) } \notag\\
   &=\frac{|A| + (n + 1) |B|}{\sigma_V}\lesssim \frac{ \sqrt{ \inf_{i \in [n - 1]} \text{Var} \left( g_i \right) } }{ \sigma_V }.
\end{align}
We use the estimates that $|A|\lesssim |a|nd_n^2+\sqrt{\inf_{i\in[n-1]}\text{Var}(g_i)}$, which follow from \eqref{eq:gi-mean} and \eqref{eq:gi-mean-diff}, and $|B|\lesssim |a|d_n^2$.
Combine \eqref{eq:d_W(V, N)-a} with \eqref{eq:d_W/K(W4,N)} and \eqref{eq:sigmakmeanlowerbound} to obtain
\begin{align} \label{eq:d_W/K(W4,N)final1}
    d_{W / K} \left( \mathcal{L} \left( \widehat{W ^ {(4)}} \right), \mathcal{N}\left( 0, 1 \right)  \right) 
    \lesssim \sqrt{ \frac{ \inf_{i \in [n - 1]} \text{Var} \left( g_i \right) }{ \E_k \left[ \sigma_k ^ 2 \right]  } }
    \lesssim \frac{ 1 }{ \sqrt{ n } }
\end{align}
for sufficiently large $n$.
Putting together inequalities \eqref{eq:d_W/K(W,W1)}, \eqref{eq:d_K(W1,W2)}, \eqref{eq:d_W(W1,W2)}, \eqref{eq:d_K(W2,W3)}, \eqref{eq:d_W(W2,W3)}, \eqref{eq:d_K(W3,W4)}, \eqref{eq:d_W(W3,W4)}, \eqref{eq:d_W/K(W4,N)final1}, we have
\begin{align*}
    d_{W/K} \left( \mathcal{L} ( \widehat{W} ), \mathcal{N}\left( 0, 1 \right)  \right) \leq \sum_{i=0}^4 d_{W/K} \left( \mathcal{L} ( \widehat{W^{(i)}} ), \mathcal{L} ( \widehat{W^{(i+1)}} )\right)\lesssim \frac{ 1 }{ \sqrt{ n } }.
\end{align*}
This completes the proof of Theorem \ref{thm:normalgeneral}.

\end{proof}

\subsection{Proofs of Theorem \ref{thm:AsymptoticNormalityCorePartition}
and Corollary \ref{cor:1}}

In this subsection, we provide proofs of Theorem \ref{thm:AsymptoticNormalityCorePartition}
and Corollary \ref{cor:1}.

\begin{proof}[Proof of Theorem \ref{thm:AsymptoticNormalityCorePartition}]
We will prove the asymptotic normality of $\{L_{n, i}\}_{i=1}^3$ and $\{S_{n, i}\}_{i=1}^3$ by showing that the corresponding assumptions in Theorem \ref{thm:normalgeneral} are satisfied.

\textbf{Asymptotic normality of $\bm{S_{n, 1}}$ and $\bm{S_{n, 2}}$}. Owing to the representations of $S_{n, 1}$ and $S_{n, 2}$ in Lemma \ref{lem:ls}, we define the function $f: \mathbb{R}^{n-1}\rightarrow \mathbb{R}$ as below
\begin{align*}
    &\ \ \ \ f (x_1, \cdots, x_{n - 1}) \\
    & = \sum_{i = 1}^{n - 1} \left( \frac{ n - 1 }{ 2 } \left( x_i + \frac{ d_n + 1 }{ 2 } \right)  ^ 2 + \left( i - \frac{ n - 1 }{ 2 } \right) \left( x_i + \frac{ d_n + 1 }{ 2 } \right)  \right) \\
    &\ \ \ \ - \sum_{i < j}^{} \left( x_i + \frac{ d_n + 1 }{ 2 } \right) \left( x_j + \frac{ d_n + 1 }{ 2 } \right)  \\
    & = \sum_{i = 1}^{n - 1} \left( \frac{ n - 1 }{ 2 } \left( x_i + \frac{ d_n + 1 }{ 2 } \right) ^ 2 + \left( i - \frac{ n - 1 }{ 2 } - \frac{ (n - 2) (d_n + 1) }{ 2 } \right) \left( x_i + \frac{ d_n + 1 }{ 2 } \right) \right) \\
    &\ \ \ \ - \sum_{i < j}^{} x_i x_j + \frac{ (n - 1) (n - 2) (d_n + 1) ^ 2 }{ 8 }. 
\end{align*}
In this case, we have $a = - 1$ and the function 
$g_i: \mathbb{R}\rightarrow \mathbb{R}$ is given by $$
g_i (x) = \frac{ n - 1 }{ 2 } \left( x + \frac{ d_n + 1 }{ 2 } \right) ^ 2 + \left( i - \frac{ n - 1 }{ 2 } - \frac{ (n - 2) (d_n + 1) }{ 2 } \right) \left( x + \frac{ d_n + 1 }{ 2 } \right).
$$ 
\emph{Assume that $d_n \ge 3$.} Obviously, $g_i$ has root $- \frac{ d_n + 1 }{ 2 }$ and $f \left( - \frac{ d_n + 1 }{ 2 }, \cdots, - \frac{ d_n + 1 }{ 2 } \right) = 0$. It is also clear that for any given $x\in \mathbb{R}$ the sequence $\{g_i(x)\}_{1\leq i \leq n-1}$ is an arithmetic progression. Hence, Condition \ref{cond:ArithmeticProgression} holds. 

Let $X \sim \mathcal{U} \left( \left\{ - \frac{ d_n - 1 }{ 2 }, - \frac{ d_n - 3 }{ 2 }, \cdots, \frac{ d_n - 1 }{ 2 } \right\} \right)$. When $n \ge 2$ and $d_n \ge 3$, one can check for all $i \in [n -1]$ that 
\begin{align*}
    \text{Var} \left( g_i (X) \right)
    & = \text{Var} \left( \frac{ n - 1 }{ 2 } X ^ 2 + \left( i + 1 + \frac{ d_n - n }{ 2 } \right) X \right) \\
    & = \frac{ (n - 1) ^ 2 }{ 4 } \text{Var} \left( X ^ 2 \right)
    + \left( i + 1 + \frac{ d_n - n }{ 2 } \right) ^ 2 \E [X ^ 2] \\
    & \ge \frac{ (n - 1) ^ 2 }{ 4 } \text{Var} \left( X ^ 2 \right) \\
    & = \frac{ (n - 1) ^ 2 (d_n ^ 2 - 1) (d_n ^ 2 - 4) }{ 720 } \\
    & \gtrsim n ^ 2 d_n ^ 4
\end{align*}
Hence, Condition \ref{cond:NearIndependence} is satisfied. 

Moreover, there exists an absolute constant $C$ such that for all $i \in [n - 1]$ and $x \in \left\{ - \frac{ d_n - 1 }{ 2 }, - \frac{ d_n - 3 }{ 2 }, \cdots, \frac{ d_n - 1 }{ 2 } \right\}$ it holds that
\begin{align*}
    g_i ^ 2 (x) 
    & \leq \left( n d_n ^ 2 + \left( \frac{ n - 1 }{ 2 } + (n - 2) d_n\right) d_n \right)  ^ 2  \lesssim n ^ 2 d_n ^ 4  < C \inf_{j \in [n - 1]} \text{Var} \left( g_j \right)
\end{align*}
Hence, Condition \ref{cond:Boundedness} is satisfied. 

One can check for $i \in [n - 1]$ that
\begin{align*}
    \text{Cov} \left( g_i (X), X \right) ^ 2
    & = \left( i + 1 + \frac{ d_n - n }{ 2 } \right) ^ 2 (\E [X ^ 2] )^ 2  \lesssim \E [X ^ 2] n ^ 2 d_n ^ 4 \\
    & \lesssim \frac{ (n - 1) ^ 2 }{ 4 } \text{Var} \left( X ^ 2 \right) \E [X ^ 2].
\end{align*}So there exists a universal constant $C \in (0, 1)$ such that
\begin{align*}
    \text{Cov} \left( g_i (X), X \right) ^ 2
    & = \left( i + 1 + \frac{ d_n - n }{ 2 } \right) ^ 2 (\E [X ^ 2]) ^ 2 \\
    & \leq C \left( \frac{ (n - 1) ^ 2 }{ 4 } \text{Var} \left( X ^ 2 \right) \E [X ^ 2] + \left( i + 1 + \frac{ d_n - n }{ 2 } \right) ^ 2 (\E [X ^ 2]) ^ 2 \right) \\
    & = C \text{Var} \left( X \right) \text{Var} \left( g_i(X) \right).
\end{align*}
Hence, Condition \ref{cond:Nondegenerate} is satisfied. 

We set $W = f \left( X_1, \cdots, X_{n - 1} \right) $, where $\left( X_1 + \frac{ d_n + 1 }{ 2 }, \cdots, X_{n - 1} + \frac{ d_n + 1 }{ 2 } \right) $ is chosen uniformly at random from $\mathcal{SB}_n$. It is clear that $S_{n, 2}$ and $W$ are identically distributed. Hence, by Theorem \ref{thm:normalgeneral}, we have
\begin{align*}
    d_{W / K} \left( \mathcal{L} ( \widehat{S_{n, 2}} ), \mathcal{N}( 0, 1 )  \right)=d_{W / K} \left( \mathcal{L} ( \widehat{W} ), \mathcal{N}\left( 0, 1 \right)  \right) \lesssim \frac{ 1 }{ \sqrt{ n } }.
\end{align*}

Similar to Step 1 
in the proof of Theorem \ref{thm:normalgeneral}, we can prove that when $d_n \ge 1$ and $n$ sufficiently large,
\begin{align*}
    d_{W / K} \left( \mathcal{L} ( \widehat{S_{n, 1}} ), \mathcal{N}\left( 0, 1 \right)  \right) \lesssim \frac{ 1 }{ \sqrt{ n } }.
\end{align*}

\textbf{Asymptotic normality of $\bm{L_{n, 1}}$ and $\bm{L_{n, 2}}$}. In this case, we define $f (x_1, \cdots, x_{n - 1}) = \sum_{i = 1}^{n - 1} \left( x_i + \frac{ d_n + 1 }{ 2 } \right)$. Hence, we have $a = 0$ and $g_i (x) = x + \frac{ d_n + 1 }{ 2 }$ for $1 \leq i \leq n - 1$. In this case, one can easily check that all the assumptions in Theorem \ref{thm:normalgeneral} are satisfied. When $d_n \ge 1$, it is a classical result that
\begin{align*}
    d_{W / K} \left( \mathcal{L} ( \widehat{L_{n, 1}} ), \mathcal{N}\left( 0, 1 \right)  \right) \lesssim \frac{ 1 }{ \sqrt{ n } }.
\end{align*}
When $d_n \ge 2$, Theorem \ref{thm:normalgeneral} yields
\begin{align*}
    d_{W / K} \left( \mathcal{L} ( \widehat{L_{n, 2}} ), \mathcal{N}\left( 0, 1 \right)  \right) \lesssim \frac{ 1 }{ \sqrt{ n } }.
\end{align*}
When $d_n = 1$, we have $L_{n, 2} = U$, whose distribution is given in Section \ref{sec:PolyaFrequence}. See the definition of $U$ in Section \ref{sec:PolyaFrequence}. It was showed in \cite{KST20} that
\begin{align}
    \text{Var} \left( U \right) = \frac{ \sqrt{ 5 } }{ 25 } n + O (1).
\end{align}
Hence, when $d_n = 1$, we apply Corollary \ref{coro:weightPolyaFrequence} and Proposition \ref{prop:hoePrivault} to obtain
\begin{align*}
    d_{W / K} \left( \mathcal{L} ( \widehat{L_{n, 2}} ), \mathcal{N}\left( 0, 1 \right)  \right) \lesssim \frac{ 1 }{ \sqrt{ n } }.
\end{align*}

\textbf{Asymptotic normality of $\bm{L_{n, 3}}$ and $\bm{S_{n, 3}}$}. The characterization of $S_{n, 3}$ is given in Lemma \ref{lem:SelfConjugateDistribution}. We assume that $(X_1, \cdots, X_n)$ is uniformly distributed on $\mathcal{MD}_n$ and note that $\{(X_i, X_{n + 1 - i})\}_{1 \leq i < \frac{ n + 1 }{ 2 }}$ are independent. For $1 \leq i \leq n, i \neq \frac{ n + 1 }{ 2 }$, we define $h_i (x) = n x ^ 2 + (2 i - n - 1) x$ and $S_{n, 3, i} = h_i (X_i) + h_{n + 1 - i} (X_{n + 1 - i}) -\E [h_i] - \E [h_{n + 1 - i}]$. Then we have
\begin{align*}
    S_{n, 3} - \E [S_{n, 3}] = \sum_{1 \leq i < \frac{ n + 1 }{ 2 }}^{} S_{n, 3, i},
\end{align*}
which is the Hoeffding decomposition of $S_{n, 3} - \E [S_{n, 3}]$. For $1 \leq i \leq n, i \neq \frac{ n + 1 }{ 2 }$, we have $|h_i (x)| \leq 3 n e_n ^ 2$, where $e_n=E_n/(2n)$ and $E_n$ is the perimeter of a random self-conjugate $n$-core partition (see Section \ref{subsec:self-conjugate}). By Proposition \ref{prop:hoePrivault}, we have
\begin{align} \label{eq:SelfConjuNormalDistance}
    d_{W / K} \left( \mathcal{L} ( \widehat{S_{n, 3}} ), \mathcal{N}\left( 0, 1 \right)  \right) \lesssim 
    \frac{ \sqrt{ \sum_{1 \leq i < \frac{ n + 1 }{ 2 }}^{} \mathbb{E} \big[ S_{n, 3, i} ^ 4 \big] } }{ \sum_{1 \leq i < \frac{ n + 1 }{ 2 }}^{} \mathbb{E} \big[ S_{n, 3, i} ^ 2 \big] }
    \lesssim \frac{ n e_n ^ 2 }{ \sqrt{ \sum_{1 \leq i < \frac{ n + 1 }{ 2 }}^{} \mathbb{E} \big[ S_{n, 3, i} ^ 2 \big] } }.
\end{align}
Now we give a lower bound of $\mathbb{E} \big[ S_{n, 3, i} ^ 2 \big] = \text{Var} \left( h_i + h_{n + 1 - i} \right)$ for $1 \leq i < \frac{ n + 1 }{ 2 }$. By the definition of $\mathcal{MD}_n$, $(X_i, X_{n + 1 - i})$ is chosen uniformly at random from $\{(0, 0)\} \cup_{}^{} \{(0, 1), \cdots, (0, e_n)\} \cup_{}^{} \{(1, 0), \cdots, (e_n, 0)\}$. We can compute that when $e_n \ge 2$,
\begin{align*}
    \text{Var}_{(X_i, X_{n + 1 - i}) \sim \mathcal{U} \left( \{(0, 1), \cdots, (0, e_n)\} \right) } \left( h_i + h_{n + 1 - i} \right) 
    & = \text{Var}_{X_{n + 1 - i} \sim \mathcal{U} \left( [e_n] \right) } \left( h_{n + 1 - i} \right) \\
    & \ge n ^ 2 e_n ^ 2 \text{Var} \left( X_{n + 1 - i} \right) \\
    & \gtrsim n ^ 2 e_n ^ 4.
\end{align*}
Similarly, we have
$$
\text{Var}_{(X_i, X_{n + 1 - i}) \sim \mathcal{U} \left( \{(1, 0), \cdots, (e_n, 0)\} \right) } \left( h_i + h_{n + 1 - i} \right) \gtrsim n ^ 2 e_n ^ 4.
$$
By the law of total variance, when $e_n \ge 2$,
\begin{align*}
    \mathbb{E} \big[ S_{n, 3, i} ^ 2 \big] &=\text{Var} \left( h_i + h_{n + 1 - i} \right)\notag\\
    & \ge \frac{ e_n }{ 2 e_n + 1 }
    \left(  \text{Var}_{X_{n + 1 - i} \sim \mathcal{U} \left( [e_n] \right) } \left( h_{n + 1 - i} \right)
    +  \text{Var}_{X_{i} \sim \mathcal{U} \left( [e_n] \right) } \left( h_{i} \right) \right) \gtrsim n ^ 2 e_n ^ 4.
\end{align*}
This is also true when $e_n = 1$. Plugging this bound into \eqref{eq:SelfConjuNormalDistance}, we obtain
\begin{align*}
    d_{W / K} \left( \mathcal{L} ( \widehat{S_{n, 3}} ), \mathcal{N}\left( 0, 1 \right)  \right) \lesssim \frac{ 1 }{ \sqrt{ n } }.
\end{align*}
Similarly, we can prove that when $e_n \ge 1$,
\begin{align*}
    d_{W / K} \left( \mathcal{L} ( \widehat{L_{n, 3}} ), \mathcal{N}\left( 0, 1 \right)  \right) \lesssim \frac{ 1 }{ \sqrt{ n } }.
\end{align*}
\end{proof}


\begin{proof}[Proof of Corollary \ref{cor:1}]

Let $D_n = dn$, then $d_n = d$ is a constant. It's obvious that $\lambda$ is a strict $(n, dn + 1)$-core partition if and only if $\lambda$ is a strict $n$-core partition with perimeter at most $dn$.
\end{proof}

\subsection{Proof of Theorem \ref{thm:concengeneral}}

In this subsection, we provide the proof of Theorem \ref{thm:concengeneral}.
We first introduce some notations. For a nonempty set $J \subseteq [n - 1]$, we define the set $J ^ {(n)}$ as
$$
J ^ {(n)} = \big\{(x_1, \cdots, x_{n - 1}) \ : \ x_i = 0 \ \text{for} \ i \not \in J \ \text{and} \ x_i \in [d_n] \ \text{for} \ i \in J \big\}.
$$
It is clear that $\left\{J^{(n)}: J\subseteq [n-1]\right\}$ are disjoint. For $0 \leq k \leq \lfloor n/2 \rfloor$, we define 
\begin{equation}\label{eq:sb_n_k}
\mathcal{SB}_{n, k} = \cup_{J \in \binom{[n - k]}{k}}^{} \tau (J) ^ {(n)}.
\end{equation}
Then we have the partition of $\mathcal{SB}_{n}$ as follows
\begin{equation}\label{eq:sb_n}
\mathcal{SB}_{n}=\cup_{0\leq k\leq \lfloor n/2\rfloor} \mathcal{SB}_{n, k}.
\end{equation}
Let $\mu_n$ and $\mu_{n, k}$ be the uniform distributions on $\mathcal{SB}_{n}$ and $\mathcal{SB}_{n, k}$, respectively. For $J \in \binom{[n - k]}{k}$, we denote by $\mu_{n, k, \tau(J)}$ the uniform distributions on $\tau(J)^{(n)}$. It can be seen from \eqref{eq:sb_n} and \eqref{eq:sb_n_k} that $\mu_n$ is the mixture of $\mu_{n, k}$ with weight ${n-k\choose k}d_n^k/|\mathcal{SB}_n|$ and that $\mu_{n, k}$ is the uniform mixture of $\mu_{n, k, \tau(J)}$.

Let $f$ be a function satisfying both assumptions in  Theorem \ref{thm:concengeneral}. Let $X, X_k$ and $X_{k, \tau(J)}$ be random variables with distributions $\mu_n, \mu_{n, k}$ and $ \mu_{n, k, \tau(J)}$, respectively. For $\lambda\in \mathbb{R}$, we define the logarithmic moment generating functions
\begin{align}
\alpha_n (\lambda) &= \log_{} \E\left[ e ^ {\lambda \left( f (X) - \E [f (X)] \right) } \right], \label{eq:alpha_n}\\ 
\alpha_{n, k} (\lambda) &= \log_{} \E\left[ e ^ {\lambda \left( f (X_k) - \E [f (X_k)] \right) } \right], \label{eq:alpha_n_k}\\
\alpha_{n, k, \tau(J)} (\lambda) &= \log_{} \E\left[ e ^ {\lambda \left( f (X_{k, \tau(J)}) - \E [f (X_{k, \tau(J)})] \right) } \right] \label{eq:alpha_n_k_J}.
\end{align}

Our proof of Theorem \ref{thm:concengeneral} proceeds as follows. Owing to Proposition \ref{prop:subgaudef}, the concentration inequality \eqref{eq:concensbn} follows from an upper bound of $\alpha_n (\lambda)$, which can be established, via Lemma \ref{lem:mixcon}, by upper-bounding $\alpha_{n, k} (\lambda)$, since the distribution of $f(X)$ can be written as a mixture of that of $f(X_{k})$. Similarly, the upper bound of $\alpha_{n, k} (\lambda)$ follows, again via Lemma \ref{lem:mixcon}, from estimates of $\alpha_{n, k, \tau(J)} (\lambda)$, since the distribution of $f(X_{k})$ is the uniform mixture of that of $f(X_{k, \tau(J)})$.

\begin{proof}[Proof of Theorem \ref{thm:concengeneral}] In the following, we sequentially establish upper bounds of $\alpha_{n, k, \tau(J)} (\lambda)$, $ \alpha_{n, k} (\lambda)$ and $ \alpha_n (\lambda)$.

\textbf{Upper bound of} $\bm{\alpha_{n, k, \tau(J)} (\lambda)}$. For any $0< k\leq \lfloor n/2\rfloor$ and $J\in {[n-k] \choose k}$, we apply
Proposition \ref{prop:boundiff} and Condition \ref{cond:BoundDiff} to obtain for any $r\geq 0$ that
\begin{equation*}
    \mathbb{P} \left( \left| f (X_{k, \tau(J)}) - \E [f (X_{k, \tau(J)})] \right| \ge r \right)
    \leq 2 \exp \left( - \frac{ 2 r ^ 2 }{ k C_1 ^ 2 } \right).
\end{equation*}
Together with Proposition \ref{prop:subgaudef}, the deviation inequality above implies for all $\lambda \in \R$ that
\begin{equation}\label{eq:alpha_n_k_tau(J)-bound}
    \alpha_{n, k, \tau(J)} (\lambda) 
    \leq 2 k C_1 ^ 2 \lambda ^ 2.
\end{equation}

\textbf{Upper bound of} $\bm{\alpha_{n, k} (\lambda)}$.  For $0 < k < n/2$, as mentioned before, the distribution of $f(X_{k})$ is the uniform mixture of that of $f(X_{k, \tau(J)})$ for $J\in {[n-k] \choose k}$. It follows from Lemma \ref{lem:mixcon} and \eqref{eq:alpha_n_k_tau(J)-bound} that  
\begin{align}\label{eq:alpha_n_k-bound}
    \alpha_{n, k} (\lambda)
     &\leq \sup_{J \in \binom{[n - k]}{k}} \alpha_{n, k, \tau (J)} (\lambda)
    + \log_{} \mathbb{E}_{J\sim \mathcal{U}\big({[n-k] \choose k}\big)} 
    \left[ e ^ {\lambda \big( \E [f(X_{k, \tau(J)})] - \E [f(X_k)] \big)} \right]\notag\\
     & \leq 2 k C_1 ^ 2 \lambda ^ 2
     +\log_{} \mathbb{E}_{J\sim \mathcal{U}\bigl({[n-k] \choose k}\bigr)} 
    \left[ e ^ {\lambda \big( \E [f(X_{k, \tau(J)})] - \E [f(X_k)] \big)} \right].
\end{align}
Note that $\E_{J\sim \mathcal{U}\big({[n-k] \choose k}\big)}[\E [f(X_{k, \tau(J)})]]=\E[f(X_k)]$. Hence, the second summand is the logarithmic moment generating function of $g(J):=\E [f(X_{k, \tau(J)})]$ with $J\sim \mathcal{U}\big({[n-k] \choose k}\big)$. Let $X_{k, J}$ be a random variable uniformly distributed on $J^{(n)}$. 
For any $J, J' \in \binom{[n - k]}{k}$ such that $\left| J \Delta J' \right| = 2$, it follows from Conditions \ref{cond:BoundDiff} and \ref{cond:QuasiAP} that
\begin{align*}
    \left| g (J) - g (J') \right| 
    & \leq \left| \big( \E[f(X_{k, J})] - \E [f(X_{k, \tau(J)})] \big) 
    - \big( \E[f(X_{k, J'})] - \E [f(X_{k, \tau(J')})] \big) \right|  \\
    &~~~~+\left| \E[f(X_{k, J})] - \E[f(X_{k, J'})] \right| \\
    & \leq 2 C_1 + C_2.
\end{align*}
Hence, by Proposition \ref{prop:ConcentrationSlice}, we have for $r > 0$ that
\begin{align*}
    \mathbb{P} \left( \left| g (J) - \E_{J\sim \mathcal{U}\big({[n-k] \choose k}\big)} [g (J)] \right| \ge r \right)
    \leq 2 \exp \left( - \frac{ r ^ 2 }{ \min \left\{ k, n - 2 k \right\} (2 C_1 + C_2) ^ 2 } \right).
\end{align*}
Then we apply Proposition \ref{prop:subgaudef} to obtain for $\lambda \in \R$, 
\begin{equation*}
     \log_{} \mathbb{E}_{J\sim \mathcal{U}\big({[n-k] \choose k}\big)} 
    \left[ e ^ {\lambda \left( g(J) - \E_{J\sim \mathcal{U}\big({[n-k] \choose k}\big)} [g(J)] \right)} \right] \leq 4 \min \left\{ k, n - 2 k \right\} (2 C_1 + C_2) ^ 2 \lambda ^ 2.
\end{equation*}
Plugging this estimate into \eqref{eq:alpha_n_k-bound} to obtain that for $0 < k < n / 2$, we have
\begin{align}\label{eq:alpha_n_k-bound-1}
    \alpha_{n, k} (\lambda)
 \leq 2 k C_1 ^ 2 \lambda ^ 2
 + 4 \min \left\{ k, n - 2 k \right\} (2 C_1 + C_2) ^ 2 \lambda ^ 2.
\end{align}
It is not hard to see that $\alpha_{n, 0}(\lambda)=0$. If $n$ is even and $k = n / 2$, we have $\alpha_{n, \lfloor n/2\rfloor}(\lambda)\leq nC_1^2\lambda^2$. Hence, inequality \eqref{eq:alpha_n_k-bound-1} holds for all $0 \leq k \leq \lfloor n / 2 \rfloor$.

\textbf{Upper bound of} $\bm{\alpha_{n} (\lambda)}$. We denote by $\eta$ the distribution on $\{0, 1, \cdots, \lfloor n/2 \rfloor\}$ such that $\eta(k)={n-k\choose k}d_n^k/|\mathcal{SB}_n|$. As mentioned before, the distribution of $f(X)$ is the mixture of that of $f(X_{k})$ with weight $\eta(k)$. It follows from Lemma \ref{lem:mixcon} and \eqref{eq:alpha_n_k-bound-1} that  
\begin{align}\label{eq:alpha_n-bound}
    \alpha_{n} (\lambda)
     &\leq \sup_{0\leq k\leq \lfloor n/2 \rfloor} \alpha_{n, k} (\lambda)
    + \log_{} \mathbb{E}_{k\sim \eta} 
    \left[ e ^ {\lambda \big( \E [f(X_{k})] - \E [f(X)] \big)} \right]\notag\\
     & \leq 2 k C_1 ^ 2 \lambda ^ 2
 + 4 \min \left\{ k, n - 2 k \right\} (2 C_1 + C_2) ^ 2 \lambda ^ 2 +\log_{} \mathbb{E}_{k\sim \eta} 
    \left[ e ^ {\lambda \big( \E [f(X_{k})] - \E [f(X)] \big)} \right].
    \end{align}
The last summand is the logarithmic moment generating function of $\E [f(X_{k})]$ with $k\sim \eta$. It is known from Corollary \ref{coro:weightPolyaFrequence} that there exist independent Bernoulli random variables $\{Y_i\}_{0\leq i\leq \lfloor n/2 \rfloor}$ such that $\sum_{0\leq i\leq \lfloor n/2 \rfloor} Y_i\sim \eta$. Consider $F: \{0, 1\} ^ {\lfloor n/2 \rfloor } \to \R$ defined as
\begin{equation}\label{eq:def-F-b}
    F (y_1, \cdots, y_{\lfloor n/2 \rfloor}) 
    = \E [f (X_{y_1+\cdots+y_{\lfloor n/2\rfloor}})].
\end{equation}
Clearly, the distribution of $F(Y_1, \cdots, Y_{\lfloor n/2\rfloor})$ is identical to that of $\E [f(X_{k})]$ with $k\sim\eta$. We will upper-bound the logarithmic moment generating function of  $F(Y_1, \cdots, Y_{\lfloor n/2\rfloor})$. 

Recall that  $X_{k, \tau(J)}$ is uniformly distributed on $\tau(J)^{(n)}$ and $X_k$ is the uniform mixture of $X_{k, \tau(J)}$ for $J\sim \mathcal{U}\big({[n-k] \choose k}\big)$. 
For any $0 \leq k < \lfloor n/2 \rfloor$, we have
\begin{align} \label{eq:bound-diff-f}
    &~~~~\left| \E\left[f (X_{k+1})]  - \E[f(X_k)\right] \right| \notag\\
    & = \left| \E_{I \sim \mathcal{U}\big({[n-k-1] \choose k+1}\big)} \left[ \E [f(X_{k+1, \tau(I)})] \right] 
    - \E_{J \sim \mathcal{U}\big({[n-k] \choose k}\big)} \left[ \E [f(X_{k, \tau(J)})] \right] \right| \nonumber \\
    & \leq \left| \E_{I \sim \mathcal{U}\big({[n-k-1] \choose k+1}\big)} \left[ \E [f(X_{k+1, \tau(I)})] \right]
    - \E_{K \sim \mathcal{U}\big({[n-k-1] \choose k}\big)} \left[ \E [f(X_{k, \tau(K)})] \right] \right| \nonumber \\
    &~~~~ + \left| \E_{K \sim \mathcal{U}\big({[n-k-1] \choose k}\big)} \left[ \E [f(X_{k, \tau(K)})] \right]
    - \E_{J \sim\mathcal{U}\big({[n-k] \choose k}\big)} \left[ \E [f(X_{k, \tau(J)})] \right] \right|. 
\end{align}

We now establish the upper bound of the first summand. Given $I=\{i_1, \cdots, i_{k+1}\}\in {[n-k-1] \choose k+1}$. Here, we assume that $i_1<\cdots<i_k<i_{k+1}$. Then we define $I'=I\backslash \{i_{k+1}\}$. By Condition \ref{cond:BoundDiff}, we have
\begin{equation}\label{eq:bound-diff-f-term1-a}
\left| \E_{I \sim \mathcal{U}\big({[n-k-1] \choose k+1}\big)} \left[ \E [f(X_{k+1, \tau(I)})]-\E [f(X_{k, \tau(I')})] \right] \right|\leq C_1.
\end{equation}
Note that the random set $K\sim \mathcal{U}\big({n-k-1 \choose k}\big)$ can be obtained from the random set $I\sim \mathcal{U}\left({[n-k-1] \choose k+1}\right)$ by randomly deleting an element of $I$. Hence, for $I'$ and $K$ obtained from a given $I$, we have $|I'\Delta K|\leq 2$. Then we apply Conditions \ref{cond:BoundDiff} and  \ref{cond:QuasiAP} to obtain that
\begin{align}\label{eq:bound-diff-f-term1-b}
&~~~\left| \E_{I \sim \mathcal{U}\big({[n-k-1] \choose k+1}\big)} \left[ \E [f(X_{k, \tau(I')})] \right]
    - \E_{K \sim \mathcal{U}\big({[n-k-1] \choose k}\big)} \left[ \E [f(X_{k, \tau(K)})] \right] \right| \notag\\
    &\leq \left| \E_{I \sim \mathcal{U}\big({[n-k-1] \choose k+1}\big)} \left[ \E [f(X_{k, I'})] \right]
    - \E_{K \sim \mathcal{U}\big({[n-k-1] \choose k}\big)} \left[ \E [f(X_{k, K})] \right] \right|\notag\\
    &~~~+\left| \E_{I \sim \mathcal{U}\big({[n-k-1] \choose k+1}\big)} \Big[\left(\E [f(X_{k, I'})]-\E [f(X_{k, \tau(I')})]\right)+ \left(\E [f(X_{k, K})]-\E [f(X_{k, \tau(K)})]\right)\Big] \right|\notag\\
    &\leq 2C_1+C_2.
\end{align}
Putting together \eqref{eq:bound-diff-f-term1-a} and \eqref{eq:bound-diff-f-term1-b}, we have
\begin{align}\label{eq:bound-diff-f-term1}
    \left| \E_{I \sim \mathcal{U}\big({[n-k-1] \choose k+1}\big)} \left[ \E [f(X_{k+1, \tau(I)})] \right]
    - \E_{K \sim \mathcal{U}\big({[n-k-1] \choose k}\big)} \left[ \E [f(X_{k, \tau(K)})] \right] \right|
    \leq 3 C_1 + C_2. 
\end{align}

Now we establish the upper bound of the second summand. Note that the distribution $\mathcal{U}\big({[n-k] \choose k}\big)$ is the mixture of $\mathcal{U}\big({[n-k-1] \choose k}\big)$ and the distribution of $L\cup\{n-k\}$ where $L\sim \mathcal{U}\big({[n-k-1] \choose k-1}\big)$ with weights ${n-k-1 \choose k}/{n-k\choose k}=(n-2k)/(n-k)$ and ${n-k-1 \choose k-1}/{n-k\choose k}=k/(n-k)$, respectively. Also, a random set $L \sim \mathcal{U} \left( \binom{[n - k - 1]}{k - 1} \right)$ can be obtained from a random set $K \sim \mathcal{U} \left( \binom{[n - k - 1]}{k} \right)$ via deleting a random element of $K$. These observations and Conditions \ref{cond:BoundDiff} and \ref{cond:QuasiAP} yield
\begin{align} \label{eq:bound-diff-f-term2}
    &\ \ \ \left| \E_{K \sim \mathcal{U}\big({[n-k-1] \choose k}\big)} \left[ \E [f(X_{k, \tau(K)})] \right]
    - \E_{J \sim\mathcal{U}\big({[n-k] \choose k}\big)} \left[ \E [f(X_{k, \tau(J)})] \right] \right| \notag\\
    & = \left| \E_{K \sim \mathcal{U}\big({[n-k-1] \choose k}\big)} \left[ \E [f(X_{k, \tau(K)})] \right]
    - \frac{ n-2k }{ n-k } \cdot
    \E_{J \sim\mathcal{U}\big({[n-k-1] \choose k}\big)} \left[ \E [f(X_{k, \tau(J)})] \right] \right. \notag\\
    &~~~~~~~~~~\qquad \qquad \qquad \qquad \qquad \qquad 
    \left. - \frac{ k }{ n-k } 
    \cdot\E_{L \sim \mathcal{U}\big({[n-k-1] \choose k-1}\big)} \left[ \E [f(X_{k, \tau(L\cup\{n-k\})})] \right]  \right| \notag\\  
    & = \frac{ k }{ n - k }
    \left| \E_{K \sim \mathcal{U}\big({[n-k-1] \choose k}\big)} \left[ \E [f(X_{k, \tau(K)})] \right]
    - \E_{L \sim \mathcal{U}\big({[n-k-1] \choose k-1}\big)} \left[ \E [f(X_{k, \tau(L\cup\{n-k\})})] \right]  \right| \notag\\
    & \leq \frac{ k }{ n - k }
    \left| \E_{K \sim \mathcal{U}\big({[n-k-1] \choose k}\big)} \left[ \E [f(X_{k, K})] \right]
    - \E_{L \sim \mathcal{U}\big({[n-k-1] \choose k-1}\big)} \left[ \E [f(X_{k, L\cup\{n-k\}})] \right] \right| \notag\\
    &\ \ \ + \frac{ k }{ n - k }
    \left| \E_{K \sim \mathcal{U}\big({[n-k-1] \choose k}\big)} \left(\left[ \E [f(X_{k, K})] \right] - \left[ \E [f(X_{k, \tau(K)})] \right]\right) \right. \notag\\
    & \qquad \qquad~
    \left. - \E_{L \sim \mathcal{U}\big({[n-k-1] \choose k-1}\big)} \left(\left[ \E [f(X_{k, L\cup\{n-k\}})] \right] - \left[ \E [f(X_{k, \tau(L\cup\{n-k\})})] \right]\right)\right| \notag\\ 
    & \leq 2 C_1 + C_2.
\end{align}
Putting together \eqref{eq:bound-diff-f}, \eqref{eq:bound-diff-f-term1} and \eqref{eq:bound-diff-f-term2}, we have
\begin{align*}
    \left| \E\left[f (X_{k+1})]  - \E[f(X_k)\right] \right|
    \leq 5 C_1 + 2 C_2.
\end{align*}
By the definition of $F$ in \eqref{eq:def-F-b}, we have for $y_1, \cdots, y_{i - 1}, y_i, y_i', y_{i + 1}, \cdots, y_{\lfloor n/2 \rfloor } \in \{0, 1\}$ that
\begin{align*}
    \left| F \left( y_1, \cdots, y_i,\cdots, y_{\lfloor n/2 \rfloor } \right) 
    - F \left( y_1, \cdots, y_i', \cdots, y_{\lfloor n/2 \rfloor } \right)  \right| 
    \leq 5 C_1 + 2 C_2.
\end{align*}
By Proposition \ref{prop:boundiff}, for $r > 0$,
\begin{align*}
    \mathbb{P} \left( \left| F \left( Y_1, \cdots, Y_{\lfloor n/2 \rfloor} \right) 
    - \E \left[ F \left( Y_1, \cdots, Y_{\lfloor n/2 \rfloor} \right) \right]  \right| \ge r \right)
    \leq 2 \exp \left( - \frac{ 2 r ^ 2 }{ \lfloor n/2 \rfloor (5 C_1 + 2 C_2) ^ 2 } \right).
\end{align*}
This deviation inequality and Proposition \ref{prop:subgaudef} imply for $\lambda \in \R$ that
\begin{align*}
\log_{} \mathbb{E}_{k\sim \eta} 
    \left[ e ^ {\lambda \big( \E [f(X_{k})] - \E [f(X)] \big)} \right]\leq n (5 C_1 + 2 C_2) ^ 2 \lambda ^ 2.
\end{align*}
Plugging this inequality into \eqref{eq:alpha_n-bound}, we obtain for all $\lambda \in \R$ that
\begin{align*}
    \alpha_n (\lambda)
    & \leq \sup_{0 \leq k \leq \left\lfloor \frac{ n }{ 2 } \right\rfloor} \left( 2 k C_1 ^ 2 \lambda ^ 2
    + 4 \min \left\{ k, n - 2 k \right\} (2 C_1 + C_2) ^ 2 \lambda ^ 2 \right) 
    + n (5 C_1 + 2 C_2) ^ 2 \lambda ^ 2 \\
    & \leq \frac{ K n }{ 4 } (C_1 + C_2) \lambda ^ 2,
\end{align*}
where $K$ is an absolute constant. Then we apply Proposition \ref{prop:subgaudef} to obtain for all $r > 0$ that
\begin{align*}
    \mathbb{P} \left( \left| f (X) - \E [ f (X) ] \right| \ge r \right)
    \leq 2 \exp \left( - \frac{ r ^ 2 }{ K n (C_1 + C_2) ^ 2 } \right).
\end{align*}
This completes the proof.
\end{proof}

\subsection{Proof of Theorem \ref{thm:ConcentrationCorePartition}}

In this subsection, we provide the proof of Theorem \ref{thm:ConcentrationCorePartition}. 

\begin{proof}[Proof of Theorem \ref{thm:ConcentrationCorePartition}] 

Consider $f: [d_n] ^ {n - 1} \to \R$, such that $f (x_1, \cdots, x_{n - 1}) = x_1 + \cdots + x_{n - 1}$. Let $C_1 = d_n$ and $C_2 = 0$. Then $f$ satisfies Conditions \ref{cond:BoundDiff} and \ref{cond:QuasiAP} in Theorem \ref{thm:concengeneral}. So for $r > 0$,
\begin{align*}
    \mathbb{P} \left( \left| L_{n, 2} - \mathbb{E} \left[ L_{n, 2} \right] \right| \ge r  \right)
    \leq 2 \exp \left( - C \frac{ r ^ 2 }{ n d_n ^ 2 } \right).
\end{align*}Similarly, we can prove \eqref{eq:concentilln}. Lemma \ref{lem:SelfConjugateDistribution} and Proposition \ref{prop:boundiff} immediately give \eqref{eq:concentilln2}.

Consider $f: [d_n] ^ {n - 1} \to \R$, such that
\begin{align*}
    f (x_1, \cdots, x_{n - 1}) 
    = \sum_{i = 1}^{n - 1} 
    \left( \frac{n - 1}{2} x_i^2 + \left( i - \frac{n - 1}{2} \right) x_i  \right) 
    - \sum_{1 \leq i < j \leq n - 1}^{} x_i x_j.
\end{align*}For any $x_1, \cdots, x_i, x_i', \cdots, x_{n - 1} \in [d_n]$,
\begin{align*}
    &\ \ \ \left| f (x_1, \cdots, x_{n - 1}) - f (x_1, \cdots, x_i, x_i', \cdots, x_{n - 1}) \right|  \\
    & = \left| \frac{n - 1}{2} (x_i ^ 2 - {x_i'} ^ 2) 
    + \left( i - \frac{n - 1}{2} \right) (x_i - x_i')
    - (x_i - x_i') \sum_{j \neq i}^{} x_j \right| \\
    & \leq \frac{ n - 1 }{ 2 } d_n ^ 2 + \frac{ n - 1 }{ 2 } d_n + (n - 2) d_n ^ 2 \\
    & \leq 2 n d_n ^ 2.
\end{align*}Let $C_1 = 2 n d_n ^ 2$ and $C_2 = 0$. Then $f$ satisfies Conditions \ref{cond:BoundDiff} and \ref{cond:QuasiAP} in Theorem \ref{thm:concengeneral}. For $r > 0$,
\begin{align*}
    \mathbb{P} \left( \left| S_{n, 2} - \mathbb{E} \left[ S_{n, 2} \right] \right| \ge r \right)
    \leq 2 \exp \left( - C \frac{ r ^ 2 }{ n ^ 3 d_n ^ 4 } \right).
\end{align*}Similarly we can prove \eqref{eq:concentilsn}. Lemma \ref{lem:SelfConjugateDistribution} and Proposition \ref{prop:boundiff} directly imply \eqref{eq:concentilsn2}.
\end{proof}

\section*{Acknowledgments}
This work was supported by the National Science Foundation of China grants 12201155 and 62201175. 

\section*{Author Contributions}

The second author developed the main ideas and techniques to prove the main results in this paper and wrote the primary manuscript. The first author assisted in reviewing the manuscript and modified Step 5 in the proof of Theorem \ref{thm:normalgeneral}. The third author introduced the problem to the second author, helped review the manuscript, and provided the numerical results in the appendix.

\bibliographystyle{plain}
\bibliography{Stein2024}

\appendix

\section{Numerical Simulation}

Here we provide some numerical results obtained by \sl{Python} \cite{van1995python}, which support our asymptotic normality results.
For positive integers $k$ and $d$,  we define the $k$-th standardized moment of  $L_{n, 1}, ~L_{n, 2}, ~S_{n, 1}, ~S_{n, 2}$ with $D_n=dn$ as  $m_{k,d}(L_{n, 1}), ~m_{k,d}(L_{n, 2}), ~m_{k,d}(S_{n, 1}), ~m_{k,d}(S_{n, 2})$, respectively.  Then we have the following numerical results in Tables \ref{tab:1},\ref{tab:2},\ref{tab:3} and \ref{tab:4}.  Notice that the $k$-th moments ($k\geq 3$) of the standard normal distribution are $0,3,0,15,0,105,\cdots$.  The sequences $m_{k,d}(L_{n, 1}), ~m_{k,d}(L_{n, 2}), ~m_{k,d}(S_{n, 1}), ~m_{k,d}(S_{n, 2})$ all seem tend to this sequence when $n\to\infty$.
\begin{table}[htpb!]
\caption{The $k$-th standardized moment $m_{k,3}(L_{n, 1})$ for $3\leq k \leq 8$.}
\begin{center}
\resizebox{0.999\linewidth}{!}{
\begin{tabular}{|c|c|c|c|c|c|c|c|c|c|c|}
\hline
\diagbox{k}{n}
&$5$&$6$&$7$&$8$&$9$&$10$&$11$&$12$&$13$&$14$
\\
\hline
$3$&$0.000$&$0.000$&$0.000$&$0.000$&$0.000$&$0.000$&$0.000$&$0.000$&$0.000$&$0.000$
\\
\hline
$4$&$2.660$&$2.728$&$2.773$&$2.806$&$2.830$&$2.849$&$2.864$&$2.876$&$2.887$&$2.895$
\\
\hline
$5$&$0.000$&$0.000$&$0.000$&$0.000$&$0.000$&$0.000$&$0.000$&$0.000$&$0.000$&$0.000$
\\
\hline
$6$&$10.420$&$11.253$&$11.831$&$12.256$&$12.580$&$12.836$&$13.043$&$13.214$&$13.358$&$13.480$
\\
\hline
$7$&$0.000$&$0.000$&$0.000$&$0.000$&$0.000$&$0.000$&$0.000$&$0.000$&$0.000$&$0.000$
\\
\hline
$8$&$50.458$&$58.893$&$65.152$&$ 69.949$&$73.733$&$76.789$&$79.305$&$81.413$&$83.203$&$84.741$
\\
\hline
\end{tabular}\label{tab:1}
}
\end{center}
\end{table}

\begin{table}[htpb!]
\caption{The $k$-th standardized moment $m_{k,3}(S_{n, 1})$ for $3\leq k \leq 8$.}
\begin{center}
\resizebox{0.999\linewidth}{!}{
\begin{tabular}{|c|c|c|c|c|c|c|c|c|c|c|}
\hline
\diagbox{k}{n}
&$5$&$6$&$7$&$8$&$9$&$10$&$11$&$12$&$13$&$14$
\\
\hline
$3$&$-0.087$&$-0.042$&$-0.003$&$ 0.028$&$0.051$&$0.067$&$0.079$&$0.088$&$ 0.094$&$0.098$
\\
\hline
$4$&$2.639$&$2.791$&$2.865$&$2.901$&$2.919$&$2.929$&$2.934$&$2.938$&$2.941$&$2.943$
\\
\hline
$5$&$0.114$&$0.010$&$0.186$&$0.362$&$0.516$&$0.641$&$0.739$&$0.815$&$0.872$&$0.914$
\\
\hline
$6$&$9.952$&$11.497$&$12.450$&$13.124$&$13.515$&$13.764$&$13.928$&$14.042$&$14.123$&$14.185$
\\
\hline
$7$&$1.152$&$1.956$&$3.124$&$4.329$&$5.444$&$6.414$&$7.226$&$7.886$&$8.413$&$8.826$
\\
\hline
$8$&$45.130$&$59.368$&$69.948$&$ 77.485$&$82.822$&$86.621$&$89.358$&$ 91.368$&$92.878$&$94.041$
\\
\hline
\end{tabular}\label{tab:2}
}
\end{center}
\end{table}

\begin{table}[htpb!]
\caption{The $k$-th standardized moment $m_{k,2}(L_{n, 2})$ for $3\leq k \leq 8$.}
\begin{center}
\resizebox{0.999\linewidth}{!}{
\begin{tabular}{|c|c|c|c|c|c|c|c|c|c|c|}
\hline
\diagbox{k}{n}
&$8$&$9$&$10$&$11$&$12$&$13$&$14$&$15$&$16$&$17$
\\
\hline
$3$&$0.065$&$-0.003$&$0.033$&$0.007$&$0.019$&$0.010$&$0.013$&$0.009$&$0.010$&$0.008$
\\
\hline
$4$&$2.821$&$2.716$&$2.832$&$2.793$&$2.845$&$2.836$&$2.860$&$2.862$&$2.875$&$2.880$
\\
\hline
$5$&$0.669$&$-0.118$&$0.368$&$0.024$&$0.210$&$0.071$&$0.134$&$0.080$&$0.097$&$0.075$
\\
\hline
$6$&$12.400$&$10.995$&$12.658$&$11.988$&$12.824$&$12.601$&$13.013$&$12.989$&$13.206$&$13.252$
\\
\hline
$7$&$6.101$&$-1.644$&$3.853$&$-0.241$&$2.303$&$0.416$&$1.455$&$0.643$&$1.017$&$0.673$
\\
\hline
$8$&$70.001$&$55.433$&$74.802$&$ 65.958$&$77.124$&$73.279$&$79.261$&$78.216$&$81.446$&$81.640$
\\
\hline
\end{tabular}\label{tab:3}
}
\end{center}
\end{table}

\begin{table}[htpb!]
\caption{The $k$-th standardized moment $m_{k,2}(S_{n, 2})$ for $3\leq k \leq 8$.}
\begin{center}
\resizebox{0.999\linewidth}{!}{
\begin{tabular}{|c|c|c|c|c|c|c|c|c|c|c|}
\hline
\diagbox{k}{n}
&$8$&$9$&$10$&$11$&$12$&$13$&$14$&$15$&$16$&$17$
\\
\hline
$3$&$-0.052$&$-0.060$&$ -0.070$&$-0.074$&$-0.079$&$-0.081$&$-0.083$&$-0.085$&$-0.085$&$ -0.086$
\\
\hline
$4$&$2.326$&$2.404$&$2.463$&$2.514$&$2.555$&$2.590$&$2.621$&$2.647$&$2.670$&$2.690$
\\
\hline
$5$&$-0.141$&$-0.229$&$-0.334$&$-0.400$&$ -0.463$&$ -0.508$&$ -0.546$&$-0.576$&$ -0.601$&$ -0.620$
\\
\hline
$6$&$7.21$&$7.851$&$8.394$&$8.884$&$ 9.311$&$9.692$&$10.028$&$10.328$&$10.596$&$10.837$
\\
\hline
$7$&$-0.210$&$-0.711$&$ -1.387$&$ -1.909$&$-2.433$&$-2.874$&$-3.277$&$-3.626$&$-3.935$&$-4.203$
\\
\hline
$8$&$26.277$&$30.363$&$34.193$&$ 37.841$&$41.260$&$44.455$&$ 47.426$&$50.181$&$52.733$&$55.097$
\\
\hline
\end{tabular}\label{tab:4}
}
\end{center}
\end{table}

We also provide some numerical results for $M_{n, 3}^{(k)}$.
For positive integers $k'$ and $d$,  we define the $k'$-th standardized moment of  $M_{n, 3}^{(k)}$ with $E_n=2dn$ as  $m_{k',d}(M_{n, 3}^{(k)})$.  Then we have the following numerical results in Tables \ref{tab:5},\ref{tab:6},\ref{tab:7} and \ref{tab:8}.  Notice that the $k'$-th moments ($k'\geq 3$) of the standard normal distribution are $0,3,0,15,0,105,\cdots$.  The sequences $m_{k',d}(M_{n, 3}^{(k)})$ also seem to tend to this sequence when $n\to\infty$.

\begin{table}[htpb!]
\caption{the $k'$-th standardized moment $m_{k',2}(M_{n, 3}^{(0)})$ for $3\leq k' \leq 8$.}
\begin{center}
\resizebox{0.999\linewidth}{!}{
\begin{tabular}{|c|c|c|c|c|c|c|c|c|c|c|}
\hline
\diagbox{$k'$}{n}
&$6$&$7$&$8$&$9$&$10$&$11$&$12$&$13$&$14$&$15$
\\
\hline
$3$&$-0.198$&$-0.198$&$-0.171$&$-0.171$&$ -0.153$&$-0.153$&$-0.140$&$-0.140$&$-0.129$&$-0.129$
\\
\hline
$4$&$2.615$&$2.615$&$2.711$&$2.711$&$2.769$&$2.769$&$2.807$&$2.807$&$2.835$&$2.835$
\\
\hline
$5$&$-1.623$&$-1.623$&$1.483$&$1.483$&$-1.369$&$-1.369$&$1.275$&$1.275$&$-1.197$&$-1.197$
\\
\hline
$6$&$10.196$&$10.196$&$11.290$&$11.290$&$11.981$&$11.981$&$12.456$&$12.456$&$12.802$&$12.802$
\\
\hline
$7$&$-11.998$&$-11.998$&$-12.074$&$-12.074$&$-11.768$&$-11.768$&$-11.357$&$-11.357$&$-10.933$&$-10.933$
\\
\hline
$8$&$50.891$&$50.891$&$61.705$&$ 61.705$&$69.008$&$69.008$&$74.235$&$74.235$&$78.149$&$78.149$
\\
\hline
\end{tabular}\label{tab:5}
}
\end{center}
\end{table}

\begin{table}[htpb!]
\caption{the $k'$-th standardized moment $m_{k',2}(M_{n, 3}^{(1)})$ for $3\leq k' \leq 8$.}
\begin{center}
\resizebox{0.999\linewidth}{!}{
\begin{tabular}{|c|c|c|c|c|c|c|c|c|c|c|}
\hline
\diagbox{$k'$}{n}
&$6$&$7$&$8$&$9$&$10$&$11$&$12$&$13$&$14$&$15$
\\
\hline
$3$&$0.353$&$0.377$&$ 0.307$&$0.323$&$ 0.276$&$0.287$&$0.252$&$0.260$&$0.233$&$0.240$
\\
\hline
$4$&$2.659$&$2.687$&$2.745$&$2.761$&$2.797$&$2.806$&$2.831$&$2.837$&$2.855$&$2.860$
\\
\hline
$5$&$2.746$&$2.957$&$2.560$&$2.702$&$2.387$&$2.491$&$2.238$&$2.318$&$2.111$&$2.176$
\\
\hline
$6$&$11.114$&$11.577$&$12.105$&$12.383$&$12.694$&$12.879$&$13.084$&$13.216$&$13.361$&$13.460$
\\
\hline
$7$&$19.358$&$21.334$&$20.196$&$21.571$&$20.057$&$21.080$&$19.579$&$20.378$&$18.991$&$19.636$
\\
\hline
$8$&$61.427$&$ 67.070$&$ 73.041$&$ 76.845$&$80.013$&$82.707$&$84.581$&$86.577$&$87.778$&$89.312$
\\
\hline
\end{tabular}\label{tab:6}
}
\end{center}
\end{table}

\begin{table}[htpb!]
\caption{the $k'$-th standardized moment $m_{k',2}(M_{n, 3}^{(2)})$ for $3\leq k' \leq 8$.}
\begin{center}
\resizebox{0.999\linewidth}{!}{
\begin{tabular}{|c|c|c|c|c|c|c|c|c|c|c|}
\hline
\diagbox{$k'$}{n}
&$6$&$7$&$8$&$9$&$10$&$11$&$12$&$13$&$14$&$15$
\\
\hline
$3$&$0.596$&$0.624$&$0.519$&$0.537$&$  0.465$&$0.478$&$0.425$&$ 0.435$&$ 0.394$&$0.402$
\\
\hline
$4$&$2.827$&$2.859$&$ 2.871$&$2.890$&$2.898$&$2.909$&$2.915$&$ 2.923$&$2.927$&$2.933$
\\
\hline
$5$&$4.546$&$4.815$&$4.258$&$4.435$&$3.983$&$4.110$&$3.742$&$3.839$&$3.535$&$3.612$
\\
\hline
$6$&$14.094$&$14.848$&$14.684$&$15.125$&$14.930$&$15.217$&$15.044$&$15.246$&$15.101$&$15.251$
\\
\hline
$7$&$33.447$&$36.560$&$34.658$&$36.693$&$34.295$&$35.747$&$33.401$&$34.500$&$32.347$&$33.215$
\\
\hline
$8$&$96.178$&$106.877$&$108.242$&$ 115.207$&$113.413$&$118.220$&$115.635$&$119.130$&$116.501$&$119.148$
\\
\hline
\end{tabular}\label{tab:7}
}
\end{center}
\end{table}

\begin{table}[htpb!]
\caption{the $k'$-th standardized moment $m_{k',2}(M_{n, 3}^{(3)})$ for $3\leq k' \leq 8$.}
\begin{center}
\resizebox{0.999\linewidth}{!}{
\begin{tabular}{|c|c|c|c|c|c|c|c|c|c|c|}
\hline
\diagbox{$k'$}{n}
&$6$&$7$&$8$&$9$&$10$&$11$&$12$&$13$&$14$&$15$
\\
\hline
$3$&$0.761$&$0.779$&$0.663$&$0.674$&$ 0.595$&$0.603$&$0.544$&$0.550$&$0.504$&$0.509$
\\
\hline
$4$&$ 2.956$&$2.980$&$2.969$&$2.982$&$2.976$&$2.984$&$ 2.980$&$2.986$&$2.983$&$2.987$
\\
\hline
$5$&$5.607$&$ 5.860$&$5.302$&$5.457$&$4.988$&$5.094$&$4.705$&$4.783$&$4.456$&$4.517$
\\
\hline
$6$&$16.213$&$17.044$&$16.635$&$17.093$&$16.679$&$16.964$&$16.612$&$16.805$&$16.513$&$16.653$
\\
\hline
$7$&$41.818$&$45.359$&$43.576$&$45.743$&$43.281$&$44.742$&$42.260$&$43.315$&$40.999$&$41.801$
\\
\hline
$8$&$119.635$&$132.780$&$133.606$&$ 141.844$&$138.630$&$144.102$&$139.856$&$143.697$&$139.432$&$142.255$
\\
\hline
\end{tabular}\label{tab:8}
}
\end{center}
\end{table}

\end{document}